\documentclass[12pt]{amsart}
\usepackage[a4paper,margin=2.5cm]{geometry}
\usepackage{amsmath,amssymb,amsfonts,amsthm,mathrsfs}
\usepackage{epic}
\usepackage{amsopn,amscd,graphicx}
\usepackage{color,transparent} 
\usepackage[usenames, dvipsnames]{xcolor}
\usepackage
[colorlinks, breaklinks,
bookmarks = false,
linkcolor = NavyBlue,
urlcolor = ForestGreen,
citecolor = ForestGreen,
hyperfootnotes = false
]
{hyperref}
\usepackage{comment}
\usepackage{palatino, mathpazo}
\usepackage{dsfont}
\usepackage{enumerate}

\numberwithin{equation}{section}
\usepackage[dvipsnames]{xcolor}
 1
 1
 1

\numberwithin{equation}{section}

\newtheorem{theorem}{Theorem}[section]
\newtheorem{corollary}[theorem]{Corollary}

\newtheorem{lemma}[theorem]{Lemma}

\theoremstyle{definition}
\newtheorem{definition}[theorem]{Definition}
\newtheorem{example}[theorem]{Example}
\newtheorem{remark}[theorem]{Remark}

\newtheoremstyle{named}{}{}{\itshape}{}
{\bfseries}{.}{.5em}{\thmnote{#3}#1}
\theoremstyle{named}

\DeclareMathOperator{\End}{End} 
\DeclareMathOperator{\Dom}{Dom}
\DeclareMathOperator{\Spec}{Spec} 
\DeclareMathOperator{\Ker}{Ker} 
\DeclareMathOperator{\supp}{supp}

\DeclareMathOperator{\Ph}{Ph}

\newcommand{\R}{\mathbb{R}}
\newcommand{\C}{\mathbb{C}}
\newcommand{\N}{\mathbb{N}}
\newcommand{\Z}{\mathbb{Z}}
\newcommand{\ol}{\overline}

\def\ker{{\rm Ker}~}
\def\Re{\mathrm{Re}}
\def\Im{\mathrm{Im}}
\def\cC{\mathscr{C}}

\def\cCc{\mathscr{C}^\infty_c}

\begin{document}
\title[Semi-classical spectral asymptotics of Toeplitz operators]
{Semi-classical spectral asymptotics of Toeplitz operators on CR manifolds}

\author[Hendrik Herrmann]{Hendrik Herrmann}
\address{Bergische Universit\"at Wuppertal, 
Fakult\"at 4 - Mathematik und Naturwissenschaften, 
Gau{\ss}stra{\ss}e 20, 42119 Wuppertal, Germany}
\thanks{}
\email{hherrmann@uni-wuppertal.de or post@hendrik-herrmann.de}

\author[Chin-Yu Hsiao]{Chin-Yu Hsiao}
\address{Institute of Mathematics, Academia Sinica, Astronomy-Mathematics Building, 
No. 1, Sec. 4, Roosevelt Road, Taipei 10617, Taiwan}
\thanks{Chin-Yu Hsiao was partially supported by the Taiwan Ministry 
of Science and Technology projects  108-2115-M-001-012-MY5, 
109-2923-M-001-010-MY4. George Marinescu and Wei-Chuan Shen are 
partially supported by the DFG funded projects
SFB/TRR 191 "Symplectic Structures in Geometry, Algebra and Dynamics"
(Project-ID 281071066-TRR 191), 
and the ANR-DFG project QuaSiDy (Project-ID 490843120). 
Hendrik Herrmann is partially supported by the 
ANR-DFG project QuaSiDy (Project-ID 490843120).
}
\email{chsiao@math.sinica.edu.tw or chinyu.hsiao@gmail.com}
	
\author[George Marinescu]{George Marinescu}
\address{Universit{\"a}t zu K{\"o}ln,  Mathematisches Institut,
Weyertal 86-90, 50931 K{\"o}ln, Germany
\newline\mbox{\quad}\,Institute of Mathematics `Simion Stoilow', 
Romanian Academy, Bucharest, Romania}
\thanks{}
\email{gmarines@math.uni-koeln.de}
	
\author[Wei-Chuan Shen]{Wei-Chuan Shen}
\address{Universit{\"a}t zu K{\"o}ln,  Mathematisches Institut,
Weyertal 86-90,   50931 K{\"o}ln, Germany}
\email{wshen@uni-koeln.de}
		
\date{\today}

\begin{abstract}
Let $X$ be a compact strictly pseudoconvex embeddable CR manifold and let
$T_P$ be the Toeplitz operator on $X$ associated with a first order
pseudodifferential operator $P$. We consider the operator $\chi_k(T_P)$ 
defined by functional
calculus of $T_P$, where $\chi$ is a smooth function
with compact support in the positive real line and 
$\chi_k(\lambda):=\chi(k^{-1}\lambda)$. We
show that $\chi_k(T_P)$ admits a full asymptotic expansion as 
$k\to+\infty$. As applications, we obtain several CR analogues of results concerning 
the high powers of line bundles in complex geometry. In particular, 
we establish a Kodaira type embedding theorem, 
Tian's convergence theorem and an 
embedding theorem of strictly pseudoconvex CR manifolds 
into perturbed spheres.
\end{abstract}
\maketitle
\tableofcontents	
\newpage
\section{Introduction}
The study of reproducing kernels 
in several complex variables plays an important role in
several fields of mathematics. 
The Szeg\H{o} kernel is is the reproducing kernel 
of the space of square integrable CR functions on a CR manifold.
Our starting point is a fundamental result by 
Boutet de Monvel--Sj\"ostrand \cite[(1.5) Th\'eor\`eme]{BouSj75}, 
showing that the Szeg\H{o} kernel of a compact strictly pseudoconvex
CR manifold of dimension greater or equal to five
is a complex Fourier integral operator (see \eqref{eq:PiFIO}). 
This result has many applications in complex geometry and 
geometric quantization. 
It can be used to produce many smooth 
CR functions and is related to a global CR embedding 
theorem by Boutet de Monvel~\cite{Bou75}, 
cf.~also \cite[\S 12]{ChSh01}, stating that
for a compact strictly pseudoconvex CR 
manifold $X$ of dimension greater or equal to five
there are global smooth CR functions $f_1,\ldots,f_N$ 
such that the map
\begin{equation}\label{e-gue221208yyd}
\Phi: X\rightarrow\mathbb C^N,\quad
x\in X\rightarrow(f_1(x),\ldots,f_N(x))\in\mathbb C^N,
\end{equation}
is an embedding. Boutet de Monvel's embedding theorem 
is of very general nature.
For the study of specific CR geometry problems, 
such as the stability of CR structures or the existence of spherical CR embeddings, 
we need to refine the embedding theorem, 
and find adapted embeddings \eqref{e-gue221208yyd}.
In K\"ahler geometry 
similar problems have been well studied using semi-classical 
Bergman kernel asymptotic 
expansions~\cite{Cat99,DLM06, HM14, MM07, Zel98} 
and the Kodaira embedding theorem. Therefore, we believe it is 
important to consider a semiclassical version of the 
Boutet de Monvel and Sj\"ostrand's result and a Kodaira-type embedding 
theorem for strictly pseudoconvex CR manifolds. 
In this paper, we achieve these semi-classical theorems 
by studying the weighted spectral projections 
of first-order elliptic self-adjoint Toeplitz operators involving packages of eigenvalues
indexed by a semi-classical parameter $k$ drifting to the right on 
the real axis when $k\to+\infty$, cf.~Theorem 
\ref{thm:ExpansionMain}. 
In the following we will formulate the most important results.
	
Let $(X, T^{1,0}X)$ be an orientable  
compact strictly pseudoconvex Cauchy--Riemann (CR) manifold
of hypersurface type and of dimension $2n+1$, $n\geq1$, 
where $T^{1,0}X$ denotes the CR structure of $X$. 
Let $J\in \End(HX)$ be the complex structure map on the 
associated Levi distribution 
$HX={\rm Re\,}(T^{1,0}X)\subset TX$. 
Let $\xi\in\cC^\infty(X,T^*X)$ be a contact form on $X$ 
such that the Levi form
$\mathcal{L}=\frac12d\xi(\cdot,J\cdot)$ 
is positive definite and $dV_{\xi}:=\frac{2^{-n}}{n!}\xi\wedge \left(d\xi\right)^n$ 
be the volume form induced by $\xi$. Let $\xi_x:=\xi(x)$, $x\in X$.
Let $dV$ be any volume form on $X$ and consider
$L^2(X)=L^2(X,dV)$ the associated space of 
of square integrable functions. 
	
Let $\overline{\partial}_b$ be the tangential Cauchy--Riemann
operator on $X$. We assume that the Kohn Laplacian 
$\Box_b^{(0)}=\overline{\partial}^*_b\overline{\partial}_b$ has closed range in 
$L^2$. We recall that in this case, according to Boutet de Monvel's proof 
\cite{Bou74,Bou75} and Kohn's argument \cite{Koh86}, 
for all compact strictly pseudoconvex CR manifolds,
the closed range condition is equivalent to the CR embeddability of 
$X$ into some $\mathbb{C}^N$. We also recall that the closed range condition holds 
for any compact strictly pseudoconvex CR manifold when 
$n\geq 2$, cf.~\cite{Bou75}, or when $n=1$ and $X$ 
admits a transversal CR $\R$-action, cf.~\cite{Lem92,MY07}. 
	
We denote the space of square integrable CR functions by 
\begin{equation}\label{eq:h0b}
H^0_b(X):=\{u\in L^2(X):\overline\partial_b u=0\},
\end{equation} 
where $\overline\partial_b u$ is defined in the sense of currents. 
It equals $\ker\Box_b^{(0)}$ and is a closed subspace of 
$L^2(X)$. The Szeg\H{o} projection is the 
orthogonal projection $\Pi:L^2(X)\to H^0_b(X)$,
and the associated Schwartz kernel $\Pi(x,y)\in\mathscr{D}'(X\times X)$
is called the Szeg\H{o} kernel. If $\Box_b^{(0)}$ has
closed range in $L^2$, then $\Pi$ maps $\cC^\infty(X)$
into $\cC^\infty(X)$, cf.\ \cite[Theorem 1.2]{Hs10} for example.

Let $P\in L^1_\mathrm{cl}(X)$ be a first-order formally 
self-adjoint classical pseudodifferential operator and  let
$T_P:=\Pi P\Pi:\cC^\infty(X)\to\cC^\infty(X)$ be the Toeplitz operator 
associated to $P$. If the principal symbol $\sigma_P$ of $P$ 
restricted to the symplectic cone
\begin{equation}\label{eq:symco}
\Sigma:=\{(x,t\xi_x):x\in X,~t>0\}\subset T^*X
\end{equation}
is everywhere positive, we say that $T_P$ is elliptic. 
Then $T_P$ has a self-adjoint $L^2$-extension, 
cf.\ Theorem \ref{thm:T_P is self-adjoint}, and
the spectrum $\Spec(T_P)\subset\R$ of $T_P$
consists only of isolated
eigenvalues, is bounded from below
and has only $+\infty$ as a point of accumulation.
Moreover, for every 
$\lambda\in\Spec(T_P)\setminus\{0\}$, the eigenspace 
$\Ker(T_P-\lambda I)$
is a finite dimensional subspace of $H^0_b(X)\cap\cC^\infty(X)$ 
(see Theorem \ref{thm:Spec(T_P)}).

The main analytic object of this paper is a spectral operator
and its Schwartz kernel.
We consider a function $\chi\in\cC^\infty_c(\R)$ with 
$\supp\chi\subset(0,+\infty)$ and define for $k>0$ the function
\begin{equation}\label{eq:chik}
\chi_k:\R\to\C\,,\quad \chi_k(t)=\chi\big(k^{-1}t\big).
\end{equation}
Then $\chi_k\in\cCc(\R)$, $\supp\chi_k=k\supp\chi\subset(0,+\infty)$.
We will study the spectral operator $\chi_k(T_P)$ constructed
by the functional calculus.  The kernel of $\chi_k(T_P)$
is given by 
\begin{equation}\label{eq:KerChiTp}
\chi_k(T_P)(x,y)=\sum_{j=1}^{+\infty}
\chi_k(\lambda_j)f_j(x)\overline{f_j}(y)\in\cC^\infty(X\times X),
\end{equation} 
where $\lambda_1\leq \lambda_2\leq\ldots\leq\lambda_j\leq\ldots\,$ 
is the sequence of non-zero eigenvalues of \(T_P\) arranged in 
increasing order and counted with multiplicity, 
and \(\{f_j\}_{j=1}^{+\infty}\) is an orthonormal system
for $H^0_b(X)\cap\cC^\infty(X)$
consisting of eigenfunctions of \(T_P\) corresponding 
to the sequence of eigenvalues 
$\{\lambda_j\}_{j=1}^{+\infty}$. 
The range of the operator $\chi_k(T_P)$ contains
the eigenspaces of $T_P$ corresponding to eigenvalues
in the set $k\supp\chi$, which drifts to the right and increases as
$k\to+\infty$.
We consider a function $\chi$ with support in $(0,+\infty)$
in order to avoid the zero eigenvalue of $T_P$,
which has infinite multiplicity and whose 
corresponding eigenfunctions are not necessarily CR.

Our first result describes the kernel 
of the spectral operator $\chi_k(T_P)$ 
as a semi-classical Fourier integral modulo $k$-negligible 
smooth kernels. To state it, we recall
Boutet de Monvel--Sj\"ostrand's fundamental theorem \cite{BouSj75} 
(see also \cite{Hs10,HM17JDG}) about the structure of the 
Szeg\H{o} kernel, cf.\ Theorem \ref{Boutet-Sjoestrand theorem}. 
For any coordinate patch $(D,x)$ on $X$,
there is a smooth function $\varphi:D\times D\to\C$ with
\begin{equation}
\label{Eq:PhaseFuncMainThm}
\begin{split}
&\operatorname{Im}\varphi(x,y)\geq 0,\\
&\varphi(x,y)=0~\text{if and only if}~y=x,\\
&d_x\varphi(x,x)=-d_y\varphi(x,x)=\xi(x),
\end{split}
\end{equation} 
such that we have on $D\times D$, the Szeg\H{o} projector
can be approximated by a Fourier integral operator
\begin{equation}\label{eq:PiFIO}
\Pi(x,y)=\int_0^{+\infty} 
e^{it\varphi(x,y)}s(x,y,t)dt+F(x,y),
\end{equation}
where $F(x,y)\in\cC^\infty(D\times D)$ and $s(x,y,t)
\in S^{n}_{\operatorname{cl}}(D\times D\times{\R}_+)$ 
is a classical H\"ormander symbol satisfying $s(x,y,t)
\sim\sum_{j=0}^{+\infty}s_j(x,y)t^{n-j}$ in 
$S^{n}_{1,0}(D\times D\times{\R}_+)$ and 
\begin{equation}
\label{eq:leading term s_0 intro}
s_0(x,x)=\frac{1}{2\pi^{n+1}}\frac{dV_\xi}{dV}(x).
\end{equation} 
\begin{theorem}
		\label{thm:ExpansionMain}
		Let $(X,T^{1,0}X)$ be an orientable 
		compact strictly pseudoconvex Cauchy--Riemann manifold
		of dimension $2n+1$, $n\geq1$,
		such that the Kohn Laplacian on $X$ has closed range in 
		$L^2(X)$. Let $\xi$ be be a contact form on $X$ such that the Levi form
		$\mathcal{L}=\frac12d\xi(\cdot,J\cdot)$ is positive definite.
		Let $(D,x)$ be any coordinate patch and let 
		$\varphi:D\times D\to\C$ be the phase function satisfying 
		\eqref{Eq:PhaseFuncMainThm} and \eqref{eq:PiFIO}.
		Then for any formally self-adjoint first order pseudodifferential operator 
		$P\in L^1_\mathrm{cl}(X)$ with
		$\sigma_P(\xi)>0$ on $X$,  
		and for any $\chi\in\cC^\infty_c((0,+\infty))$, $\chi\not\equiv 0$,
		the Schwartz kernel of $\chi_k(T_P)$, 
		$\chi_{k}(\lambda):=\chi\left(k^{-1}\lambda\right)$,
		can be represented for $k$ large by
		\begin{equation}
		\label{eq:asymptotic expansion of chi_k(T_P)}
		\chi_k(T_P)(x,y)=\int_0^{+\infty} 
		e^{ikt\varphi(x,y)}{A}(x,y,t,k)dt+O\left(k^{-\infty}\right)~\text{on}~D\times D,
		\end{equation}
		where ${A}(x,y,t,k)\in S^{n+1}_{\mathrm{loc}}
		(1;D\times D\times{\R}_+)$,
		\begin{equation}
		\label{Eq:LeadingTermMainThm}
		\begin{split}
		&{A}(x,y,t,k)\sim\sum_{j=0}^{+\infty} {A}_{j}(x,y,t)k^{n+1-j}~
		\mathrm{in}~S^{n+1}_{\mathrm{loc}}(1;D\times D\times{\R}_+),\\
		&A_j(x,y,t)\in\mathscr{C}^\infty(D\times D\times{\R}_+),~j=0,1,2,\ldots,\\
		&{A}_{0}(x,x,t)=\frac{1}{2\pi ^{n+1}}
		\frac{dV_{\xi}}{dV}(x)\,\chi(\sigma_P(\xi_x)t)\,t^n\not\equiv 0,
		\end{split}
		\end{equation}
		and for some compact interval $I\Subset\R_+$,
		\begin{equation}
		\begin{split}
		\supp_t A(x,y,t,k),~\supp_t A_j(x,y,t)\subset I,\ \ j=0,1,2,\ldots~.
		\end{split}
		\end{equation}
		Moreover, for any $\tau_1,\tau_2\in\cC^\infty(X)$ 
		such that $\supp(\tau_1)\cap\supp(\tau_2)=\emptyset$, 
		we have
		\begin{equation}
		\label{Eq:FarAwayDiagonalMainThm}
		\tau_1\chi_k(T_P)\tau_2=O\left(k^{-\infty}\right).
		\end{equation}
	\end{theorem}

The representation \eqref{eq:asymptotic expansion of chi_k(T_P)} 
of the kernel of $\chi_k(T_P)$
can be viewed as a semi-classical version of the theorem of 
Boutet de Monvel and Sj\"ostrand. It
shows that $\chi_k(T_P)$ is a semi-classical Fourier integral 
operator with complex phase and
with canonical relation generated by the phase $\varphi(x,y)t$. 
The integral in \eqref{eq:asymptotic expansion of chi_k(T_P)} 
is a smooth kernel, because $t$ runs in the bounded interval
$I$. The term $O(k^{-\infty})$ denotes $k$-negligible
smooth kernels (cf.\ Definition \ref{D:knegl}). We refer to 
\eqref{eq:s^mloc} for the definition of 
the space $S^{n+1}_{\mathrm{loc}}(1;D\times D\times{\R}_+)$.
		
It was shown in~\cite[Lemma 12.2]{BG81} 
that there is a classical elliptic pseudodifferential operator 
$Q$ on $X$ so that 
\begin{equation}\label{e-gue230619yyd}
Q\Pi\equiv\Pi Q,\ \ T_P\equiv\Pi Q\Pi. 
\end{equation}
From \eqref{e-gue230619yyd}, we can check that 
\begin{equation}\label{e-gue230619yydI}
\chi_k(T_P)=\Pi\circ\chi_k(Q)\circ\Pi+O(k^{-\infty}). 
\end{equation} 
Although it is possible to investigate how $\chi_k(T_P)$ behaves for 
large values of $k$ by 
using to \eqref{e-gue230619yyd} and \eqref{e-gue230619yydI}, 
we adopt an alternative method in this article. 
Namely, we develop a calculus for $z$-dependent Fourier integral 
operators with complex phase in order to obtain an expansion of 
$(z-T_P)^{-1}\Pi$, cf.~ Theorem~\ref{t-gue221218yyd}, 
and we use the Helffer--Sj\"ostrand's
formula (cf.~\cite[\S 2]{Dav95} or \cite[\S 8]{DiSj99}) 
to reduce the study of
$\chi_k(T_P)$ to the asymptotic expansion of 
\begin{equation}
\int_\C\frac{\partial\widetilde\chi_k}{\partial\overline{z}}(z-T_P)^{-1}
\Pi\,\frac{dz\wedge d\overline{z}}{2\pi i},
\end{equation}
as $k\to+\infty$, where $\widetilde\chi_k$ 
is an almost analytic extension of $\chi_k$.
We refer to Section \ref{sec:StrategyOfProof} 
for further details about the strategy of proving 
Theorem~\ref{thm:ExpansionMain}, 
and to Theorems \ref{thm:chi_k(T_P) by Psi} and 
\ref{thm:ExpMain Sec 4} for the proof. 
A spectral asymptotics result in 
the case of an arbitrary function $\chi\in\cCc(\R)$ will be given in 
Theorem \ref{thm:EGCO}.

Although it is possible to study the expansion of 
$(z-T_P)^{-1}\Pi$ by using the calculus of 
Hermite type Fourier integral operators in 
\cite[\S 2-11]{BG81}, it is unclear how to get the 
phase function \eqref{Eq:PhaseFuncMainThm} and 
the semi-classical symbol \eqref{Eq:LeadingTermMainThm}
to describe $\chi_k(T_P)$ as in \eqref{eq:asymptotic expansion 
of chi_k(T_P)} by using such kind of calculus.  
The calculus used in this paper not only 
gives us an expansion formula for the operator 
$(z-T_P)^{-1}\Pi$, but also yields 
the semi-classical Fourier integral operator 
\eqref{eq:asymptotic expansion of chi_k(T_P)} 
suitable to explore CR geometry.

We will later explore the application of this to CR embedding problems. 
Additionally, \eqref{eq:asymptotic expansion 
of chi_k(T_P)} could aid further research on geometric quantization 
of CR manifolds, building on the works of \cite{HsHua21,HsMaMar23}.
	
A direct consequence of Theorem \ref{thm:ExpansionMain} is the following. 
\begin{corollary}
\label{C:asyk}
In the situation of Theorem \ref{thm:ExpansionMain}, 
the kernel of the operator $\chi_k(T_P)$ has the asymptotic expansion on the diagonal 
\begin{equation}\label{eq:1.8}
\chi_k(T_P)(x,x)\sim\sum_{j=0}^{+\infty}
\mathcal{A}_j(x)k^{n+1-j}~\text{in}~S^{n+1}_{\rm loc}(1;X),
\end{equation}
where $\mathcal{A}_j\in\mathscr{C}^\infty(X)$ and
		\begin{equation}\label{eq:1.9}
		\mathcal{A}_0(x)=\frac{1}{2\pi^{n+1}}
		\frac{dV_{\xi}(x)}{\sigma_P(\xi_x)^{n+1}dV(x)}
		\int_{\R_+}\chi(t)t^ndt.
		\end{equation}
		The trace of $\chi_k(T_P)$ has the asymptotics as $k\to+\infty$,
		\begin{equation}\label{eq:1.9a}
		\operatorname{Tr}\chi_k(T_P)=\int_X \chi_k(T_P)(x,x)dV=
		\frac{k^{n+1}}{2\pi^{n+1}}
		\int_X\frac{dV_{\xi}(x)}{\sigma_P(\xi_x)^{n+1}}
		\int_{\R_+}\chi(t)t^ndt+O(k^{n})\,.
		\end{equation}
	\end{corollary}
	
We can rewrite the integral from \eqref{eq:1.9a} in geometric terms.
Let us denote by $dV_\Sigma$ the symplectic volume form on $\Sigma\subset T^*X$.
Then from \eqref{eq:symco} we immediately have
\begin{equation}\label{eq:19b}
\int_X\frac{dV_{\xi}(x)}{\sigma_P(\xi_x)^{n+1}}
\int_{\R_+}\chi(t)t^ndt=\int_\Sigma\chi(\sigma_P|_\Sigma)dV_\Sigma\,,
\end{equation}
thus formula \eqref{eq:1.9a} becomes 
\begin{equation}\label{eq:1.9c}
\operatorname{Tr}\chi_k(T_P)=
\frac{k^{n+1}}{2\pi^{n+1}}
\int_\Sigma\chi(\sigma_P|_\Sigma)dV_\Sigma+O(k^{n})\,,\quad k\to+\infty\,.
\end{equation}
We will illustrate Theorem~\ref{thm:ExpansionMain}
by three examples. Firstly, we will directly exhibit the asymptotic 
expansion of $\chi_k(T_P)(x,x)$ on the circle bundle as $k\to+\infty$
in Theorem~\ref{thm:Trace on circle bundle}. 
In Theorem~\ref{thm:main thm is sharp} we will also 
calculate the spectral asymptotics of $\chi_k(T_P)$
on the diagonal when $\chi\in\cCc(\R)$,  $\chi\geq 0$
and $\chi(0)\neq 0$, where we will see the remainder 
in Theorem~\ref{thm:EGCO} is sharp.
Finally, we will calculate in Theorem \ref{thm:Spectrum} 
the spectrum of a Toeplitz operator 
on a Grauert tube over a torus explicitly.

In the sequel we present several applications of Theorem \ref{thm:ExpansionMain}.
The first one concerns a Szeg\H{o} type limit theorem for scaled spectral measures.
Szeg\H{o} limit theorems give the weak
asymptotics of the spectral measure of a family of operators with 
respect to the parameter of the family, see \cite{Gui79}. 
Szeg\H{o} first obtained such a result for 
truncated infinite Toeplitz matrices, as the size of the truncation tends to infinity.
Let \(X\) and \(P\) be as in Theorem~\ref{thm:ExpansionMain}. 
Let $(\lambda_j)_{j\in\N}$
be the  sequence of non-zero eigenvalues
of \(T_P:\Dom T_P\to H^0_b(X)\) counted with multiplicity.
The counting spectral measure of $T_P$ is
$\mu=\sum_{j=1}^{+\infty}\delta\left(t-\lambda_j\right)$,
where \(\delta\) denotes the Dirac measure at zero. 
By the Weyl law for $T_P$ (cf.\ \cite[Theorem 13.1]{BG81}), 
the spectrum counting function $N(k)=\#\{j\in\N:\lambda_j\leq k\}$
has the asymptotics 
\footnote{The normalizing constant is different from \cite{BG81} 
since we have a different normalization of the volume form $dV_\xi$.}
\begin{equation}\label{eq:Nk}
N(k)=\frac{\operatorname{vol}(\Sigma_1)}{2\pi^{n+1}}k^{n+1}
+O(k^n),\quad k\to+\infty,
\end{equation}
where $\operatorname{vol}(\Sigma_1)$ is the symplectic volume of the
subset $\Sigma_1\subset\Sigma$ where $\sigma_P\leq 1$.
Note that $N(k)=\langle\mu,\mathds{1}_{(-\infty,k]}\rangle$.
In order to obtain a refined picture of the distribution of the eigenvalues of $T_P$,
we consider $\chi\in\cC^\infty_c(\R_+)$ and we test the measure $\mu$ 
with functions $\chi_k$ (cf.\ \eqref{eq:chik}), 
whose supports drift to the right on the real axis
as $k\to+\infty$. We have 
$\langle\mu,\chi_k\rangle=\sum_{j=1}^{+\infty}\chi(k^{-1}\lambda_j)=
\langle\sum_{j=1}^{+\infty}\delta\left(t-k^{-1}\lambda_j\right),\chi\rangle$.
We are thus led to consider the scaled spectral measures \(\mu_k\) given by 
\begin{equation}\label{eq:scaledspectraclmeasuredefinition}
\mu_k=k^{-n-1}\sum_{j=1}^{+\infty}\delta\left(t-k^{-1}\lambda_j\right).
\end{equation} 
\begin{theorem}\label{thm:ScaledSpectralMeasures}
In the situation of Theorem~\ref{thm:ExpansionMain}, 
the scaled spectral measures \(\mu_k\) 
converges weakly as \(k\to+\infty\) to the continuous measure 
\(\mu_\infty\) on $(0,+\infty)$ defined by 
\begin{equation}
\mu_\infty=\mathcal{C}_P\,t^ndt\,,\quad \text{with
$\mathcal{C}_P:=\frac{1}{2\pi^{n+1}}
\int_X\frac{dV_\xi}{\sigma_P(\xi)^{n+1}}\,,$}
\end{equation}
where $dt$ is the Lebesgue measure on $\R$. 
\end{theorem} 
	
Note that by \eqref{eq:19b} we have
\begin{equation}
\mu_\infty=\frac{1}{2\pi^{n+1}}(\sigma_P)_*(dV_\Sigma),
\end{equation}
where $(\sigma_P)_*(dV_\Sigma)$ is the push-forward 
by the symbol map
$\sigma_P:\Sigma\to(0,+\infty)$ of the measure defined by 
the symplectic volume form $dV_\Sigma$ on $\Sigma$. 
Related results for the scaled spectral measures of Toeplitz operators 
can be found in \cite[\S 13]{BG81}.
	
	Next, we present a 
	Kodaira type CR embedding theorem, 
	which is a refinement of Boutet de Monvel's CR embedding theorem.
	Applying Theorem \ref{thm:ExpansionMain} and taking coordinates as
	in \cite[Part I, Chapter 8]{Hs10} to compute the tangential Hessian of 
	$\varphi$ (cf.~Theorem \ref{thm:tangential hessian of varphi}), 
	we can use ingredients from \cite{Hs15,Hs18,HLM21} 
	to get the following CR embedding for $k$ large.
	\begin{theorem}
		\label{thm:embedding}
		Consider a CR manifold $(X,T^{1,0}X)$ and a Toeplitz operator 
		$T_P$ as in Theorem~\ref{thm:ExpansionMain}. 
		We denote by \(0<\lambda_1\leq\lambda_2\leq\ldots\) the positive eigenvalues 
		counting multiplicities of $T_P$ and let \(f_1,f_2\ldots\) 
		be an orthonormal set in \(H^0_b(X)\) such that $T_Pf_j=\lambda_jf_j$ for all $j\in\N$.
		Let \(0<\delta_1<\delta_2\) and \(\chi\in\cC^\infty_0((\delta_1,\delta_2))\).
		Let \(N_k=\#\{j\in\N:0<\lambda_j\leq k\delta_2\}\). Then the map
		\(G_k\colon X\to \C^{N_k}\),
		\begin{equation}
		G_k(x)=\big(\chi(k^{-1}\lambda_1)f_1,\ldots,
		\chi(k^{-1}\lambda_{N_k})f_{N_k}\big).
		\end{equation}
		is a CR embedding for all sufficiently large \(k\).
	\end{theorem}
	In the situation of Theorem \ref{thm:embedding}, 
	for \(k>0\), we consider the CR map
	\begin{equation}\label{eq:embeddingMap}
	F_k\colon X\to \C^{N_k},\quad
	F_k(x)=\sqrt{\frac{2\pi^{n+1}}{k^{n+1}}}G_k(x).
	\end{equation}
	Theorem \ref{thm:embedding} implies that \(F_k\) in 
	\eqref{eq:embeddingMap} is also a CR embedding for \(k\) sufficiently large. 
	Let us consider  the following structures on \(\C^N\).
	For \(\beta=(\beta_1,\ldots,\beta_N)\in(\R_+)^N\), \(N\geq 2\), 
	a holomorphic \(\R\)-action on \(\C^N\) is given by $t\circ (z_1,\ldots,z_N)=(e^{i\beta_1t}z_1,\ldots,e^{i\beta_Nt}z_N)$.
	The infinitesimal action is then given by the vector field
	\begin{equation}\label{eq:DefinitionTbetaIntroduction}
	\mathcal{T}_\beta=\sum_{j=1}^Ni\beta_j\left(z_j\frac{\partial}{\partial z_j}-\overline{z}_j\frac{\partial}{\partial \overline{z}_j}\right),\,\,\,\, (z_1,\ldots,z_N)\in\C^N,
	\end{equation}
	and we consider the smooth $1$-form \(\omega_{\beta}\) on \(\C^N\setminus\{0\}\) defined by
	\begin{equation}\label{eq:DefinitionOmegabetaIntroduction}
	\omega_{\beta}=\frac{1}{2i}\frac{\sum_{j=1}^N
		(\overline{z}_jdz_j-z_jd\overline{z}_j)}{\sum_{j=1}^N\beta_j|z_j|^2}.
	\end{equation}
	It follows that \(\omega_{\beta}(\mathcal{T}_\beta)=1\) 
	for all \(z\neq 0\). Let \(S^{2N-1}\subset\C^N\) 
	be a shpere of positive radius centered in zero and denote 
	the inclusion map by \(\iota\colon S^{2N-1}\to\C^N\). 
	Then, \(S^{2N-1}\) is a compact, orientable, 
	strictly pseudoconvex codimension one CR manifold where 
	the CR structure is induced by the complex structure of \(\C^N\). 
	We find that \(\mathcal{T}_\beta\) defines a transversal vector field 
	on \(S^{2N-1}\) which is a CR vector field meaning that the flow 
	of \(\mathcal{T}_\beta\) preserves the CR structure. 
	The corresponding real $1$-form on \(S^{2N-1}\) annihilating 
	the CR structure is given by \(\iota^*\omega_{\beta}\). 
	
	We will now examine to which extend the map \(F_k\) 
	preserves \(\mathcal{T}_\beta\) and 
	\(\omega_{\beta}\) for \(\beta=\lambda(k):=(\lambda_1,\ldots,\lambda_{N_k})\).

	\begin{theorem}
		\label{thm:vectorfield}
		Consider the situation in Theorem \ref{thm:ExpansionMain} and 
		$F_k$ as in \eqref{eq:embeddingMap}. 
		Choose the volume form and the pseudodifferential operator to be \(dV:=dV_{\xi}\) and $P:=-i\mathcal{T}$ respectively,  
		where \(\mathcal{T}\) is the Reeb vector field determined by 
		\(\iota_\mathcal{T}d\xi=0\) and \(\iota_{\mathcal{T}}\xi=1\). 
		Then we have that in 
		$\mathscr{C}^\infty$-topology as $k\to+\infty$,
		\begin{equation}
		\begin{split}
		&|k^{-1}dF_k\mathcal{T}|^2=
		C'_{\chi}+O(k^{-1}),\\
		&|k^{-1}\mathcal{T}_{\lambda(k)}\circ F_k|^2=
		C'_{\chi}+O(k^{-1}),\\
		&|k^{-1}\mathcal{T}_{\lambda(k)}\circ F_k-
		k^{-1}dF_k\mathcal{T}|^2=O(k^{-1}), 
		\end{split}
		\end{equation}
		where \(C'_\chi=\int_0^{+\infty} t^{n+2}|\chi|^2(t)dt\).
	\end{theorem}

		If \(\mathcal{T}\) 
		is in addition a CR vector field we have that \(F_k\) is equivariant 
		with respect to the flow of \(\mathcal{T}\) on \(X\) and \(\mathcal{T}_{\lambda(k)}\) 
		on \(\C^{N_k}\) respectively. 
		Hence we find  \((F_k)_*\mathcal{T}=\mathcal{T}_{\lambda(k)}\) in that case. 
		In general, the equality \((F_k)_*\mathcal{T}=\mathcal{T}_{\lambda(k)}\) 
		may fail. However one still has an asymptotic equality 
		as described in Theorem~\ref{thm:vectorfield}.

	\begin{theorem}
		\label{thm:pullbackweightoneform}
		Under the hypotheses of Theorem \ref{thm:ExpansionMain} 
		and with $F_k$ as in~\eqref{eq:embeddingMap} 
		we have that \(F_k^*\omega_{\lambda(k)}\) 
		is well-defined for all sufficiently large \(k>0\) and we have 
		\begin{equation}\label{eq:pullbackweightoneform}
		F_k^*\omega_{\lambda(k)}=\sigma_P(\xi)^{-1}\xi+O(k^{-1})\,,\quad
		k\to+\infty,
		\end{equation}
		in the $\mathscr{C}^\infty$-topology on $X$.
	\end{theorem}
	\begin{corollary}\label{cor:TianCR2FormV2}
		In the situation of Theorem~\ref{thm:pullbackweightoneform} 
		with \(P=-i\mathcal{T}\) as in Theorem~\ref{thm:vectorfield} 
		we have 
		\begin{equation}\label{eq:pullbackweighttwoform}
		F_k^*d\omega_{\lambda(k)}=d\xi+O(k^{-1})\,,\quad
		k\to+\infty,
		\end{equation}
		in the $\mathscr{C}^\infty$-topology on $X$.
	\end{corollary} 
		We note that in the situation of 
		Theorem~\ref{thm:pullbackweightoneform}, 
		\(d\omega_{\lambda_k}\) induces a Levi form on any sphere 
		\(S^{2N_k-1}\subset \C^{N_k}\) of positive radius centered in zero. 
		Furthermore, the Levi form on \(X\) with respect to \(\xi\) 
		is induced by \(d\xi\) (see Definition~\ref{D:Leviform}). 
		Hence,  Corollary~\ref{cor:TianCR2FormV2} can be seen as a 
		CR analog of a theorem of Tian~\cite{Ti90} who proved that the 
		pullback of the Fubini--Study metric on \(\C\mathbb{P}^N\) 
		under the Kodaira map for large tensor powers of a positive 
		holomorphic line bundle over a complex manifold approximates 
		the line bundle curvature. 
		We also point out that 
		\eqref{eq:pullbackweightoneform} and
		\eqref{eq:pullbackweighttwoform} 
		are invariant under scaling of the map \(F_k\) 
		by constant factors and that the leading terms are independent 
		of the choice of $\chi$. Moreover, Theorem~\ref{thm:pullbackweightoneform} 
		holds for any choice of pseudodifferential operator \(P\) 
		satisfying the assumptions in Theorem~\ref{thm:ExpansionMain}.

Theorem~\ref{thm:pullbackweightoneform} indicates
that the CR and contact structures on \(X\) are closely related 
to those of spheres in \(\C^{N_k}\) via the map \(F_k\) when \(k\) 
becomes large. We will show now that under the right choices of the 
operator \(T_P\) we can achieve that \(F_k(X)\) is actually lying 
in a sphere up to a small perturbation when \(k\) becomes large.
We note that the study of the embeddability of 
of strictly pseudoconvex CR manifolds in spheres is a 
important topic in CR geometry and complex analysis.
Fornaess~\cite{For76} obtained 
the embeddability in boundaries of strictly convex domains 
in complex Euclidean space 
for a large class of strictly pseudoconvex CR manifolds,.  
It was shown by Faran~\cite{Faran88} (see also Forstneric~\cite{Forst86}) 
that there exist real analytic  strictly pseudoconvex hypersurfaces 
in the complex Euclidean space which are not CR-embedabble into 
finite dimensional spheres. However, Lempert~\cite{Lempert90} 
(see also~\cite{Lempert82}) proved that every compact, 
real analytic, strictly pseudoconvex hypersurface in a complex 
Euclidean space can be CR embedded into an infinite dimensional sphere.
Forstneric~\cite{Forst86} and L{\o}w~\cite{Low85} 
independently proved that any bounded strictly pseudoconvex 
domain with a sufficiently smooth boundary can be properly and 
holomorphically embedded into some finite dimensional unit ball
and in addition proved that such an embedding can be chosen 
to be continuous up to the boundary.
Theorem~\ref{thm:embeddingperturbedsphere} provides an embedding 
result for strictly pseudoconvex CR manifolds into arbitrarily small 
perturbations of spheres, preserving the Reeb vector field asymptotically.
	
	Let \(S^{2N-1}\subset \C^N\) be the sphere of radius 
	one centered in zero. Given a smooth function 
	\(f\colon S^{2N-1}\to (-1,\infty)\) we define the perturbed sphere 
	\(S^{2N-1}(f)\) by
	\begin{equation}
	S^{2N-1}(f):=\left\{z\in\C^N\setminus\{0\}\colon |z|=
	1+f\Big(\frac{z}{|z|}\Big)\right\}.
	\end{equation}
	The projection \(z\mapsto z/|z|\) defines a 
	diffeomorphism between \(S^{2N-1}(f)\) and \(S^{2N-1}\). 
	Furthermore, \(S^{2N-1}(f)\) carries the induced CR structure 
	as a real hypersurface in \(\C^N\setminus\{0\}\). Given \(\beta\in\R_+^N\), 
	we define a $1$-form \(\alpha_{f,\beta}\) on \(\C^N\setminus\{0\}\) by
	\begin{equation}\label{eq:DefinitionAlphafbetaIntroduction}
	\alpha_{f,\beta}=-\frac{i(\partial-\overline{\partial})\rho(z)}
	{\sum_{j=1}^N\beta_j|z_j|^2},
	\end{equation}
	where \(\rho\) is the defining function for \(S^{2N-1}(f)\) 
	given by $\rho(z)=|z|-f\left(z/|z|\right)-1$.
	It turns out that the pullback of \(\alpha_{f,\beta}\) to 
	\(S^{2N-1}(f)\) via the inclusion map is a non-vanishing real 
	$1$-form annihilating the CR structure \(T^{1,0}S^{2N-1}(f)\) 
	of \(S^{2N-1}(f)\).  We will show now that any compact strictly 
	pseudoconvex CR manifold as in Theorem \ref{thm:ExpansionMain} 
	can be CR embedded in a perturbed sphere where the perturbation 
	can be chosen arbitrary small. This phenomenon can be seen as an 
	almost spherical embedding.
	\begin{theorem}
		\label{thm:embeddingperturbedsphere}
		Consider the situation in Theorem~\ref{thm:ExpansionMain} 
		and denote by \(\mathcal{T}\) the Reeb vector field with 
		respect to \(\xi\) determined by \(\iota_{\mathcal{T}}d\xi=0\) 
		and \(\iota_{\mathcal{T}}\xi=1\). 
		Fix any Hermitian metric and a family of \(\mathscr{C}^m\)-norms, 
		\(m\geq 0\), on \(X\). For any \(\varepsilon>0\) and any \(m\in\N\) 
		there exists \(N\in\N\), a perturbed sphere \(S^{2N-1}(f)\subset \C^N\) 
		where \(f\colon S^{2N-1}\to (-1,\infty)\) is smooth, \(\beta\in(\R_+)^N\)  
		and a CR embedding \(F\colon X\to\C^N\) with the following properties:
		\begin{itemize}
\item[(i)] \(F(X)\subset S^{2N-1}(f)\).
\item[(ii)] \(\sup_{S^{2N-1}}|f|\leq \varepsilon\).
\item[(iii)] There exists an open neighborhood \(U\) of 
\(\{\widetilde{F}(x)\colon x\in X\}\) in 
\(S^{2N-1}\), where \(\widetilde{F}(x):=F(x)/|F(x)|\) for \(x\in X\), 
such that  \(\sup_U|df|\leq \varepsilon \).
\item[(iv)] \(\|f\circ \widetilde{F}\|_{\mathscr{C}^m(X)}\leq \varepsilon\), 
\item[(v)] \(|F^*\alpha_{f,\beta}-\xi|\leq \varepsilon\).
\item[(vi)] \(F_*\mathcal{T}\) is transversal to 
\(T^{1,0}S^{2N-1}(f)\oplus T^{0,1}S^{2N-1}(f) \) at any point of \(F(X)\). 
\item[(vii)] We have
\[\Big\|1-\frac{|F_*\mathcal{T}|}{|\mathcal{T}_\beta\circ F|}
\Big\|_{\mathscr{C}^m(X)}\leq \varepsilon,\quad 
\Bigl\|1-\frac{\langle F_*\mathcal{T},
\mathcal{T}_\beta\circ F\rangle}{|F_*\mathcal{T}|
|\mathcal{T}_\beta\circ F|}\Bigr\|_{\mathscr{C}^m(X)}\leq 
\varepsilon,\quad 
\|F^*\omega_\beta-\xi\|_{\mathscr{C}^m(X)}\leq \varepsilon.
\]
\end{itemize}
\end{theorem}
Note that the choice of \(\varepsilon>0\) may have an effect on the 
dimension \(N\). 
As a direct consequence of Theorem~\ref{thm:embeddingperturbedsphere},
any strictly pseudoconvex, orientable, compact CR manifold \(X\)
of dimension greater or equal to five can be embedded in a 
perturbed sphere. 

It was shown in \cite{OV07}
that any compact Sasakian manifold admits a CR-embedding into a
Sasakian manifold which is diffeomorphic to a sphere.
If in addition the first 
de Rham cohomology group of \(X\) vanishes, it was
proved in \cite{VCoe11} that $X$ can be CR-embedded into a sphere with 
a transversally deformed Sasakian structure. 
See also the related result \cite{LP22} in Sasakian geometry. 
We note that Theorem~\ref{thm:embeddingperturbedsphere} 
can be also stated as a result on CR embeddings into spheres 
with non-standard CR structures. 
We refer to Theorem~\ref{cor:SphericalEmbedding} for a precise statement. 

\section{Preliminaries}
\subsection{Notations}
We use the following notations throughout this article 
$\mathbb{Z}$ is the set of integers, 
$\N=\{1,2,3,\ldots\}$ is the set of natural numbers and we 
put $\N_0=\N\bigcup\{0\}$; $\mathbb R$ is the set of
real numbers. Also, $\R_+:=\{x\in\R:x>0\}$ and 
$\overline{\mathbb R}_+=\R_+\cup\{0\}$.
For a multi-index $\alpha=(\alpha_1,\ldots,\alpha_n)\in\N^n_0$ 
and $x=(x_1,\ldots,x_n)\in\mathbb R^n$, we set
\begin{equation}
\begin{split}
&x^\alpha=x_1^{\alpha_1}\ldots x^{\alpha_n}_n,\\
& \partial_{x_j}=\frac{\partial}{\partial x_j}\,,\quad
\partial^\alpha_x=\partial^{\alpha_1}_{x_1}\ldots\partial^{\alpha_n}_{x_n}
=\frac{\partial^{|\alpha|}}{\partial x^\alpha}\cdot
\end{split}
\end{equation}
Let $z=(z_1,\ldots,z_n)$, $z_j=x_{2j-1}+ix_{2j}$, $j=1,\ldots,n$, 
be coordinates on $\C^n$. We write
\begin{equation}
\begin{split}
&z^\alpha=z_1^{\alpha_1}\ldots z^{\alpha_n}_n\,,
\quad\ol z^\alpha=\ol z_1^{\alpha_1}\ldots\ol z^{\alpha_n}_n\,,\\
&\partial_{z_j}=\frac{\partial}{\partial z_j}=
\frac{1}{2}\Big(\frac{\partial}{\partial x_{2j-1}}-
i\frac{\partial}{\partial x_{2j}}\Big)\,,
\quad\partial_{\ol z_j}=\frac{\partial}{\partial\ol z_j}
=\frac{1}{2}\Big(\frac{\partial}{\partial x_{2j-1}}+
i\frac{\partial}{\partial x_{2j}}\Big),\\
&\partial^\alpha_z=\partial^{\alpha_1}_{z_1}\ldots\partial^{\alpha_n}_{z_n}
=\frac{\partial^{|\alpha|}}{\partial z^\alpha}\,,\quad
\partial^\alpha_{\ol z}=\partial^{\alpha_1}_{\ol z_1}
\ldots\partial^{\alpha_n}_{\ol z_n}
=\frac{\partial^{|\alpha|}}{\partial\ol z^\alpha}\,.
\end{split}
\end{equation}
For $j, s\in\mathbb Z$, set $\delta_{js}=1$ if $j=s$, 
$\delta_{js}=0$ if $j\neq s$.
	
Let $U\subset\R^{n_1}$ and $V\subset\R^{n_2}$ be open sets. 
Let $\cCc(V)$ and $\mathscr{C}^\infty(U)$ be the space of smooth functions with
compact support in $V$ and the space of smooth functions on $U$,
respectively; $\mathscr{D}'(U)$ and $\mathscr{E}'(V)$ be the space
of distributions on $U$ and the space of distributions
with compact support in $V$, respectively. For all $s\in\mathbb{R}$,
we consider the Sobolev space $H^s(\R^{n_1})$
of order $s$ on $\mathbb R^{n_1}$, and we define 
$H^s_{\operatorname{loc}}(U):=\{u\in\mathscr{D}'(U):\rho 
u\in H^s(\R^{n_1})~\text{for all}~\rho\in\cCc(U)\}$
and $H^s_{\operatorname{comp}}(U):=
H^s_{\operatorname{loc}}(U)\cap\mathscr{E}'(U)$. 
Let $M$ be a smooth compact manifold. 
By using a partition of unity, we define the Soboloev space 
$H^s(M)$, $s\in\mathbb R$, of order $s$ on $M$ in the standard way. 
We write $\|\cdot\|_{H^s}$ to denote a Sobolev norm on $H^s(M)$. 
	
Let $F:\cC^\infty_c(V)\to\mathscr{D}'(U)$ be a continuous operator
and let $F(x,y)\in\mathscr{D}'(U\times V)$ be the Schwartz kernel of $F$, 
cf.~\cite[\S 5.2]{Hoe03}. In this work, we will identify $F$ with $F(x,y)$.
	We say that
	$F$ is a smoothing operator 
	if $F(x,y)\in\cC^\infty(U\times V)$. 
	Note that the following conditions are equivalent.
	\begin{equation}
	\begin{split}
	&F(x,y)\in\mathscr{C}^\infty(U\times V).\\\
	&F:\mathscr{E}'(V)\to \mathscr{C}^\infty(U)~\text{is continuous}.\\
	&F:H^{-s}_{\mathrm{comp}}(V)\to H^s_{\mathrm{loc}}(U)~\text{is continuous for all}~s\in\mathbb{N}_0.
	\end{split}
	\end{equation}
	For two continuous linear operators $A,B:\cC^\infty_c(V)\to\mathscr{D}'(U)$, we write $A\equiv B$ (on $U\times V$) or $A(x,y)\equiv B(x,y)$ (on $U\times V$) if $A-B$ is a smoothing operator, where $A(x,y),B(x,y)\in\mathscr{D}'(U\times V)$ are the distribution kernels of $A$ and $B$, respectively.
	
	Let $g\in\cC(\mathbb C\setminus\{0\})$ be a positive continuous function and let   $G:\cCc(V)\to\mathscr{D}'(U)$ be an operator possibly depending on some parameter $z\in\C$. We will use the notation $G=O_{H^s(V)\to H^\ell(U)}(g)$, where $s, \ell\in\mathbb R$, if for any $\tau\in\cC^\infty_c(V)$, $\widehat\tau\in\cC^\infty_c(U)$, $\widehat\tau G\tau:H^s(\mathbb R^{n_2})\to H^\ell(\mathbb R^{n_1})$ is bounded and there is some constant $C_{s,\ell}>0$ independent of $z$ such that the operator norm of $\widehat\tau G\tau:H^s(\mathbb R^{n_2})\to H^\ell(\mathbb R^{n_1})$ is bounded by
	$C_{s,\ell}g$, and we simply put $G=O_{H^s\to H^\ell}(g)$ if $U=V$.
	
	For a smooth function $f\in\mathscr{C}^\infty(U\times U)$, 
	we write $f(x,y)=O\left(|x-y|^{+\infty}\right)$ if for all multi-indices  
	$\alpha,\beta\in\N_0^{n_1}$, 
	$\partial_x^\alpha\partial_y^\beta f(x,x)=0$, for all $x\in U$. 
	
	Let us introduce some notion in microlocal analysis used in this paper. Let $D\subset\R^{2n+1}$ be an open set. 
	
	\begin{definition}
		For any $m\in\mathbb R$, $S^m_{1,0}(D\times D\times\mathbb{R}_+)$ 
		is the space of all $s(x,y,t)\in\cC^\infty(D\times D\times\mathbb{R}_+)$ 
		such that for all compact sets $K\Subset D\times D$, all $\alpha, 
		\beta\in\N^{2n+1}_0$ and $\gamma\in\N_0$, 
		there is a constant $C_{K,\alpha,\beta,\gamma}>0$ satisfying the estimate
		\begin{equation}
		\left|\partial^\alpha_x\partial^\beta_y\partial^\gamma_t a(x,y,t)\right|\leq 
		C_{K,\alpha,\beta,\gamma}(1+|t|)^{m-|\gamma|},\ \ 
		\mbox{for all $(x,y,t)\in K\times\mathbb R_+$, $|t|\geq1$}.
		\end{equation}
		We put $S^{-\infty}(D\times D\times\mathbb{R}_+)
		:=\bigcap_{m\in\mathbb R}S^m_{1,0}(D\times D\times\mathbb{R}_+)$.
	\end{definition}
	Let $s_j\in S^{m_j}_{1,0}(D\times D\times\mathbb{R}_+)$, 
	$j=0,1,2,\ldots$ with $m_j\rightarrow-\infty$ as $j\rightarrow+\infty$. 
	By the argument of Borel construction, there always exists $s\in S^{m_0}_{1,0}(D\times D\times\mathbb{R}_+)$ 
	unique modulo $S^{-\infty}$ such that 
	\begin{equation}
	s-\sum^{\ell-1}_{j=0}s_j\in S^{m_\ell}_{1,0}(D\times D\times\mathbb{R}_+)
	\end{equation}
	for all $\ell=1,2,\ldots$. If $s$ and $s_j$ have the properties above, we write
	\begin{equation}
	s(x,y,t)\sim\sum^{+\infty}_{j=0}s_j(x,y,t)~\text{in}~ 
	S^{m_0}_{1,0}(D\times D\times\mathbb{R}_+).
	\end{equation}
	Also, we use the notation 
	\begin{equation}  
	s(x, y, t)\in S^{m}_{{\rm cl\,}}(D\times D\times\mathbb{R}_+)
	\end{equation}
	if $s(x, y, t)\in S^{m}_{1,0}(D\times D\times\mathbb{R}_+)$ and we can find $s_j(x, y)\in\cC^\infty(D\times D)$, $j\in\N_0$, such that 
	\begin{equation}
	s(x, y, t)\sim\sum^{+\infty}_{j=0}s_j(x, y)t^{m-j}\text{ in }S^{m}_{1, 0}
	(D\times D\times\mathbb{R}_+).
	\end{equation}
For smooth paracompact manifolds $M_1, M_2$,
we define the symbol spaces 
$$S^m_{{1,0}}(M_1\times M_2\times\mathbb R_+),
\quad S^m_{{\rm cl\,}}(M_1\times M_2\times\mathbb R_+)$$ 
and asymptotic sums thereof in the standard way \cite{Hoe07,GriSj94}.
		
For $m\in\R$, let $L^m_{{\rm cl\,}}(D)$ denote the 
space of classical pseudodifferential operators of order $m$ on $D$. 
For $P\in L^m_{{\rm cl\,}}(D)$, we denote by 
$\sigma_P$ denote the principal symbol of $P$. 
We use similar notations in the case of manifolds.

	We recall some notations of semi-classical analysis.
	Let $U$ be an open set in $\mathbb{R}^{n_1}$ and let $V$ 
	be an open set in $\mathbb{R}^{n_2}$. 
	
	\begin{definition}\label{D:knegl}
		We say a $k$-dependent continuous linear operator
		$F_k:\cCc(V)\to\mathscr{D}'(U)$ is $k$-negligible 
		if for all $k$ large enough, $F_k$ is a smoothing operator,
		and for any compact set $K$ in $U\times V$, for all multi-index
		$\alpha\in\N_0^{n_1}$, $\beta\in\N_0^{n_2}$ and $N\in\mathbb{N}_0$, 
		there exists a constant $C_{K,\alpha,\beta,N}>0$ such that
		\begin{equation} 
		\left|\partial^\alpha_x\partial^\beta_y F_k(x,y)\right|\leq 
		C_{K,\alpha,\beta,N} k^{-N},
		\end{equation}
		for all $x,y\in K$ and $k$ large enough.
		We write $F_k(x,y)=O(k^{-\infty})$ (on $U\times V$) or $F_k=O(k^{-\infty})$ (on $U\times V$) if $F_k$ is $k$-negligible. 
	\end{definition}

	Next, we recall semi-classical symbol spaces. Let $W$ be an open set in $\mathbb{R}^{N}$, we define the space
	\begin{equation}
	S(1;W):=\{a\in\mathscr{C}^\infty(W):\sup_{x\in W}|\partial^\alpha_x a(x)|<\infty~\text{for all}~\alpha\in\mathbb{N}_0^N\}.
	\end{equation}
	{Consider the space $S^0_{\operatorname{loc}}(1;W)$ containing all smooth functions $a(x,k)$ with
		real parameter $k$ such that for all multi-index $\alpha\in\mathbb{N}_0^N$, any
		cut-off function $\chi\in\cCc(W)$, we have
		\begin{equation}
		\sup_{\substack{k\in\mathbb{R}\\k\geq 1}}\sup_{x\in W}|\partial^\alpha_x(\chi(x)a(x,k))|<+\infty.
		\end{equation}
		For general $m\in\mathbb{R}$, we can also consider
		\begin{equation}\label{eq:s^mloc}
		S^m_{\operatorname{loc}}(1;W):=\{a(x,k):k^{-m}a(x,k)\in\ S^0_{\operatorname{loc}}(1;W)\}.  
		\end{equation}
		In other words, $S^m_{\operatorname{loc}}(1;W)$ takes all the smooth function $a(x,k)$ with parameter $k\in\mathbb{R}$ satisfying the following estimate. For any compact set $K\Subset W$, any multi-index $\alpha\in\N^n_0$, there is a constant $C_{K,\alpha}>0$ independent of $k$ such that 
		\begin{equation}
		|\partial^\alpha_x (a(x,k))|\leq C_{K,\alpha} k^{m},~\text{for all}~x\in K,~k\geq1. 
		\end{equation}
		For a sequence of $a_j\in S^{m_j}_{\operatorname{loc}}(1;W)$ where 
		$m_j\searrow-\infty$, and $a\in S^{m_0}_{\operatorname{loc}}(1;W)$
		we say that
		\begin{equation}\label{eq_assum}
		a(x,k)\sim\sum_{j=0}^\infty a_j(x,k)\:\:
		\text{in}~S^{m_0}_{\operatorname{loc}}(1;W),
		\end{equation}
		if for all $l\in\mathbb{N}$, we have
		\begin{equation}
		a-\sum_{j=0}^{\ell-1} a_j\in S^{m_\ell}_{\operatorname{loc}}(1;W).
		\end{equation}
		In fact, for any sequence $a_j$ as above, there exists an 
		asymptotic sum $a$ as in \eqref{eq_assum}, 
		which is unique up to 
		$S^{-\infty}_{\operatorname{loc}}(1;W):=
		\cap_{m} S^m_{\operatorname{loc}}(1;W)$.}  
For detailed treatment of semi-classical analysis we refer to \cite{DiSj99}.
All the notations introduced above can be generalized to paracompact manifolds.
	
\subsection{Boutet de Monvel--Sj\"ostrand analysis of Szeg\H{o} kernel}
\label{sec:CRmanifoldsMicroLocal}
Let $X$ be a smooth orientable manifold of real dimension $2n+1,~n\geq 1$. 
We say $X$ is a (codimension one) Cauchy--Riemann (CR for short) 
manifold with CR structure \(T^{1,0}X\) if 
$T^{1,0}X\subset\mathbb{C}TX$ is a subbundle such that
\begin{enumerate}
\item[(i)] $\dim_{\mathbb{C}}T^{1,0}_{p}X=n$ for any $p\in X$.
\item[(ii)] $T^{1,0}_p X\cap T^{0,1}_p X=\{0\}$ for any $p\in X$, where 
$T^{0,1}_p X:=\overline{T^{1,0}_p X}$.
\item[(iii)] For $V_1, V_2\in \mathscr{C}^{\infty}(X,T^{1,0}X)$, 
we have $[V_1,V_2]\in\mathscr{C}^{\infty}(X,T^{1,0}X)$, where 
$[\cdot,\cdot]$ stands for the Lie bracket between vector fields. 
\end{enumerate}
Fix a smooth Hermitian metric $\langle\cdot|\cdot\rangle$ 
on the complexified tangent bundle 
$\C TX$ such that $T^{1,0}X$ is orthogonal to $T^{0,1}X$. 
Then locally there is a real vector field 
$\mathcal{T}$ with $|\mathcal{T}|^2:=\langle\mathcal{T}|\mathcal{T}\rangle=1$ 
which is pointwise orthogonal to $T^{1,0}X\oplus T^{0,1}X$. 
Notice that $\mathcal{T}$ 
is unique up to the choice of sign. 
Denote  $T^{*1,0}X$ and $T^{*0,1}X$ the dual bundles in 
$\C T^*X$ annihilating $\C\mathcal{T}\bigoplus T^{0,1}X$
and $\C\mathcal{T}\bigoplus T^{1,0}X$, respectively. 
Define the bundles of $(0,q)$ forms by $T^{*0,q}X:=\wedge^qT^{*0,1}X$. 
The Hermitian metric on $\C TX$ induces, by duality, a Hermitian metric 
on $\mathbb CT^*X$, 
and also on $T^{*0,q}X$ for all $q\in\{0,\ldots,n\}$. 
We shall also denote all these induced metrics by 
$\langle\cdot|\cdot\rangle$. For any open set $D\subset X$, 
denote $\Omega^{0,q}(D)$ 
be the space of smooth sections of $T^{*0,q}X$ over $D$ and let $\Omega^{0,q}_c(D)$ 
be the subspace of $\Omega^{0,q}(D)$ whose elements have compact support in $D$.

Let $\{\omega_j\}_{j=1}^n$ of $T^{*1,0}X$ be a local orthonormal frame. The real 
$(2n)$-form $$\omega:=i^n\omega_1\wedge
\overline{\omega}_1\wedge\ldots\wedge\omega_n\wedge
\overline{\omega}_n$$ is 
independent of the choice of orthonormal frame, so it is globally defined. 
Locally, there exists a real $1$-form 
$\xi$ with $|\xi|^2:=\langle\xi|\xi\rangle=1$ 
and orthogonal to $T^{*1,0}X\oplus T^{*0,1}X$. 
The $\xi$ is unique up to the choice of sign. 
Since $X$ is orientable, there is a nowhere vanishing $(2n+1)$-form 
$\Theta$ on $X$. Thus, $\xi$ can be specified uniquely by 
requiring that $\omega\wedge\xi=f\Theta$, 
where $f$ is a positive smooth function. 
	
\begin{definition}
The so chosen $1$-form $\xi$ is globally defined, 
and we call it the characteristic form on $X$.
\end{definition} 
	
The Levi distribution $HX$ of the CR manifold $X$ is the real part of 
$T^{1,0}X \oplus T^{0,1}X$,
i.e., the unique subbundle $HX$ of $TX$ such that 
\begin{align}\label{eq:2.5}
\C HX=T^{1,0}X \oplus T^{0,1}X.
\end{align} 
Let $J:HX\to HX$ be the complex structure given by 
$J(u+\ol u)=iu-i\ol u$, for $u\in T^{1,0}X$. 
If we extend $J$ complex linearly to $\C HX$ we have
$T^{1,0}X \, = \, \left\{ V \in \C HX \,:\, \, JV \, 
=  \,  iV  \right\}$.
Thus the CR structure $T^{1,0}X$ is determined by
the Levi distribution and $J$.
The annihilator $(HX)^{\perp}\subset T^*X$ is given by 
$(HX)^{\perp}=\R\xi$, so
we have
$\langle\,\xi(x),u\,\rangle=0$,
for any $u\in H_x X,\: x\in X$ 
(we also use the notation $\xi_x$ for $\xi(x)$).
The restriction of $d\xi$ on $HX$ is a $(1,1)$-form. 

\begin{definition}\label{D:Leviform}
The Levi form $\mathcal{L}_x=\mathcal{L}^{\xi}_x$ of $X$ at $x\in X$ 
associated to $\xi$ is the symmetric bilinear map
\begin{equation}\label{eq:2.12}
\mathcal{L}_x:H_x X\times H_x X\to\R,\quad \mathcal{L}_x(u,v)
=\frac12d\xi(u,Jv),\quad \text{ for } u, v\in H_x X.
\end{equation}  
It induces a Hermitian form
\begin{equation}\label{eq:2.12b}
\mathcal{L}_x:T^{1,0}_xX\times T^{1,0}_xX\to\C,
\:\: \mathcal{L}_x(U,V)=\frac{1}{2i}d\xi(U, \ol V) ,
\:\:U, V\in T^{1,0}_xX.
\end{equation}  
\end{definition}
\begin{definition}
A CR manifold $X$ is said to be  strictly pseudoconvex if 
there exists a characteristic $1$-form $\xi$
such that for every $x\in X$ the Levi form
$\mathcal{L}^{\xi}_x$ is positive definite. 
In this case, such $1$-form $\xi$ is also called a contact form 
because $\xi\wedge(d\xi)^n\neq 0$, where $\dim_\R X=2n+1$.
\end{definition} 
	
From now on, we will assume that $(X,T^{1,0}X)$ 
is a compact strictly pseudoconvex CR manifold of 
dimension $2n+1$ and we fix a characteristic $1$-form 
$\xi$ on $X$ as in Theorem \ref{thm:ExpansionMain}
so that $\mathcal{L}^{\xi}_x$ is positive definite at every $x\in X$.
	
\begin{definition}[{\cite[Lemma/Definition 1.1.9]{Gei08}}]\label{D:Reebfield}
The Reeb field $\mathcal{T}=\mathcal{T}(\xi)$
associated with the contact form $\xi$ is
the vector field which is uniquely defined by the equations
$\iota_{\mathcal{T}}\xi=1$, $\iota_{\mathcal{T}}d\xi=0$.
\end{definition} 

We take $\langle\cdot|\cdot\rangle$ so that 
$|\mathcal{T}|=1$, $|\xi|=1$ on $X$.
With respect to the given Hermitian metric 
$\langle\cdot|\cdot\rangle$, we consider  the orthogonal projection
\begin{equation}
\pi^{(0,q)}:\Lambda^q\mathbb{C}T^*X\to T^{*0,q}X.
\end{equation}
The tangential Cauchy--Riemann operator is defined to be 
\begin{equation}
\label{tangential Cauchy Riemann operator}
\overline{\partial}_b:=\pi^{(0,q+1)}\circ d:
\Omega^{0,q}(X)\to\Omega^{0,q+1}(X).
\end{equation}
By Cartan's formula, we can check that 
\begin{equation}
\overline{\partial}_b^2=0.  
\end{equation}
In this paper, we work with two volume forms on $X$:
\begin{itemize}
\item[({i})] A given smooth positive $(2n+1)$-form $dV(x)$.
\item[({ii})] The $(2n+1)$-form 
$dV_{\xi}:=\frac{2^{-n}}{n!}\xi\wedge\left(d\xi\right)^n$ 
given by the contact form $\xi$.
\end{itemize}

\begin{remark}
By the change of variable $t\mapsto \sigma_P(\xi)^{-1}t$, 
Theorem \ref{thm:ExpansionMain} can be rewritten such that the phase 
function becomes $\sigma_P(\xi)^{-1}\varphi$ 
and the leading term of the symbol becomes
\begin{equation}
\label{Eq:LeadingTermMainThmII}
\frac{1}{2\pi ^{n+1}}\frac{dV_{\xi}}
{\sigma_P(\xi)^{n+1}dV}\chi(t)t^n.
\end{equation}	
The $1$-form \(\sigma_P(\xi)^{-1}\xi\)
does not depend on the choice of the contact form 
\(\xi\). We call \(\alpha_P:=\sigma_P(\xi)^{-1}\xi\)
the characteristic form of \(X\) with respect to \(P\).
It is a smooth, non-vanishing real one form 
with \(\C\otimes\operatorname{Ker}\alpha_P=T^{1,0}X\oplus T^{0,1}X\). 
One finds that 
\begin{equation}
\frac{dV_{\xi}}{\sigma_P(\xi)^{n+1}}
=\frac{2^{-n}}{n!}\alpha_P\wedge (d\alpha_P)^n.
\end{equation}
This yields a representation of the leading term given by 
\eqref{Eq:LeadingTermMainThmII} independent of \(\xi\).
\end{remark} 	
	
Take the $L^2$-inner product $(\cdot|\cdot)$ on 
$\Omega^{0,q}(X)$ induced by $dV$ and $\langle\cdot|\cdot\rangle$ via
\begin{equation}
(f|g):=\int_X\langle f|g\rangle dV,~f,g\in\Omega^{0,q}(X).
\end{equation}
We denote by $L^2_{(0,q)}(X)$ the completion of 
$\Omega^{0,q}(X)$ with respect to $(\cdot|\cdot)$, 
and we write $L^2(X):=L^2_{(0,0)}(X)$. 
Let $\|\cdot\|$ denote the corresponding $L^2$-norm on $L^2_{(0,q)}(X)$
and let $\overline{\partial}_b^*$ be the formal adjoint of $\overline{\partial}_b$
with respect to $(\cdot|\cdot)$. The Kohn Laplacian $\Box^{(0)}_{b}$ is defined by 
\begin{equation}
\Box^{(0)}_{b}:=\overline{\partial}_b^*\overline{\partial}_b:
\mathscr{C}^\infty(X)\to\mathscr{C}^\infty(X).
\end{equation}
It is not an elliptic operator because its principal symbol 
$\sigma(\Box^{(0)}_{b})\in\mathscr{C}^\infty(T^*X)$ 
has a non-empty characteristic set
$\{(x,\eta)\in T^*X:\eta=\lambda\xi(x),~\lambda\in\mathbb{R}\setminus\{0\}\}$. 
The symplectic cone \eqref{eq:symco} is the
component of the characteristic set defined by $\lambda>0$.
In fact, $\Box^{(0)}_{b}$ may not even be hypoelliptic, i.e, 
$\Box^{(0)}_{b}u\in\mathscr{C}^\infty(X)$ might not imply that 
$u\in\mathscr{C}^\infty(X)$. We extend $\overline{\partial}_b$ 
to the $L^2$-space by
\begin{equation}
\label{e-gue201223yydu}
\begin{split}
&\overline{\partial}_b: \Dom\overline{\partial}_b
\subset L^2(X)\to L^2_{(0,1)}(X),\\
&\Dom\overline{\partial}_b:=\big\{u\in L^2(X):
\overline{\partial}_bu\in L^2_{(0,1)}(X)\big\},
\end{split}
\end{equation} 
where $\overline{\partial}_bu$ is calculated in the sense of
distributions. The kernel of this extension of $\overline{\partial}_b$
is the space $H^0_b(X)$ of square integrable CR functions \eqref{eq:h0b}.

Let $\overline{\partial}^*_{b,H}: \Dom\overline{\partial}^*_{b,H}
\subset L^2_{(0,1)}(X)\to L^2(X)$ be the $L^2$ adjoint of $\overline{\partial}_{b}$. 
Let $\Box^{(0)}_b: \Dom\Box^{(0)}_b\subset L^2(X)\to L^2(X)$, where
\begin{equation}
\Dom\Box^{(0)}_b=\Big\{u\in L^2(X): u\in\Dom\overline{\partial}_b, 
\overline{\partial}_bu\in\Dom\overline{\partial}^*_{b,H}\Big\},   
\end{equation}
and $\Box^{(0)}_bu=\overline{\partial}^*_{b,H}\,
\overline{\partial}_{b}u$, $u\in \Dom\Box^{(0)}_b$. 
	
We can check that $H_b^{0}(X)=\ker\Box^{(0)}_{b}=
\big\{u\in\Dom\Box^{(0)}_{b}:\Box^{(0)}_{b} u=0\big\}$. 
The results in \cite{Bou75, Koh85} yield the following theorem. 
\begin{theorem}\label{T:BK} 
Let $(X,T^{1,0}X)$ be a compact strictly pseudoconvex 
Cauchy--Riemann manifold
such that the Kohn Laplacian on $X$ has closed range in 
$L^2(X)$. Then $\dim H_b^{0}(X)=+\infty$ and $\Box^{(0)}_{b}$ 
is not hypoelliptic.
\end{theorem}
	\begin{definition}
		The Szeg\H{o} projection on $X$ is the orthogonal projection
		\begin{equation}
		\Pi: L^2(X)\to H^0_b(X).
		\end{equation}
		The Schwartz kernel $\Pi(x,y)\in\mathscr{D}'(X\times X)$ 
		of $\Pi$ is called the Szeg\H{o} kernel.
	\end{definition} 
The singularities of the Szeg\H{o} kernel are described 
by means of microlocal analysis in \cite{BouSj75}, 
see also \cite{Hs10,HM17JDG}.
We use here the formulation from \cite[Part I, Theorem 1.2]{Hs10}.
\begin{theorem}\label{Boutet-Sjoestrand theorem}  
Let $(X,T^{1,0}X)$ and $\xi$ be as in Theorem \ref{thm:ExpansionMain}. 
Then $\Pi(x,y)$ is smooth away from the diagonal of $X\times X$, 
and on any coordinate chart $(D,x)$ on $X$ it can be expressed 
by an oscillatory integral,
\begin{equation}\label{eq:Szego FIO}
\begin{split}
&\Pi(x,y)\equiv\Pi_\varphi(x,y)~\text{on}~X\times D,\\
&\Pi_\varphi(x,y)=\int_0^{+\infty} e^{it\varphi(x,y)}s^\varphi(x,y,t)dt,
\end{split}
\end{equation}
where
\begin{equation}\label{eq:varphi(x,y)}
\begin{split}
&\varphi(x,y)\in \mathscr{C}^\infty(D\times D),\\
&\mbox{${\rm Im\,}\varphi(x,y)\geq 0$},\\
&\mbox{$\varphi(x,y)=0$ if and only if $y=x$},\\
&\mbox{$d_x\varphi(x,x)=-d_y\varphi(x,x)=\xi(x)$},
\end{split}
\end{equation}
and
\begin{equation}
\label{eq:s(x,y,t)}
\begin{split}
&s^\varphi(x,y,t)\in S^n_{1,0}(D\times D\times\mathbb R_+),\\
&s^\varphi(x,y,t)\sim\sum^{+\infty}_{j=0}s^\varphi_j(x,y)t^{n-j}
\:\:\text{in}~S^n_{1,0}(D\times D\times\mathbb R_+),\\
&s^\varphi_j(x,y)\in\mathscr{C}^\infty(D\times D),~j=0,1,2,\ldots,\\
&s^\varphi(x,y,t),~s_j^\varphi(x,y)
\:\:\text{are properly supported on}~D\times D,~j=0,1,\ldots,\\
&s^\varphi_0(x,x)=\frac{1}{2\pi^{n+1}}\frac{dV_{\xi}}{dV}(x)\neq 0.
\end{split}
\end{equation}
\end{theorem}
\begin{remark}
Originally, \cite[Part I Theorem 1.2]{Hs10} is stated such that 
\eqref{eq:Szego FIO} holds on $D\times D$ with not necessarily 
properly supported $s^\varphi(x,y,t)$ and $s_j^\varphi(x,y)$, $j=0,1,\ldots$\,. 
However, it is always possible to adjust these functions to be properly 
supported thanks to $\Pi$ being
smoothing away from diagonal on $X$, by using 
a properly supported partition of unity 
\cite[p.~28-29]{GriSj94}. Thus, we can 
ensure that \eqref{eq:Szego FIO} holds on $X\times D$ not just on $D\times D$. 
This setup makes it easier to discuss the composition 
of Fourier integral operators later on.
\end{remark}
\begin{definition}\label{d-gue230125yyd}
Let $(X,T^{1,0}X)$ and $\xi$ be as in 
Theorem \ref{thm:ExpansionMain}. For any coordinate patch 
$(D,x)$ on $X$ and any $\Lambda\in\cC^\infty(X,\R_+)$, 
we let ${\Ph}(\Pi,\Lambda\xi,D)$ be the set all phase 
functions $\psi(x,y)\in\cC^\infty(D\times D)$ satisfying
\begin{equation}\label{eq:phase function psi}
\begin{split}
&\mbox{${\rm Im\,}\psi(x,y)\geq 0$},\\
&\mbox{$\psi(x,y)=0$ if and only if $y=x$},\\
&\mbox{$d_x\psi(x,x)=-d_y\psi(x,x)=\Lambda(x)\xi(x)$},
\end{split}
\end{equation}
such that there exists a symbol 
$s^\psi(x,y,t)\in S^n_{\rm cl}(D\times D\times\R_+)$ with 
\begin{equation}   \Pi(x,y)\equiv\int_0^{+\infty}e^{it\psi(x,y)}s^\psi(x,y,t)dt.
\end{equation}
\end{definition} 
	
For any $\psi\in {\Ph}(\Pi,\Lambda\xi,D)$, 
we denote from now on by $s^\psi(x,y,t)$ the full symbol in 
$S^n_{\rm cl}(D\times D\times{\R}_+)$ 
up to $S^{-\infty}(D\times D\times{\R}_+)$ such that $\Pi\equiv\Pi_\psi$,
\begin{equation}\label{eq:FIO Pi_psi}
\Pi_\psi(x,y):=\int_0^{+\infty} e^{it\psi(x,y)}s^\psi(x,y,t)dt,
\end{equation} 
and $s^\psi(x,y,t)$ and $s_j(x,y)$ are properly supported on 
$D\times D$, $j\in\N_0$. We have the following formula for 
$s_0^\psi(x,x)$.
\begin{theorem}
\label{thm:s_0^psi(x,x)}
For any $\psi\in {\Ph}(\Pi,\Lambda\xi,D)$, 
the leading term $s_0^\psi(x,y)$ of $\Pi_\psi$ satisfies
\begin{equation}\label{eq:s_0^psi(x,x)}
s^\psi_0(x,x)=\frac{1}{2\pi^{n+1}}\frac{dV_\xi}{dV}(x)\Lambda(x)^{n+1}.
\end{equation}
\end{theorem}
\begin{proof}
We first recall the classical formula \cite[(1.6)]{BouSj75} 
for $x\neq 0$, $\mathrm{Re}~x\geq 0$ and $m\in\mathbb{Z}$. 
We have in the sense of distributions,
\begin{equation}
\int_0^{+\infty} e^{-tx}t^mdt=
m!(x+i0)^{-m-1},\:\:\text{for $m\geq0$},
\end{equation}   
where the distribution $(x+i0)^{-m-1}$ is defined 
as in \cite[\S 3.2]{Hoe03}. Moreover, we have for the finite 
part ($\operatorname{F.P.}$) distribution, c.f.~\cite[\S 3.2]{Hoe03},
\begin{equation}
\operatorname{F.P.}\int_0^{+\infty} e^{-tx}t^mdt=
\frac{(-1)^m}{(-m-1)!}(x+i0)^{-m-1}
\left(\log (x+i0)+\gamma-\sum_{j=1}^{-m-1}\frac{1}{j}\right),
\:\:\text{for $m<0$},
\end{equation}
where $\gamma:=\lim_{m\to+\infty}
\left(\sum_{j=1}^m\frac{1}{j}-\log m\right)$ 
is the Euler constant. Combining these formulas with 
Theorem \ref{Boutet-Sjoestrand theorem}, we have
\begin{equation}\label{eq:plsSz}
\begin{split}
&\Pi_\varphi(x,y)=\frac{F_\varphi(x,y)}{(-i(\varphi(x,y)+i0))^{(n+1)}}
+G_\varphi(x,y)\log(-i(\varphi(x,y)+i0)),\\
&F_\varphi(x,y)=\sum_{j=0}^n (n-j)! s_j^\varphi(x,y)
(-i\varphi(x,y))^j+F(x,y)\varphi(x,y)^{n+1},\\
&G_\varphi(x,y)\equiv\sum_{j=0}^{+\infty}
\frac{(-1)^{j+1}}{j!}s^\varphi_{n+j}(x,y)(-i\varphi(x,y))^j,
\end{split}
\end{equation}
where $F(x,y),F_\varphi(x,y),G_\varphi(x,y)\in\cC^\infty(D\times D)$.
By \cite[Theorem 5.4]{HM17JDG}
we have a smooth function $f(x,y)$ with $f(x,x)\neq 0$ such that
\begin{equation}\label{eq:equivalence of Szego phase function}
\varphi(x,y)=f(x,y)\psi(x,y)+O(|x-y|^{+\infty}).
\end{equation}
For any point $p\in D$, we can take a local coordinates $x$ 
around $p$ such that $\mathcal T=\frac{\partial}{\partial x_{2n+1}}\cdot$ 
Identifying $p$ to $0\in\R^{2n+1}$, \eqref{eq:varphi(x,y)} implies that
\begin{equation}\label{eq:f((0,x_2n+1,0))}
f((0,x_{2n+1}),0)=\Lambda(0)^{-1}+O(|x_{2n+1}|).
\end{equation}
Combining $\Pi_\varphi\equiv\Pi_\psi$, 
\eqref{eq:plsSz}, 
\eqref{eq:equivalence of Szego phase function}, 
\eqref{eq:f((0,x_2n+1,0))} and 
$0\neq\psi((0,x_{2n+1}),0)=O(|x_{2n+1}|)$, we can check that 
\begin{equation}
s_0^\varphi((0,x_{2n+1}),0)=
\Lambda(0)^{-n-1}s_0^\psi((0,x_{2n+1}),0)+O(|x_{2n+1}|),
\end{equation}
and hence $s_0^\varphi(0,0)=\Lambda(0)^{-n-1}s_0^\psi(0,0)$ by continuity. 
The argument above works for all point in $D$, so 
the transformation rule \eqref{eq:s_0^psi(x,x)} of the leading term follows.
\end{proof}

We point out that we adopt the global 
description of the leading term $s_0(x,x)$ for \eqref{eq:s(x,y,t)}. 
In the local picture, near any point $p\in X$, let $\{W_j\}_{j=1}^n$ 
be an orthonormal frame of $T^{1,0}X$ in a neighborhood of $p$ 
such that $\mathcal{L}_p(W_j,\overline{W}_s)=\delta_{js}\mu_j$, 
for some $\mu_j>0$ and $j,s\in\{1,\ldots,n\}$. Define
\begin{equation}
\det\mathcal{L}_p:=\Pi_{j=1}^n\mu_j(p).
\end{equation} 
If we set $g_{jk}=\langle\frac{\partial}{\partial x_j}|
\frac{\partial}{\partial x_k}\rangle$ and 
$g=(g_{jk})_{j,k=1}^{2n+1}$ then
\begin{equation}
dV_\xi(x)=\det\mathcal{L}_x\sqrt{\det(g)}\,dx.
\end{equation}
\begin{remark}  
The result 
\cite[Theorem 1.2]{Hs10} is stated and proved for general 
$(0,q)$-forms under the assumption that the Levi form 
is globally non-degenerate with constant signature $(n_-,n_+)$. 
In this case, the condition $Y(q)$ of Kohn holds 
if and only if $q\notin \{n_-,n_+\}$. In particular, 
the strict pseudoconvexity implies that condition $Y(0)$ fails.
The characteristic form $\omega_0$ in \cite[Part I]{Hs10} 
corresponds to our $-\xi$ when $q=0$, in order to 
specify the case $q=n_-$ and $q=n_+$ naturally. 
If $\varphi_+$ and $\varphi_-$ are phase functions 
of the Fourier integral operators describing the 
Szeg\H{o} projections 
on $(0,n_+)$-forms and $(0,n_-)$-forms, respectively, 
then $\omega_0$ is chosen such that 
$d_x\varphi_\pm(x,x)=\pm\omega_0(x)$. 
The convention $\omega_0=-\xi$ applies to the series of papers
\cite{HM14,HM17CPDE,HM17JDG,Hs18,GaHs23,GaHs21}. 
We use here the convention that $\xi$ satisfies \eqref{eq:2.12}
in order to be consistent to the convention of 
\cite{BG81}. 
\end{remark} 

When $X$ is as in Theorem \ref{thm:ExpansionMain}, it is known
by \cite{BouSj75,Hs10,HM17JDG}, 
that $\Pi$ is a pseudodifferential operator of 
order $0$ of type $\big(\frac{1}{2},\frac{1}{2}\big)$. 
Thus, by the Calderon--Vaillancourt's theorem \cite[Theorem 18.6.6]{Hoe07}
we have
\begin{equation}
\label{Sobolev bounded of Pi}
\Pi=O_{H^s\to H^s}(1),\quad s\in\mathbb Z.
\end{equation}
We end this subsection by the following local description 
of the tangential Hessian of the phase function. 
Such local picture is essential for calculating the leading coefficient 
of the Szeg\H{o} kernel expansion \cite{BouSj75,Hs10,HM17JDG}. 
Also, it turns out that it is crucial in our proof of the 
CR embedding theorem (Theorem~\ref{thm:embedding}).
\begin{theorem}[{\cite[Part I, Theorem 1.4]{Hs10}}]
\label{thm:tangential hessian of varphi}
For any given point $p\in X$, let $\{W_j\}_{j=1}^n$ 
be an orthonormal frame of $T^{1,0}X$ in a neighborhood 
of $p$ such that $\mathcal{L}_p(W_j,\overline{W}_s)=\delta_{js}\mu_j$ 
for some $\mu_j>0$ and $j,s\in\{1,\ldots,n\}$. 
There exist local coordinates near $p$ such that
\begin{equation}
\label{eq_loccoord}
x(p)=0,\quad \xi(p)=dx_{2n+1},\quad
\mathcal{T}=\frac{\partial}{\partial x_{2n+1}},
\end{equation}
and 
\begin{equation}
\begin{split}
W_j=\frac{\partial}{\partial z_j}+i\mu_j\overline{z}_j
\frac{\partial}{\partial x_{2n+1}}-d_jx_{2n+1}
+\frac{\partial}{\partial x_{2n+1}}+
\sum_{\ell=1}^{2n}e_{j\ell}(x)\frac{\partial}{\partial x_\ell},
\end{split}
\end{equation} 
where $d_j\in\mathbb C$, $e_{j\ell}=O(|x|)$, 
$e_{j\ell}\in\cC^\infty$, $j\in\{1,\ldots,n\}$, 
$\ell\in\{1,\ldots,2n\}$. In these coordinates we have
for any $\varPhi\in{\Ph}(\Pi,\xi,D)$, 
\begin{equation}
\label{thm:tangential hessian of varphi in Theorem}
\begin{split}
\varPhi(x,y)&=x_{2n+1}-y_{2n+1}+
i\sum_{j=1}^n\mu_j|z_j-w_j|^2+
\sum_{j=1}^n i\mu_j(\overline{z}_jw_j-z_j\overline{w}_j)\\
&+\sum_{j=1}^n \left(d_j(z_jx_{2n+1}-
w_jy_{2n+1})+\overline{d}_j(\overline{z}_jx_{2n+1}
-\overline{w}_jy_{2n+1})\right)\\
&+(x_{2n+1}-y_{2n+1})f(x,y)+O\left(|(x,y)|^3\right),
\end{split}
\end{equation}
where $f\in\mathscr{C}^\infty(D\times D)$ satisfies 
$f(x,y)=\overline{f}(y,x)$ and $f(0,0)=0$.
\end{theorem}
According to \cite[Part I Proposition 7.16]{Hs10}, 
there exist $C>0$ such that 
in a small enough neighborhood we have in these coordinates,
\begin{equation}
\label{eq:Im varphi>C|z-w|^2}
\Im\varPhi(x,y)\geq C|z-w|^2.
\end{equation}
\begin{remark}\label{rmk:varphi(x,y)=x_2n+1+g(x',y)}
We can make the calculation in section \S 3 simpler by using 
the following specific phase function.
Let $\varphi$ be the one described in 
Theorem~\ref{Boutet-Sjoestrand theorem} and let $(D,x)$ be coordinates
as in Theorem \ref{thm:tangential hessian of varphi}. 
By using the Malgrange preparation theorem \cite[Theorem 7.5.5]{Hoe03}
and the Melin--Sj\"ostrand equivalence of phase functions 
\cite[Definition 4.1 \& Theoren 4.2]{MeSj75}, 
we can make the following choice for $\varphi$, 
\begin{equation}\label{e-gue221212yyd}
\varphi(x,y)=x_{2n+1}+\widehat\varphi(x',y), 
\end{equation}
where $\widehat\varphi\in\cC^\infty(D\times D)$, 
$x'=(x_1,\ldots,x_{2n})$. If $\varphi_1, \varphi_2\in {\Ph}(\Pi,\xi,D)$
satisfy 
\eqref{e-gue221212yyd}, 
we see from~\cite[Theorem 5.4]{HM17JDG} 
that $\varphi_1-\varphi_2=O(|x-y|^{+\infty})$. 
From now on, unless otherwise stated, we will 
work with the phase function $\varphi\in {\Ph}(\Pi,\xi,D)$ satisfying 
\eqref{e-gue221212yyd}. 
Note that $\varphi$ is uniquely determined up to an error of 
size $O(|x-y|^{+\infty})$. 
\end{remark}
\subsection{Strategy for the proof of Theorem 
\ref{thm:ExpansionMain}}\label{sec:StrategyOfProof}
	
Let us explain the strategy we use 
to prove Theorem \ref{thm:ExpansionMain}. 
We study $\chi_k(T_P)$ by examining
\begin{equation}
\chi_k(T_P)=\chi_k(T_P)\circ\Pi=
\int_\mathbb C\frac{\partial\widetilde{\chi}_k}
{\partial\overline{z}}(z-T_P)^{-1}\circ\Pi~
\frac{dz\wedge d\overline{z}}{2\pi i},
\end{equation}
where $\widetilde\chi_k$ is an almost analytic extension of $\chi_k$. 
The last equality above holds by the Helffer--Sj\"ostrand formula and 
is applied in \cite[Theorem 1.1]{GaHs21} to study functional calculus 
of zeroth-order elliptic self-adjoint Toeplitz operators. 

In our approach $T_P$ is treated 
as a pseudodifferential operator of type 
$(\frac{1}{2},\frac{1}{2})$ and as a classical Fourier integral 
operator with complex phase, and we cannot obtain the 
asymptotic expansion of $(z-T_P)^{-1}$ in the parameter $z$
directly from the calculus of pseudodifferential operators \cite{Hoe07}. 
Instead, we use the calculus and analysis of 
Melin--Sj\"ostrand \cite{MeSj75} and 
Boutet de Monvel--Sj\"ostrand \cite{BouSj75}
to study the asymptotic expansion of $(z-T_P)^{-1}\Pi$. 
This method has been thoroughly 
established for zeroth-order operators 
\cite{GaHs21}, however, further investigation is necessary 
for our first-order situation. 
Roughly speaking, in Theorem \ref{t-gue221218yyd}, 
for all $N\in\mathbb{N}_0$ we can find  
$\psi\in{\Ph}(\Pi,\sigma_P(\xi)^{-1}\xi,D)$
and $z$-dependent Fourier integral operators
$B_z^{(0)}$,\ldots, $B_{z}^{(N)}$ and $R_z^{(N+1)}$ 
with Schwartz kernel
\begin{equation}\label{eq:B_z^j in Section 2.3}
 B_z^{(j)}(x,y)=\int^{+\infty}_0 e^{it{\psi(x,y)}}
 \frac{\sum^{p_j}_{h=0}b_h^{(j)}(x,y,t)z^h}{(z-t)^{\ell_j}}
 t^{n-j-1}dt,~j=0\ldots,N,
\end{equation}
and
\begin{equation}
 R_z^{(N+1)}(x,y)=\int^{+\infty}_0 e^{it\psi(x,y)}
 \frac{\sum^{p_{N+1}}_{h=0}r_h^{(N+1)}(x,y,t)z^h}
 {(z-t)^{\ell_{N+1}}}t^{n-N-1}dt,
  \end{equation}
respectively, such that 
$\Pi B_z^{(j)}\equiv B_z^{(j)}\Pi
\equiv B_z^{(j)}\mod\mathscr{C}^\infty_z$, $j=0,\ldots,N$, and
\begin{equation}\label{eq:z-microlocal expansion 2.3}
\begin{split}
&(z-T_P)\sum_{j=0}^{N}B_z^{(j)}\equiv\Pi+
{R_z^{(N+1)}}\mod\mathscr{C}^\infty_z,
\end{split}
\end{equation}
where 
\[
\begin{split}
&p_j\leq\ell_j, j=0,\ldots,N+1, 
b_h^{(j)}(x,y,t)\in S^{\ell_j-h}_{\rm cl}(D\times D\times\mathbb R_+), 
j=0,\ldots,N,\\ 
&r_h^{(N+1)}(x,y,t)\in 
S^{\ell_{N+1}-h}_{\rm cl}(D\times D\times\mathbb R_+)
\end{split}
\]
and $\mathscr{C}^\infty_z$ is the class of $z$-dependent smoothing operators
introduced in Definition \ref{def:C infty_z}, and for $j=0$ in 
\eqref{eq:B_z^j in Section 2.3} we have $\ell_0=1$, 
$p_0=0$ and $b_0^{(0)}(x,y,t)\sim
\sum_{h=0}^\infty b^{(0)}_h(x,y)t^{1-h}$
in $S^{1}_{1,0}(D\times D\times\mathbb R_+)$ 
with $b^{(0)}_0(x,x)=s_0^\varphi(x,x)$ given by 
\eqref{eq:s(x,y,t)}. 
In Lemma \ref{l-gue230112yyd}, 
we will see that for each $j=0,\ldots,N$, 
$$\displaystyle{\int_\mathbb{C}
\frac{\partial\widetilde{\chi}_k}{\partial\overline{z}}}
B_z^{(j)}(x,y)\frac{d z\wedge d\overline{z}}{2\pi i}$$ 
can always be calculated by the Cauchy--Pompeiu 
formula and turns out to be a semi-classical Fourier integral 
operator with complex phase. However, for 
\eqref{eq:z-microlocal expansion 2.3} it is difficult to define 
a suitable $z$-dependent symbol space to control 
terms such as 
$$\int_\mathbb C \frac{\partial\widetilde\chi_k}{\partial\overline{z
}}(z-T_P)^{-1}{R_z^{(N+1)}}\frac{dz\wedge d\overline{z}}{2\pi i}$$ 
in the approximation of $\chi_k(T_P)$ for fixed $k$. 
To handle this, we need a semi-classical estimate in $k$. 
In \S 4, for $k$ is large enough, after using 
a suitable truncation and the fact that $T_P$
has discrete spectrum, we establish the 
semi-classical estimate for 
$$\int_\mathbb C \frac{\partial\widetilde\chi_k}
{\partial\overline{z}}(z-T_P)^{-1}{R_z^{(N+1)}}
\frac{dz\wedge d\overline{z}}{2\pi i}$$ 
in Theorem \ref{t-gue230114yyde}, and as a by-product we can also control 
$$\int_\mathbb C \frac{\partial\widetilde\chi_k}
{\partial\overline{z}}(z-T_P)^{-1}{F_z^{(N+1)}}
\frac{dz\wedge d\overline{z}}{2\pi i},\quad F_z^{(N+1)}\in\cC^\infty_z,$$ 
by $O(k^{-\infty})$ in Theorem \ref{t-gue230123yyd}. 
We can hence approximate $\chi_k(T_P)$ by semi-classical 
Fourier integral operators of complex phase in 
Theorem \ref{thm:chi_k(T_P) by Psi}. 
Moreover, using the global theory of semi-classical Fourier integral operator 
with complex phase, we will see in 
Theorem \ref{thm:ExpMain Sec 4} that we can 
choose any suitable phase functions for such approximation 
with an explicit transformation rule for the leading term. 
In particular, we conclude Theorem \ref{thm:ExpansionMain}. 
	
\section{Properties of Toeplitz operators and a 
$z$-dependent calculus}
\label{Asymptotic expansion for eesolvents of Toeplitz 
operators on CR functions}
In this section we prepare the material for the 
proof of Theorem \ref{thm:ExpansionMain}.

\subsection{The Toeplitz operator $T_P$ and its functional calculus}
We recall that in Theorem \ref{thm:ExpansionMain}, 
we consider a formally self-adjoint order one pseudodifferential operator 
$P\in L^1_\mathrm{cl}(X)$ with principal symbol 
$\sigma_P\in\mathscr{C}^\infty(T^*X)$ such that
\begin{equation}\label{eq:spp}
\begin{split}
\sigma_P(\xi)> 0.
\end{split}
\end{equation}
This means that $\sigma_P|_\Sigma>0$ everywhere 
for the symplectic cone $\Sigma$ in \eqref{eq:symco}, 
which is a connected component of the 
characteristic set of the Kohn Laplacian.
We consider the first-order elliptic Toeplitz operator
	\begin{equation}
	T_P:=\Pi P\Pi:\cC^\infty(X)\to\mathscr{C}^\infty(X),
	\end{equation}
 and denote its Schwartz kernel by $T_P(x,y)$.
After introducing a properly supported cut-off function near 
the diagonal of $D\times D$, using
 Melin--Sj\"ostrand complex stationary phase formula 
\cite[Theorem 2.3]{MeSj75} and 
Theorem \ref{Boutet-Sjoestrand theorem}, 
we can follow \cite[Theorem 4.4]{GaHs23} 
to describe the singularities of $T_P$ in the following result.
\begin{theorem}\label{BdM-Sj thm for Toeplitz operators}
In the situation in Theorem \ref{thm:ExpansionMain} 
the kernel $T_P(x,y)$ can be expressed 
in any coordinate chart $(D,x)$ on $X$ by an oscillatory integral
as follows,
\begin{equation}\label{e-gue221217yyd}
\begin{split}
&T_P(x,y)\equiv T_\varphi(x,y)~\text{on}~X\times D,\\
&T_\varphi(x,y)=\int_0^{+\infty} e^{it\varphi(x,y)}a(x,y,t)dt,
\end{split}
\end{equation}
where $\varphi$ is described in 
Remark \ref{rmk:varphi(x,y)=x_2n+1+g(x',y)} and
\begin{equation}\label{e-gue221218ycd}
\begin{split}
&a(x,y,t)\in S^{n+1}_{\operatorname{cl}}(D\times D\times{\R}_+),\\
&\mbox{$a(x,y,t)\sim\sum^{+\infty}_{j=0}a_j(x,y)t^{n+1-j}$ 
in $S^n_{1,0}(D\times D\times\mathbb R_+)$},\\
&\mbox{$a_j(x,y)\in\mathscr{C}^\infty(D\times D)$, $j=0,1,2,\ldots$},\\
&\mbox{$a_0(x,x)=s^\varphi_0(x,x)\sigma_P(\xi_x)$ },\\
&\mbox{$a(x,y,t)$, $a_j(x,y)$ are properly supported on 
$D\times D$, $j=0,1,\ldots$,}
\end{split}
\end{equation}
and $s^\varphi_0(x,x)$ is the leading term of $\Pi$ as in \eqref{eq:s(x,y,t)}.
\end{theorem}
The continuity properties of $P$ and $\Pi$ imply that
$T_P$ is a bounded operator between the Sobolev spaces
$H^{s+1}$ and $H^s$ for all $s\in\mathbb{R}$, 
\begin{equation}
\label{eq:Sobolev boundedness of T_P}
T_P=O_{H^{s+1}\to H^s}(1).
\end{equation} 
Using \cite[Lemma 12.2]{BG81}
and standard elliptic estimates we have the following.
\begin{theorem}\label{thm:elliptic estimate for T_P}
In the situation in Theorem \ref{thm:ExpansionMain},
for every $s\in\R$ there exists $C_s>0$ such that 
\begin{equation}
\|u\|_{H^{s+1}}\leq C_s(\|T_Pu\|_{H^s}+\|u\|_{H^s})
\end{equation} 
for all  $u\in H^{s+1}(X)$ with $u=\Pi u$. 
In particular, given an eigenvalue $\lambda\in\mathbb R$, $\lambda\neq0$, 
of $T_P$, then for any $s\in\Z$ there exists a constant
$\widehat C_s>0$ such that 
\begin{equation}
\label{eq:eigenfunction estimates}
\|u\|_{H^s}\leq\widehat C_s (1+|\lambda|)^s\|u\| 
\end{equation} 
for all $u\in\Ker(T_P-\lambda I)$. 
Moreover,~\eqref{eq:eigenfunction estimates} 
holds with \(\lambda=0\) for all \(u\in \ker T_P\cap H_b^0(X)\).
\end{theorem}
From Theorem \ref{thm:elliptic estimate for T_P} we infer immediately: 
\begin{theorem}\label{thm:T_P is self-adjoint}
In the situation in Theorem \ref{thm:ExpansionMain}, 
the maximal extension of $T_P$ defined by  
\begin{equation}
T_P:\Dom T_P\subset L^2(X)\to L^2(X),\quad
\Dom T_P:=\big\{u\in L^2(X):~T_P u\in L^2(X)\big\},
\end{equation}
where $T_P u=\Pi P\Pi u$ is defined in the sense of distributions,
is a self-adjoint operator. 
\end{theorem} 
Moreover, we have the following.
\begin{theorem}\label{thm:Spec(T_P)}
In the situation in Theorem \ref{thm:ExpansionMain},
the spectrum $\Spec(T_P)\subset\R$ of $T_P$
consists only of isolated
eigenvalues, is bounded from below
and has only $+\infty$ as a point of accumulation.
Moreover, for every 
$\lambda\in\Spec(T_P)\setminus\{0\}$, the eigenspace 
$\Ker(T_P-\lambda I)$
is a finite dimensional subspace of $H^0_b(X)\cap\cC^\infty(X)$.
\end{theorem}
\begin{proof}
By~\cite[Proposition 2.14]{BG81} the spectrum of
the operator $T_{P}|_{H^0_b(X)}:H^0_b(X)\to H^0_b(X)$
consists only of isolated
eigenvalues of finite multiplicity, is bounded from below
and has only $+\infty$ as a point of accumulation.
The conclusion follows from the fact that
$\Spec(T_P)\setminus\{0\}=
\Spec(T_P|_{H^0_b(X)})\setminus\{0\}$ and the restrictions to 
$\R\setminus\{0\}$ of the spectral measures
of $T_P$ and $T_P|_{H^0_b(X)}$ coincide.
\end{proof}

Combining Theorems \ref{thm:elliptic estimate for T_P}
and \ref{thm:Spec(T_P)}, by functional calculus of self-adjoint operators, 
we can define the operator $\chi_k(T_P)$ and \eqref{eq:KerChiTp} holds. 
In fact, with the same notations used in \eqref{eq:KerChiTp},
for any $\epsilon>0$ and for the indicator function 
$\mathds{1}_{[k^{1-\epsilon},k^{1+\epsilon}]}$, we also have
\begin{equation}
\label{eq:1_k(T_P)}
\mathds{1}_{[k^{1-\epsilon},k^{1+\epsilon}]}(T_P)(x,y)=\sum_{\lambda_j\in[k^{1-\epsilon},k^{1+\epsilon}]}f_j(x)\overline{f_j(y)}.
\end{equation}
This formula will be used in \S 4.

The main analytical objective in this paper is to study the 
spectral asymptotics of $T_P$ in the semi-classical 
limit by applying Helffer--Sj\"ostrand formula 
(cf.\ \cite[\S 2]{Dav95} or \cite[\S 8]{DiSj99}) 
to $T_P$ by a sequence of rescaled cut-off functions.
For any fixed $0<\delta_1<\delta_2<\infty$, and a fixed 
$\chi\in\cCc(\mathbb{R}_+)$ satisfying 
$\supp(\chi)\subset(\delta_1,\delta_2)$, 
$\chi\not\equiv 0$, we let $\chi_{k}(\lambda):=\chi\left(k^{-1}\lambda\right)$. 
The Helffer--Sj\"ostrand formula implies that
	\begin{equation}
	\label{eq:HS formula for T_P}
	\chi_k(T_P)=\frac{1}{2\pi i}\int_{\mathbb{C}}\frac{\partial\widetilde{\chi_k}}{\partial\overline{z}}(z-T_P)^{-1}dz\wedge d\overline{z},
	\end{equation}
	where $\widetilde{\chi_k}$ is an almost analytic extension of $\chi_{k}$. The integral \eqref{eq:HS formula for T_P} is a well-defined operator from $L^2(X)$ to $L^2(X)$. Notice that $\supp(\chi_k)\cap\{0\}=\emptyset,$
	hence
	\begin{equation}
	\label{eq:chi_k(T_P)Pi}
	\chi_k(T_P)=\chi_k(T_P)\Pi=
	\frac{1}{2\pi i}\int_{\mathbb{C}}
	\frac{\partial\widetilde{\chi_k}}{\partial\overline{z}}(z-T_P)^{-1}\Pi~ 
	dz\wedge d\overline{z}.
	\end{equation}
\begin{remark}\label{remark:supp tilde chi_k}
The following observation will be used frequently in \S 4. 
Let $\supp\chi\subset(\delta_1,\delta_2)$. Since for each $k$ 
we have
\begin{equation}
\chi_k(T_P)=\chi\left(k^{-1}T_P\right)=
\int_\C\frac{\partial\widetilde\chi}{\partial\overline{z}}(z-k^{-1}T_P)^{-1}\frac{dz\wedge d\overline{z}}{2\pi i},
\end{equation}	
which is independent of the choice of the almost analytic extension 
of $\widetilde{\chi}$, we can take in \eqref{eq:chi_k(T_P)Pi}
the almost analytic extension of $\chi_k$ to be
$\widetilde{\chi_k}(z)=\widetilde{\chi}(\frac{z}{k})$
with a special $\widetilde\chi$, so that 
$\supp(\widetilde{\chi})\subset\{z\in\C:
\delta_1/2<|z|<2\delta_2\}$ and for 
$k$ sufficiently large we have
\begin{equation}\label{eq:|z| approx k}
\supp\widetilde{\chi_k}\subset
\Big\{z\in\C: \frac{k\delta_1}{2}<|z|<2k\delta_2\Big\}.
\end{equation}
\end{remark}
	
\subsection{An asymptotic expansion of $(z-T_P)^{-1}\Pi$}
Inspired by the functional calculus of zeroth-order 
elliptic self-adjoint Toeplitz operators on CR manifolds 
\cite[Theorem 1.1]{GaHs21}, in this section we will show 
how to establish the asymptotic expansion of $(z-T_P)^{-1}\Pi$
by a series of $z$-dependent Fourier integral operators. 
	
We first define classes of $z$-dependent smoothing operators, 
that appear in the remainder of the expansion of 
$(z-T_P)^{-1}\Pi$. Furthermore, in \S 4, 
we will show that they only contribute as 
$O(k^{-\infty})$ when we plug the expansion of $(z-T_P)^{-1}\Pi$
into Helffer--Sj\"ostrand formula.  
We need a cut-off $\tau$ in order to avoid a possible blow-up 
of the negative power of $t$ near $0$ in the asymptotic expansion:
	\begin{equation}
	\label{eq: cut off tau}
	\tau\in\cC^\infty(\R),~\tau(t)=0\;\;\text{for $t\leq 1$},
	~\tau(t)=1\;\;\text{for $t\geq 2$}.
	\end{equation}
Let $\varepsilon>0$ be fixed but arbitrary.
\begin{definition}\label{def:z-depnednt smoothing operator type 1}
We denote by $\widetilde{\mathcal{R}}_z$ the set
of finite linear combinations of operators
\begin{equation}
\int_0^{+\infty}\mathcal{S}(x,y,t)
\frac{(1+z)^{M_2}}{(z-t)^{M_1}}\tau(\varepsilon t)dt
\end{equation}	
over $\C$, with symbol $\mathcal{S}\in 
S^{-\infty}(D\times D\times{\R}_+)$
properly supported on $D\times D$, and $M_1,M_2\in\N_0$.
\end{definition} 
\begin{definition}\label{d-gue230123yyd}
We denote by $\widehat{\mathcal{R}}_z$
the set of finite linear combinations of operators
\begin{equation}
\int_0^{+\infty}\mathcal{S}(x,y,t)\frac{(1+z)^{M_2}}{(z-t)^{M_1}}
\tau(\varepsilon t)dt
\end{equation}
over $\C$, with symbol $\mathcal{S}\in S^{-\infty}(X\times D\times{\R}_+)$, 
$M_1,M_2\in\N_0$.
\end{definition} 
\begin{definition}
\label{def:z-depnednt smoothing operator type II}
We denote by $\mathcal{R}_z$ the set of finite linear combination of operators
\begin{equation}\label{eq:mathscr R_z}
\int_0^{+\infty} e^{it\psi(x,y)}\mathscr{S}(x,y,t)
\frac{(1+z)^{M_2}}{(z-t)^{M_1}}\tau(\varepsilon t)dt
\end{equation}
where $\mathscr{S}(x,y,t)=O(|x-y|^{+\infty})$ is a symbol 
in $S^m_{\rm cl}(D\times D\times{\R}_+)$, $m\in\mathbb R$, 
$\mathscr{S}(x,y,t)$ is properly supported on $D\times D$, 
$M_1,M_2\in\N_0$, 
and $\psi\in {\Ph}(\Pi,\Lambda\xi,D)$ 
(cf.\ Definition \ref{d-gue230125yyd}). 
\end{definition} 
\begin{definition}\label{def:C infty_z}
 We define the set
\begin{equation}
\begin{split}
&\mathscr{R}_z:=\{\sum_{j\in J}c_ju_j:
c_j\in\C,~u_j\in \widetilde{\mathcal{R}}_z\cup
{\mathcal{R}}_z, |J|<+\infty\},\\
&\mathscr{C}^\infty_z:=\{\sum_{j\in J}c_ju_j:
c_j\in\C,~u_j\in \widetilde{\mathcal{R}}_z\cup
\widehat{\mathcal{R}}_z\cup\mathcal{R}_z, |J|<+\infty\}.
\end{split}
\end{equation}
\end{definition}	
	
Second, we introduce another class of $z$-dependent Fourier 
integral operators, which is the main part in the expansion of 
$(z-T_P)^{-1}\Pi$. From now on, we fix an open set $D$ of $X$ 
with local coordinates $x=(x_1,\ldots,x_{2n+1})$ on $D$. 
\begin{definition}\label{d-gue221218yyd}
Let $\psi\in {\Ph}(\Pi,\Lambda\xi,D)$ 
(cf.\ Definition \ref{d-gue230125yyd})
and $m\in\mathbb Z$. We let $\widehat S^{(m)}_\psi[z]$ 
be the set of all properly supported $z$-dependent continuous 
operators $E_z: \cC^\infty_c(D)\to\cC^\infty(X)$ such that 
the distribution kernel of $E_z$ satisfies 
\begin{equation}    
E_z(x,y)=\int^{+\infty}_0 e^{it\psi(x,y)}\frac{\sum^p_{j=0}
e_j(x,y,t)z^j}{(z-t)^\ell}\tau(\varepsilon t)t^{n-m}dt+F(x,y),
\end{equation}
where $F\in\mathscr{R}_z$, $\ell\in\N$, $p\in\N_0$, 
$p\leq\ell$, $e_j(x,y,t)\in S^{\ell-j}_{\rm cl}(D\times D\times\mathbb R_+)$,
$e_j(x,y,t)\sim\sum_{h=0}^{+\infty}e_{j,h}(x,y)
t^{\ell-j-h}$ in $S^{\ell-j}_{1,0}(D\times D\times\mathbb R_+)$, 
$e_j(x,y,t)$ and $e_{j,h}(x,y)$ are properly supported on 
$D\times D$, $j=0,1,\ldots,p$, $h\in\N_0$, and $\tau$ is as in \eqref{eq: cut off tau}.
\end{definition}
	
Finally, from the proof of \cite[Lemma 4.1]{GaHs23}, 
we can deduce the following crucial lemma which 
will be frequently used in this section.	
\begin{lemma}\label{Hsiao's lemma}
Let $m\in\R$, $\psi\in {\Ph}(\Pi,\Lambda\xi,D)$
and $I_{\psi,q}^{(m)}$ be a Fourier integral operator given by
\begin{equation}
\label{Szego type FIO I^m_varphi,q}
I_{\psi,q}^{(m)}(x,y)=\int_0^{+\infty} e^{it\psi(x,y)}q(x,y,t)dt,
\end{equation}
where $q(x,y,t)\sim\sum_{j=0}^{+\infty} q_j(x,y)t^{m-j}$
in $S^m_{1,0}(D\times D\times{\R}_+)$, and $q(x,y,t)$
and $q_j(x,y)$ are properly supported on $D\times D$, $j\in\N_0$. 
If we assume that $q_0(x,x)=0$ and there is some 
$\varPhi\in{\Ph}(\Pi,\xi,D)$ such that
\begin{equation}
		\label{microlocal self-adjoint by Szego kernel}
		I_{\psi,q}^{(m)}\equiv\Pi_\varPhi I_{\psi,q}^{(m)}\Pi_\varPhi,
		\end{equation}
		then we can find smooth functions $h(x,y)$ and $\rho(x,y)$ properly supported on $D\times D$ with $h(x,x)\neq 0$ and $\rho(x,y)=O\left(|x-y|^{+\infty}\right)$ such that
		\begin{equation}
		\label{q_0 vanishes to infinite order}
		q_0(x,y)=h(x,y)\psi(x,y)+\rho(x,y).
		\end{equation}
	\end{lemma}

We can now state and prove the main result in this section.
\begin{theorem} \label{t-gue221218yyd} 
In the situation in Theorem \ref{thm:ExpansionMain},
for every $N\in\N_0$ there exist $B^{(j)}_z\in\widehat S^{(j+1)}_\psi[z]$, 
$j=0,1,\ldots,N$,  $R^{(N+1)}_z\in\widehat S^{(N+1)}_\psi[z]$ and 
$F^{(N+1)}_z\in\mathscr{C}^\infty_z$, such that on $X\times D$
		\begin{equation}
		\label{e-gue221218yyda}
		(z-T_P)\sum^N_{j=0}B^{(j)}_z(x,y)=\Pi(x,y)+R^{(N+1)}_z(x,y)+F^{(N+1)}_z(x,y),
		\end{equation} 
		where $\psi\in {\Ph}(\Pi,\sigma_P(\xi)^{-1}\xi,D)$ and
		\begin{equation}\label{e-gue221218yydb}
		\begin{split}
		&B_z^{(0)}(x,y)=\int_{0}^{+\infty} e^{it\psi(x,y)}\frac{b_0(x,y,t)}{z-t}\tau(\varepsilon t)t^ndt,\\
		&b_0(x,y,t)\sim\sum_{j=0}^{+\infty}b_{0,j}(x,y)t^{n-j}\in S^0_{1,0}(D\times D\times{\R}_+),\\
		&b_{0,0}(x,x)=\sigma_P(\xi_x)^{-n-1}s^\varphi_0(x,x)=\frac{1}{2\pi^{n+1}}\frac{dV_\xi}{dV}(x)\sigma_P(\xi_x)^{-n-1}.
		\end{split}
		\end{equation}
	\end{theorem} 
	
\begin{proof} 
For every smoothing operator $F:\mathscr E'(D)\to\cC^\infty(X)$
and every $E_z\in \widehat S^{(m)}_\psi[z]$, $m\in\mathbb Z$,
it is straightforward to see that $FE_z\in\cC^\infty_z$. 
From this observation and with the notation in Remark \ref{rmk:varphi(x,y)=x_2n+1+g(x',y)}, we can change 
$T_P$ and $\Pi$ in \eqref{e-gue221218yyda} to $T_\varphi$ in 
\eqref{e-gue221217yyd} and $\Pi_\varphi$ in \eqref{eq:FIO Pi_psi}, 
respectively. 
We first prove \eqref{e-gue221218yyda} for $N=0$. 
Let $\mathfrak{b}^{(0)}(x,y,t)\in 
S^0_{{\rm cl\,}}(D\times D\times\mathbb R_+)$ 
be a properly supported symbol and let $\mathfrak{b}^{(0)}_0(x,y)$ 
be the leading term of $\mathfrak{b}^{(0)}(x,y,t)$.
Let 
\begin{equation}\label{e-gue221219yyd}
\begin{split}
\mathfrak{B}_z^{(0)}(x,y)
:=&\int^{+\infty}_0 e^{it\varphi(x,y)}
\frac{\mathfrak{b}^{(0)}(x,y,t)}{z-t\sigma_P(\xi_x)}
\tau(\varepsilon t)t^n dt\\
=&\int^{+\infty}_0 e^{it\varphi(x,y)}
\frac{t\mathfrak{b}^{(0)}(x,y,t)}{z-t\sigma_P(\xi_x)}
\tau(\varepsilon t)t^{n-1} dt\in\widehat{S}^{(1)}_\varphi[z].
\end{split}		
\end{equation}

The symbol $\mathfrak{b}^{(0)}(x,y,t)$ will be determined later. 
First, we notice that in the sense of oscillatory integral,
\begin{equation}\label{right integration}
\begin{split}
&\mathfrak{B}_z^{(0)}\Pi_\varphi(x,y)\\
=&\int_{D\times\ol\R_+\times\ol\R_+} 
e^{it\varphi(x,w)+i\sigma\varphi(w,y)}
\frac{\tau(\varepsilon t)t^n}{z-t\sigma_P(\xi_x)}
\mathfrak{b}^{(0)}(x,w,t)s^\varphi(w,y,\sigma)dV(w)d\sigma dt\\
=&\int_0^{+\infty}\int_{D\times\ol\R_+} e^{it(\varphi(x,w)
+i\sigma\varphi(w,y))}\mathfrak{b}^{(0)}(x,w,t)
s^\varphi(w,y,t\sigma)dV(w)d\sigma
\frac{\tau(\varepsilon t)t^{n+1}}{z-t\sigma_P(\xi_x)}dt.
\end{split}
\end{equation}
Notice that at $w=x=y$ and $\sigma=1$, 
the complex phase function $\varphi(x,w)+
i\sigma\varphi(w,y)$ satisfies $d_{w,\sigma}(\varphi(x,w)+
i\sigma\varphi(w,y))=0$ and 
$\det\operatorname{Hess}_{w,\sigma}(\varphi(x,w)
+i\sigma\varphi(w,y))\neq 0$. 
It is not difficult to see that the critical value of 
$\widetilde\varphi(x,\widetilde w)+
i\widetilde\sigma\,\widetilde\varphi(\widetilde w,y)$ 
can be taken to be of the form \eqref{e-gue221212yyd}, 
where $\widetilde\varphi$, $\widetilde{w}$ and 
$\widetilde\sigma$ are the almost analytic extensions of $\varphi$,
$w$ and $\sigma$, respectively.
From this observation and the complex stationary phase 
formula of Melin--Sj\"ostrand, it is straightforward to check that 
		\begin{equation}
		\mathfrak{B}_z^{(0)}\Pi_\varphi(x,y)=\int_0^{+\infty} e^{it\varphi(x,y)}\frac{\tau(\varepsilon t)t^n}{z-t\sigma_P(\xi_x)}\beta^{(0)}_0(x,y,t)dt+F_{1,z}(x,y), 
		\end{equation}
		where $F_{1,z}\in\mathscr{C}^\infty_z$.
		Here, $\beta^{(0)}_0(x,y,t)\in S^0_\mathrm{cl}(D\times D\times\mathbb{R}_+)$ has the leading term
		\begin{equation}
		\beta^{(0)}_0(x,x)=\mathfrak{b}^{(0)}_0(x,x).
		\end{equation}
		Now, we handle the term
		\begin{equation}
		(z\Pi_\varphi-T_\varphi)\mathfrak{B}_z^{(0)}\Pi_\varphi.  
		\end{equation}
		Let $a(x,y,t)\in S^{n+1}_{{\rm cl\,}}(D\times D\times\mathbb R_+)$, $a_0(x,y)\in\cC^\infty(D\times D)$ be as in \eqref{e-gue221218ycd}. 
We notice that $a(x,y,t)=t\alpha(x,y,t)\mod S^{-\infty}(D\times D\times{\R}_+)$
for some $\alpha\in S^n_{\operatorname{cl}}(D\times D\times{\R}_+)$,
where
$\alpha(x,y,t)\sim\sum_{j=0}^{+\infty}\alpha_j(x,y)t^{n-j}$ 
in $S^n_{1,0}(D\times D\times\R_+)$, $\alpha(x,y,t)$ and 
$\alpha_j(x,y)$ are properly supported on $D\times D$, $j\in\N_0$,
and $\alpha_0(x,x)=a_0(x,x)=s_0^\varphi(x,x)\sigma_P(\xi_x)$.  
We can hence write
\begin{equation}\label{e-gue221218ycdI}
\begin{split}
&(z\Pi_\varphi-T_\varphi)\left(\mathfrak{B}_z^{(0)}\Pi_\varphi\right)(x,y)\\
=&\int_{D\times\ol\R_+\times\ol\R_+} e^{i\sigma\varphi(x,w)+
it\varphi(w,y)}\left(zs^\varphi(x,w,\sigma)-\sigma\alpha(x,w,\sigma)\right)\\
&\times\frac{\tau(\varepsilon t)t^n}{z-t\sigma_P(\xi_w)}
\beta^{(0)}(w,y,t)dV(w)d\sigma dt+F_{2,z}(x,y),
\end{split}
\end{equation}
where $F_{2,z}\in\mathscr{C}^\infty_z$.
We apply the change of variables
\begin{equation}
t=\sigma_P(\xi_w)^{-1}\widehat{t},~\sigma=
\sigma_P(\xi_w)^{-1}\widehat{t}~\widehat{\sigma}
\end{equation} 
to the oscillatory integral
\begin{equation}
\begin{split}\label{varPhi compose varPhi}	
&\int_{D\times\ol\R_+\times\ol\R_+} e^{i\sigma\varphi(x,w)+
it\varphi(w,y)}\left(zs^\varphi(x,w,\sigma)-\sigma\alpha(x,w,\sigma)\right)\\
&\times\frac{\tau(\varepsilon t)t^n}
{z-t\sigma_P(\xi_w)}\beta^{(0)}(w,y,t)dV(w)d\sigma dt.
\end{split}
\end{equation}
For convenience, we still use $t$ and $\sigma$ instead
of $\widehat{t}$ and $\widehat{\sigma}$ for the later computation. 
By integrating over $w$ and $\sigma$ with Melin--Sj\"ostrand complex stationary phase formula, \eqref{varPhi compose varPhi} becomes
			\begin{equation}
			\label{eq:stationary phase formula for (z-A)Pi B Pi}
			\begin{split}
			&\int_0^{+\infty}\int_{D\times\overline{\mathbb{R}}_+} \exp{\left({it\frac{\sigma\varphi(x,w)+\varphi(w,y)}{\sigma_P(\xi_w)}}\right)}\\
			&\times\left(zs^\varphi\left(x,w,t\frac{\sigma}{\sigma_P(\xi_w)}\right)-\frac{t\sigma}{\sigma_P(\xi_w)}\alpha\left(x,w,t\frac{\sigma}{\sigma_P(\xi_w)}\right)\right)\\
			&\times\beta^{(0)}\left(w,y,\frac{t}{\sigma_P(\xi_w)}\right)
			\frac{t}{\sigma_P(\xi_w)^{2+n}}\tau_\varepsilon\left(\frac{t}{\sigma_P(
				\xi_w)}\right)dV(w)d\sigma\frac{t^n}{z-t}dt\\
			=&\int_0^{+\infty}e^{it\psi(x,y)}\left(z\beta^{(0),s^\varphi}(x,y,t)-t\beta^{(0),\alpha}(x,y,t)\right)\frac{\widehat\tau(\varepsilon t)t^n}{z-t}dt+\int_0^{+\infty}\frac{1}{z-t}\varepsilon_z(x,y,t)dt,
			\end{split}
			\end{equation}
where $\psi\in {\Ph}(\Pi,\sigma_P(\xi)^{-1}\xi,D)$, 
$\beta^{(0),s^\varphi}(x,y,t),\beta^{(0),\alpha}(x,y,t)\in 
S^0_{{\rm cl\,}}(D\times D\times\mathbb R_+)$, 
$$\beta^{(0),s^\varphi}(x,y,t)\sim
\sum_{j=0}^{+\infty}\beta^{(0),s^\varphi}_j(x,y)t^{-j},\quad
\beta^{(0),\alpha}(x,y,t)\sim\sum_{j=0}^{+\infty}
\beta^{(0),\alpha}_j(x,y)t^{-j}$$ in 
$S^0_{1,0}(D\times D\times\mathbb R_+)$, 
$\beta^{(0),s^\varphi}(x,y,t)$, $\beta^{(0),\alpha}(x,y,t)$, 
$\beta^{(0),s^\varphi}_j(x,y)$ and $\beta^{(0),\alpha}_j(x,y)$
are properly supported on $D\times D$ for $j\in\N_0$, and moreover
$\widehat{\tau}(t)\in\cC^\infty(\R_+)$,
$\widehat{\tau}(t)=0$ when $t\leq c_1$ and 
$\widehat{\tau}(t)=1$ when $t\geq c_2$ for some constant 
$c_2>c_1>0$, and $\varepsilon_z(x,y,t)$
is a polynomial in $z$ of order one with coefficients in
$S^{-\infty}(D\times D\times{\R}_+)$. 
So we can arrange \eqref{varPhi compose varPhi} into
\begin{equation}
			\begin{split}
			\int_0^{+\infty} e^{it\psi(x,y)}\frac{z\beta^{(0),s^\varphi}(x,y,t)-t\beta^{(0),\alpha}(x,y,t)}{z-t}\tau(\varepsilon t)t^ndt+F_{3,z}(x,y), 
\end{split}
\end{equation}
where $F_{3,z}\in\mathscr{C}^\infty_z$.
Also, from computing the Hessian of 
$$\frac{\sigma\varphi(x,w)+\varphi(w,y)}{\sigma_P(\xi_w)}$$ 
via Theorem \ref{thm:tangential hessian of varphi} 
and the leading term formula of the symbols 
$$s^\varphi\left(x,w,t\frac{\sigma}{\sigma_P(\xi_w)}\right)
\quad \frac{t\sigma}{\sigma_P(\xi_w)}
\alpha\left(x,w,t\frac{\sigma}{\sigma_P(\xi_w)}\right)$$
it is straightforward to check that 
\begin{equation}\label{eq:beta}
\beta^{(0),s^\varphi}_0(x,x)=\beta^{(0),\alpha}_0(x,x)=
\sigma_P(\xi_x)^{-n-1}\beta^{(0)}_0(x,x)=
\sigma_P(\xi_x)^{-n-1}\mathfrak{b}^{(0)}_0(x,x).
\end{equation}
As the convention \eqref{eq:FIO Pi_psi} before, we write
\begin{equation}
\Pi_\psi(x,y):=\int_{0}^{+\infty} e^{it\psi(x,y)}s^{\psi}(x,y,t)dt,
\end{equation}
where $s^\psi(x,y,t)\sim\sum_{j=0}^{+\infty}s^\psi_j(x,y)t^{n-j}$
in $S^{n}_{\rm cl}(D\times D\times\R_+)$,  
and for \eqref{e-gue221219yyd} now we take
		\begin{equation}
		\label{eq:mathfrak b^0}
		\mathfrak{b}^{(0)}(x,y,t):=\sigma_P(\xi_x)^{n+1}s^\psi_0(x,y).
		\end{equation}
		Let $\widetilde B^{(0)}_z:=\Pi_\varphi\mathfrak{B}^{(0)}_z\Pi_\varphi$. From the discussion above, we see that there are $B^{(0)}_z\in\widehat S^{(1)}_\psi[z]$ and $F_z\in\mathscr{C}^\infty_z$ such that $\widetilde B^{(0)}_z=B^{(0)}_z+F_z$, $B^{(0)}_z$ satisfies \eqref{e-gue221218yydb}, and 
		\begin{equation}\label{e-gue221214yyd}
		\begin{split}
		&A(z,x,y)\\
		:=&(z-T_\varphi)B_{z}^{(0)}(x,y)-\Pi_\psi(x,y)\\
		=&\int_0^{+\infty} e^{it\psi(x,y)}\frac{(z\beta^{(0),s^\varphi}_0(x,y)-t\beta^{(0),\alpha}_0(x,y))-(z-t)s^\psi_0(x,y)}{z-t}\tau(\varepsilon t)t^n dt\\
		+&\widehat R^{(1)}_z(x,y)+\widehat\zeta_z(x,y), 
		\end{split}
		\end{equation}
		where $\widehat R^{(1)}_z\in\widehat S^{(1)}_\psi[z]$, $\widehat\zeta_z\in\mathscr{C}^\infty_z$. 
We consider $A(0,x,y)$ and let $A(0)$ be the properly supported continuous operator $A(0): \cC^\infty_c(D)\to\cC^\infty(X)$ 
		with the distribution kernel $A(0,x,y)$. From \eqref{e-gue221214yyd}, we have $\Pi_\varphi A(0)\Pi_\varphi\equiv A(0)$ and
		\begin{equation}
		A(0,x,y)\equiv\int^{+\infty}_0 e^{it\psi(x,y)}\widehat a(x,y,t)dt,    
		\end{equation}
where $\widehat a(x,y,t)\in S^n_{{\rm cl\,}}
(D\times D\times\mathbb R_+)$, 
$\widehat a(x,y,t)\sim\sum_{j=0}^{+\infty}\widehat a_j(x,y)t^{n-j}$ 
in $S^n_{1,0}(D\times D\times\mathbb R_+)$, $\widehat a(x,y,t)$ 
and $\widehat a_j(x,y)$ are properly supported on 
$D\times D$, $j\in\N_0$, and
\begin{equation}
\widehat a_0(x,y)=\beta^{(0),\alpha}_0(x,y)-s^\psi_0(x,y).   
\end{equation}
From \eqref{eq:beta} and \eqref{eq:mathfrak b^0}, 
we deduce that $\widehat a_0(x,x)=0$.
Thus, we can apply Lemma~\ref{Hsiao's lemma}
to find some smooth functions $h_0(x,y)$ and $\rho_0(x,y)$
properly supported on $D\times D$ with $h_0(x,x)\neq 0$ 
and $\rho_0(x,y)=O(|x-y|^{+\infty})$ such that
		\begin{equation}
		\widehat a_0(x,y)=h_0(x,y)\psi(x,y)+\rho_0(x,y).
		\end{equation}
		This relation allows us to apply integration by parts with respect to $t$ in  \eqref{e-gue221214yyd}, and we have
		\begin{equation}\label{e-gue221214yydI}
		\begin{split}
		A(z,x,y)=&(z-T_\varphi)B_{z}^{(0)}(x,y)-\Pi_\psi(x,y)\\
		=&\int_0^{+\infty} e^{it\psi(x,y)}\frac{z\beta^{(0),s^\varphi}_0(x,y)-zs^\psi_0(x,y)}{z-t}\tau(\varepsilon t)t^n dt+\widetilde R^{(1)}_z(x,y)+\widetilde\zeta_z(x,y), 
		\end{split}
		\end{equation}
		where $\widetilde R^{(1)}_z\in\widehat S^{(1)}_\psi[z]$, $\widetilde\zeta_z\in\mathscr{C}^\infty_z$.
Next, we consider 
		\begin{equation}
		B(x,y):=\left.\frac{\partial}{\partial z}A(z,x,y)\middle|\right._{z=0}    
		\end{equation}
		and let $B$ be the properly supported continuous operator $B: \cC^\infty_c(D)\to\cC^\infty(X)$ 
		with the distribution kernel $B(x,y)$. From \eqref{e-gue221214yydI}, we can check that $\Pi_\varphi B\Pi_\varphi\equiv B$ and
		\begin{equation}
		B(x,y)\equiv\int^{+\infty}_0 e^{it\psi(x,y)}\widehat b(x,y,t)dt,    
		\end{equation}
		$\widehat b(x,y,t)\in S^{n-1}_{{\rm cl\,}}(D\times D\times\mathbb R_+)$, $\widehat b(x,y,t)\sim\sum_{j=0}^{+\infty}\widehat b_j(x,y)t^{n-1-j}$ in $S^{n-1}_{1,0}(D\times D\times\mathbb R_+)$, $\widehat b(x,y,t)$ and $\widehat b_j(x,y)$ are properly supported on $D\times D$, $j\in\N_0$, and 
		\begin{equation}
		\widehat b_0(x,y)=-\beta^{(0),s^\psi}_0(x,y)+s^\psi_0(x,y).  
		\end{equation}
		From \eqref{eq:beta} and \eqref{eq:mathfrak b^0}, we deduce that $\widehat b_0(x,x)=0$. So, again we can apply Lemma~\ref{Hsiao's lemma} and find some smooth functions $h_1(x,y)$, $\rho_1(x,y)$ properly supported on $D\times D$ with $h_1(x,x)\neq 0$ and $\rho_1(x,y)=O(|x-y|^{+\infty})$ such that 
  \begin{equation}
  \widehat b_0(x,y)=h_1(x,y)\psi(x,y)+\rho_1(x,y).    
  \end{equation}
  This relation enable us to do integration by parts with respect to $t$ in \eqref{e-gue221214yydI}, and we can conclude that 
		\begin{equation}
		\label{e-gue221219yydI}
		(z-T_P)B^{(0)}_z(x,y)-\Pi_\psi(x,y)=R^{(1)}_z(x,y)+F^{(1)}_z(x,y), 
		\end{equation}
		where $R^{(1)}_z\in\widehat S^{(1)}_\psi[z]$, $F^{(1)}_z\in\mathscr{C}^\infty_z$. From the observation in the beginning of our discussion and \eqref{e-gue221219yydI}, we finish the proof of \eqref{e-gue221218yyda} for $N=0$. 
		
		Now, we assume that \eqref{e-gue221218yyda} holds for $N\leq N_0$, for some $N_0\in\N_0$. We are going to prove that \eqref{e-gue221218yyda} holds for $N=N_0+1$. 
		By induction hypothesis, we can find $B^{(j)}_z\in\widehat S^{(j+1)}_\psi[z]$, $j=0,1,\ldots,N_0$, where $\psi\in {\Ph}(\Pi,\sigma_P(\xi)^{-1}\xi,D)$, $B^{(0)}_z$ satisfies \eqref{e-gue221218yydb}, 
		such that 
		\begin{equation}
		\label{eq:induction hypothesis}
		(z-T_P)\sum^{N_0}_{j=0}B^{(j)}_z(x,y)-\Pi_{\psi}(x,y)=R^{(N_0+1)}_z(x,y)+F^{(N_0+1)}_z(x,y),
		\end{equation}
		$R^{(N_0+1)}_z\in\widehat S^{(N_0+1)}_\psi[z]$, $F^{(N_0+1)}_z\in\mathscr{C}^\infty_z$. We write 
		\begin{equation}    R^{(N_0+1)}_z(x,y)=\int^{+\infty}_0e^{it\psi(x,y)}\frac{\sum^p_{j=0}r_j(x,y)t^{\ell-j}z^j}{(z-t)^\ell}\tau(\varepsilon t)t^{n-N_0-1}dt+\widehat R_z(x,y),
\end{equation}
where $\ell, p\in\N_0$, $p\leq\ell$, 
$r_j(x,y)\in\cC^\infty(D\times D)$, 
$r_j(x,y)$ is properly supported on $D\times D$, 
$j=0,1,\ldots,p$, $\widehat R_z\in \widehat S^{(N_0+2)}_\psi[z]$. 
Now, we let 
\begin{multline}    \mathfrak{B}^{(N_0+1)}_z(x,y):=\int^{+\infty}_0e^{it\varphi(x,y)}\frac{-\sum^p_{j=0}r_j(x,y)t^{\ell-j}z^j\sigma_P(\xi_x)^{(\ell-j)-N_0+n}}{(z-t\sigma_P(\xi_x))^{\ell+1}}\tau(\varepsilon t)t^{n-N_0-1}dt
			\end{multline}
			and $\widetilde B^{(N_0+1)}_z(x,y):=(\Pi_\varphi\circ\mathfrak{B}^{(N_0+1)}_z\circ\Pi_\varphi)(x,y)$. We can repeat the procedure \eqref{eq:stationary phase formula for (z-A)Pi B Pi}--\eqref{e-gue221214yyd} as before and deduce that there are $B^{(N_0+1)}_z(x,y)\in\widehat S^{(N_0+2)}_\psi[z]$ and $F_z\in\mathscr{C}^\infty_z$ such that $\widetilde B^{(N_0+1)}_z=B^{(N_0+1)}_z+F_z$ and 
			\begin{multline}
			\label{e-gue221222yyd}
			(z-T_\varphi)B^{(N_0+1)}_z(x,y)+R^{(N_0+1)}_z(x,y)\\
			=\int^{+\infty}_0 e^{it\psi(x,y)}\frac{-\sum^{p+1}_{j=0}\widehat r_j(x,y)t^{\ell+1-j}z^j+\sum^{p+1}_{j=0}\mathring r_j(x,y)t^{\ell+1-j}z^j}{(z-t)^{\ell+1}}\tau(\varepsilon t)t^{n-N_0-1}\\+\widehat \gamma_z(x,y)+\widehat F_z(x,y),
			\end{multline}
			where
\begin{equation}
			\frac{\sum^p_{j=0}r_j(x,y)t^{\ell-j}z^j}{(z-t)^\ell}=\frac{(z-t)\sum^p_{j=0}r_j(x,y)t^{\ell-j}z^j}{(z-t)^{\ell+1}}=\frac{\sum^{p+1}_{j=0}\mathring r_j(x,y)t^{\ell+1-j}z^j}{(z-t)^{\ell+1}},
			\end{equation}			
			 $\widehat r_j(x,y)\in\cC^\infty(D\times D)$, $\widehat r_j(x,x)=\mathring r_j(x,x)$, $j=0,1,\ldots,p+1$, $\widehat \gamma_z\in\widehat S^{(N_0+2)}_\psi[z]$, $\widehat F_z\in\mathscr{C}^\infty_z$. Let $Q(z,x,y):=(z-T_\varphi)B^{(N_0+1)}_z(x,y)+R^{(N_0+1)}_z(x,y)$. From \eqref{eq:induction hypothesis} and the construction of $B_z^{(N_0+1)}$, we have $\Pi_\varphi(\frac{\partial^h}{\partial z^h}Q(z,x,y)|_{z=0})\Pi_\varphi\equiv \frac{\partial^h}{\partial z^h}Q(z,x,y)|_{z=0}$ for each $h=0,1,\ldots,p+1$, so by the argument for \eqref{e-gue221219yydI} with some minor change, we can iteratively apply Lemma~\ref{Hsiao's lemma} to \eqref{e-gue221222yyd} and find some smooth functions $\widehat h_j(x,y)$ and $\widehat \rho_j(x,y)$ properly supported on $D\times D$ with $\widehat h_j(x,x)\neq 0$ and $\widehat \rho_j=O(|x-y|^{+\infty})$, $j=0,1,\ldots,p+1$, such that
			\begin{equation}
			\widehat r_j(x,y)-\mathring r_j(x,y)=\widehat h_j(x,y)\psi(x,y)+\widehat \rho_j(x,y),~ j=0,1,\ldots,p+1.
			\end{equation}
These relations permit to integrate by parts with respect to 
$t$ in \eqref{e-gue221222yyd} and we conclude that 
\begin{equation}
(z-T_P)\sum^{N_0+1}_{j=0}B^{(j)}_z(x,y)-
\Pi_\psi(x,y)=R^{(N_0+2)}_z(x,y)+F^{(N_0+2)}_z(x,y),
\end{equation}
$R^{(N_0+2)}_z(x,y)\in\widehat S^{(N_0+2)}_\psi[z]$, 
$F^{(N_0+2)}_z(x,y)\in\mathscr{C}^\infty_z$.
Thus, we have proved \eqref{e-gue221218yyda} by induction. 
\end{proof}
	
\section{The full asymptotic expansion of $\chi_{k}(T_P)$} \label{s-gue230127yyd}
The goal of this section is to prove Theorem~\ref{thm:ExpansionMain}. 
From now on, we fix  a number $\epsilon>0$ and let 
$\mathds{1}_{[k^{1-\epsilon},k^{1+\epsilon}]}$ 
be the indicator function of $[k^{1-\epsilon},k^{1+\epsilon}]$. 
We also fix a chart $(D,x=(x_1,\ldots,x_{2n+1}))$
on $X$ and assume that $\chi,\chi_k,\widetilde\chi_k$ are as in Theorem \ref{thm:ExpansionMain} and Remark \ref{remark:supp tilde chi_k}.
		\begin{definition}
			\label{def:properly supported semi-classical FIO}
			Let $m\in\mathbb Z$ and $\psi\in {\Ph}(\Pi,\Lambda\xi,D)$. We let $\mathcal{I}^{(m)}_\psi[k]$ 
			be the set of all $k$-dependent continuous operators 
			$H_k: \cCc(D)\to\cC^\infty(X)$ such that 
			the distribution kernel of $H_k$ satisfies 
			\begin{equation}  
			H_k(x,y)=\int^{+\infty}_0 e^{ikt\psi(x,y)}h(x,y,t,k)dt+G_k,
			\end{equation}
			where $G_k=O(k^{-\infty})$ on $X\times D$, $h(x,y,t,k)\sim\sum_{j=0}^{+\infty}h_j(x,y,t,k)$ in $S^{n+1-m}_{\rm loc}(1;D\times D\times\R_+)$, $h_j(x,y,t,k)\in S^{n+1-m-j}_{\rm loc}(1;D\times D\times\R_+)$, $h_j(x,y,t,k_0)\in S^{n-m-j}_{1,0}(D\times D\times\R_+)$ for each fixed $k_0>0$, $h(x,y,t,k)$ and $h_j(x,y,t,k)$ are properly supported on $D\times D$, $j\in\N_0$.
		\end{definition}
	
	\begin{lemma}\label{l-gue230112yyd}
For $m\in\N_0$ and $B^{(m)}_z$ 
	in Theorem~\ref{t-gue221218yyd}, we can find an 
		$A_k^{(m)}\in\mathcal{I}^{(m)}_\psi[k]$ for some $\psi\in {\Ph}(\Pi,\sigma_P(\xi)^{-1}\xi,D)$ such that on $X\times D$ 
		\begin{equation}
		\frac{1}{2\pi i}\int_{\mathbb C}\frac{\partial\widetilde\chi_k}
		{\partial\overline z}B^{(m)}_z(x,y)dz\wedge d\overline z
		=A_k^{(m)}(x,y),
		\end{equation}
		where $\displaystyle A_k^{(m)}(x,y)=\int^{+\infty}_0 e^{ikt\psi(x,y)}a^{(m)}(x,y,t,k)dt$, 
		\begin{equation}
		\mbox{$a^{(m)}(x,y,t,k)\sim\sum^{+\infty}_{j=0}a^{(m)}_j(x,y,t)k^{n+1-m-j}$ in $S^{n+1-m}_{{\rm loc\,}}(1;D\times D\times\mathbb R_+)$,}
		\end{equation}
		and
		\begin{equation}
		{\supp}_t\,a^{(m)}(x,y,t,k)\subset\supp\chi,~{\supp}_t\,a^{(m)}_j(x,y,t)\subset\supp\chi,~j\in\N_0.
		\end{equation}
		Moreover, for $m=0$, 
		\begin{equation}\label{e-gue230113yyd}
		a^{(0)}_0(x,x,t)=\frac{1}{2\pi^{n+1}}\frac{dV_\xi}{dV}(x)~\sigma_P(\xi_x)^{-n-1}~\chi(t)t^n. 
		\end{equation}
	\end{lemma}
	
	\begin{proof}
		We write 
		\begin{equation}
		B^{(m)}_z(x,y)=\int^{+\infty}_0e^{it\psi(x,y)}
		\frac{\sum^p_{j=0}r_j(x,y,t)z^j}{(z-t)^\ell}\tau(\varepsilon t)t^{n-m-1}dt,
		\end{equation}
		where $\ell, p\in\N_0$, $p\leq\ell$, 
		$r_j(x,y,t)\sim\sum_{j=0}^{+\infty} r_{j,h}(x,y)t^{\ell-j}$ in $S^{\ell-j}_{1,0}(D\times D\times\mathbb R_+)$, 
		$r_j(x,y,t)$ and $r_{j,h}(x,y)$ are properly supported on $D\times D$, 
		$j=0,1,\ldots,p$ and $h\in\N_0$. 
From the theory of oscillatory integral, partial integration in $t$ 
and the Cauchy--Pompeiu formula, 
		for large enough $k>0$ we have
		\begin{equation}\label{e-gue230113yydI}
		\begin{split}
		&\frac{1}{2\pi i}\int_{\mathbb C}\frac{\partial\widetilde\chi_k}
		{\partial\overline z}B^{(m)}_z(x,y)dz\wedge d\overline z\\
		=&\frac{1}{2\pi i}\int_{\mathbb C}\left(\int_{\overline{\mathbb R}_+}
		\frac{\partial\widetilde\chi_k}{\partial\overline z}
		e^{it\psi(x,y)}\frac{\sum^p_{j=0}r_j(x,y,t)z^j}{(z-t)^\ell}
		\tau(\varepsilon t)t^{n-m-1}dt\right)dz\wedge d\overline z\\
		=&\frac{1}{2\pi i}\int_{\overline{\mathbb R}_+}\left(\int_{\mathbb C}\frac{\partial\widetilde\chi_k}{\partial\overline z}\frac{(-1)^{\ell-1}}{(\ell-1)!}
		\left(\frac{\partial}{\partial t}\right)^{\ell-1}\Bigr(e^{it\psi}
		\tau(\varepsilon t)t^{n-m-1}
		\sum^p_{j=0}r_j(x,y,t)z^j\Bigr)\frac{1}{z-t}dz\wedge d\overline z\right)dt\\
		=&\int_{\R_+}\chi(k^{-1}t)\sum^{\ell-1}_{q=0}e^{it\psi}
		\psi^q\widehat r_q(x,y,t)dt\\
		=&k\int_{\R_+}\chi(t)\sum^{\ell-1}_{q=0}
		e^{ikt\psi(x,y)}\psi^q\widehat r_q(x,y,kt)dt,
		\end{split}
		\end{equation}
		where $\widehat r_q(x,y,t)\in S^{n-m+q}_{{\rm cl\,}}
		(D\times D\times\mathbb R_+)$, $q=0,1,\ldots,\ell-1$. 
		From \eqref{e-gue230113yydI} 
		and by using integration by parts in $t$ again, our lemma follows. 
	\end{proof}
	
	\begin{lemma}\label{l-gue230114yyd}
		For $m\in\N_0$ and $B^{(m)}_z$ 
		  in Theorem~\ref{t-gue221218yyd}, on $X\times D$
		\begin{equation}
		(\Pi-\mathds{1}_{[k^{1-\epsilon},k^{1+\epsilon}]}(T_P))\circ
		\frac{1}{2\pi i}\int_{\mathbb C}\frac{\partial\widetilde\chi_k}
		{\partial\overline z}B^{(m)}_zdz\wedge d\overline z=O(k^{-\infty}).
		\end{equation}
	\end{lemma}
	
	\begin{proof}
		From Lemma~\ref{l-gue230112yyd}, we have 
		\begin{equation}
		\frac{1}{2\pi i}\int_{\mathbb C}\frac{\partial\widetilde\chi_k}{\partial\overline z}(z)B^{(m)}_z dz\wedge d\overline z=A_k^{(m)},
		\end{equation}
		where 
\begin{equation}
A_k^{(m)}(x,y)=\int^{+\infty}_0 e^{ikt\psi(x,y)}a^{(m)}(x,y,t,k)dt, 
\end{equation}		
$\psi\in {\Ph}(\Pi,\sigma_P(\xi)^{-1}\xi,D)$,
$a^{(m)}(x,y,t,k)\sim\sum_{h=0}^{+\infty}a^{(m)}_h(x,y,t)
k^{n+1-m-h}$ in $S^{n+1-m}_{{\rm loc\,}}
(1;D\times D\times\mathbb R_+)$, 
$a^{(m)}(x,y,t,k)$ and $a^{(m)}_h(x,y,t)$ 
are properly supported on $D\times D$, $h\in\N_0$, 
and $\supp_t\,a^{(m)}(x,y,t,k)\subset\supp\chi$. 
So we only need to prove 
		\begin{equation}
		(\Pi-\mathds{1}_{[k^{1-\epsilon},k^{1+\epsilon}]}(T_P))\circ A_k^{(m)}=O(k^{-\infty}).
		\end{equation}
	From Theorem \ref{thm:Spec(T_P)}, we can write
		\begin{equation}\label{e-gue230114yydIII}
		\begin{split}
		&((\Pi-\mathds{1}_{[k^{1-\epsilon},k^{1+\epsilon}]}(T_P))\circ A_k^{(m)})(x,y)=I(x,y)+II(x,y),\\
		&I(x,y)=\sum_{\lambda_j>k^{1+\varepsilon}}f_j(x)\int A_k^{(m)}(u,y)\overline f_j(u)dV(u),\\
		&II(x,y)=\sum_{\lambda_j<k^{1-\varepsilon}}f_j(x)\int A_k^{(m)}(u,y)\overline f_j(u)dV(u).
		\end{split}
		\end{equation}
		We may assume that $\frac{\partial\psi}{\partial x_{2n+1}}\neq 0$ on $D\times D$. We have 
		\begin{equation}\label{e-gue230114yydV}
		\begin{split}
		&II(x,y)\\
		=&\sum_{\lambda_j<k^{1-\epsilon}}f_j(x)\int\int e^{ikt\psi(u,y)}a^{(m)}(u,y,t,k)\overline f_j(u)dV(u)dt\\
		=&\sum_{\lambda_j<k^{1-\epsilon}}f_j(x)\int\int\frac{\partial}{\partial u_{2n+1}}(e^{ikt\psi(u,y)})\frac{1}{ikt\partial_{u_{2n+1}}\psi}a^{(m)}(u,y,t,k)\overline f_j(u)dV(u)dt\\
		=&\sum_{\lambda_j<k^{1-\varepsilon}}-f_j(x)\int\int e^{ikt\psi(u,y)}\frac{\partial}{\partial u_{2n+1}}\Bigr(\frac{1}{ikt\partial_{u_{2n+1}}\psi}a^{(m)}(u,y,t,k)\overline f_j(u)\Bigr)dV(u)dt.
		\end{split}
		\end{equation}
		Form Theorem~\ref{thm:elliptic estimate for T_P} we conclude that there exists \(n_0\in\N\) such that for any \(\ell\in\N_0\) we find a constant \(C_\ell>0\) with 
		\begin{equation}\label{e-gue230114yydIV}
		\|f_j(x)\|_{\cC^\ell(X)}\leq C_\ell(1+|\lambda_j|)^{n_0+\ell},
		\end{equation}
		for all \(j\in\N\). From \eqref{e-gue230114yydV} and 
		\eqref{e-gue230114yydIV}, we can integrate by parts several 
		times and deduce that $II(x,y)=O(k^{-\infty})$ on $X\times D$. 
		
		Next, we fix a large enough number $N\in\N$ and notice that 
		\begin{equation}\label{e-gue230114yydVI}
		I=\sum_{\lambda_j>k^{1+\epsilon}}\lambda^{-N}_jf_j(x)\int((T_P)^NA_k)(u,y)\overline f_j(u)dV(u). 
		\end{equation}
		From Melin--Sj\"ostrand complex stationary phase formula, it is straightforward to check that we can find some $ A_k^{(m-N)}\in\mathcal I^{(m-N)}_\psi[k]$ such that on $X\times D$
		\begin{equation}\label{e-gue230114yydII}
		((T_P)^N A_k^{(m)})(x,y)=A_k^{(m-N)}(x,y)+O(k^{-\infty})
		\end{equation}
		where 
		\begin{equation}
		A_k^{(m-N)}(x,y)=\int_0^{+\infty} e^{ikt\widehat\psi(x,y)}r(x,y,t,k)dt,
		\end{equation}
		$r(x,y,t,k)\in S^{n+1-m+N}_{{\rm loc\,}}(1;D\times D\times\mathbb R_+)$ is properly supported on $D\times D$, $r(x,y,t,k)=0$ if $t\notin I$
		for some interval $I\Subset\mathbb R_+$, and $\widehat\psi\in {\Ph}(\Pi,\Lambda\xi,D)$. From \eqref{e-gue230114yydIV} and 
		\eqref{e-gue230114yydVI}, for every $\ell\in\N$, every $K\Subset D$, there is a constant $C>0$ independent of $k$ such that 
		\begin{equation}\label{e-gue230114yydVII}
		\|I(x,y)\|_{\cC^\ell(X\times K)}\leq C\sum_{\lambda_j>k^{1+\varepsilon}}\lambda^{-N}_j\lambda_j^{2n_0+\ell}k^{n+1-m+N+\ell}.
		\end{equation}
		From \eqref{e-gue230114yydVII}, we deduce that $I=O(k^{-\infty})$ on $X\times D$. 
	\end{proof}
	\begin{theorem}\label{t-gue230114yyd}
For $m\in\N_0$ and $B^{(m)}_z$ 
in Theorem~\ref{t-gue221218yyd}, we can find some 
$\mathscr B_k^{(m)}\in\mathcal{I}^{(m)}_\psi[k]$ 
with $\psi\in {\Ph}(\Pi,\sigma_P(\xi)^{-1}\xi,D)$
such that on $X\times D$
\begin{equation}
\label{e-gue230116yyd}
\Bigr(\mathds{1}_{[k^{1-\epsilon},k^{1+\epsilon}]}(T_P)
\circ\frac{1}{2\pi i}\int_{\mathbb C}
\frac{\partial\widetilde\chi_k}{\partial\overline z}(z)B^{(m)}_z 
dz\wedge d\overline z\Bigr)(x,y)
=\mathscr B_k^{(m)}(x,y)+O(k^{-\infty}),
\end{equation}
where $\displaystyle \mathscr B_k^{(m)}(x,y)=
\int^{+\infty}_0 e^{ikt\psi(x,y)}b^{(m)}(x,y,t,k)dt$,
		\begin{equation}
		\label{eq:b^(m)(x,y,t,k)}
		\mbox{$b^{(m)}(x,y,t,k)
		\sim\sum^{+\infty}_{j=0}b^{(m)}_j(x,y,t)k^{n+1-m-j}$ in $S^{n+1-m}_{{\rm loc\,}}(1;D\times D\times\mathbb R_+)$},
		\end{equation}
		\begin{equation}
		\mbox{$\supp_t b^{(m)}(x,y,t,k)\subset\supp\chi$,  $\supp_t b^{(m)}_j(x,y,t)\subset\supp\chi$},~j\in\N_0.
		\end{equation}
		Moreover, for $m=0$, 
		\begin{equation}
		\label{e-gue230121yyd}
		b^{(0)}_0(x,x,t)=\frac{1}{2\pi^{n+1}}\frac{dV_\xi}{dV}(x)~\sigma_P(\xi_x)^{-n-1}~\chi(t)t^n. 
		\end{equation}
	\end{theorem}
	
	\begin{proof}
		From Lemma~\ref{l-gue230114yyd}, we get on $X\times D$
		\begin{equation}\label{e-gue230116yydI}
		\begin{split}
		&\mathds{1}_{[k^{1-\varepsilon},k^{1+\varepsilon}]}(T_P)\circ\frac{1}{2\pi i}\int_{\mathbb C}\frac{\partial\widetilde\chi_k}{\partial\overline z}B^{(m)}_z(x,y)dz\wedge d\overline z\\
		=&\Pi\circ\int_{\mathbb C}\frac{\partial\widetilde\chi_k}{\partial\overline z}(z)B^{(m)}_z(x,y)dz\wedge d\overline z+O(k^{-\infty}).
		\end{split}
		\end{equation}
		From Theorem~\ref{Boutet-Sjoestrand theorem}, Lemma~\ref{l-gue230112yyd}, \eqref{e-gue230116yydI} and by using the complex stationary phase formula of Melin--Sj\"ostrand, we get \eqref{e-gue230116yyd}. The computation is straightforward, we omit the details. 
	\end{proof}
	
	\begin{lemma}\label{l-gue230114ycd}
For $A_z\in\widehat S^{(N)}_\psi[z]$, $N\in\N$, and a fixed number $\lambda\in[k^{1-\epsilon},k^{1+\epsilon}]$, where $k>0$ is large enough, we put 
		\begin{equation}\label{e-gue230114yydf}
		A_{k,\lambda}(x,y):=\frac{1}{2\pi i}\int_{\mathbb C}\frac{\partial\widetilde\chi_k}{\partial\overline z}\frac{1}{z-\lambda}A_z(x,y)dz\wedge d\overline z.
		\end{equation}
		Then, for every small enough $\delta>0$, every $\nu\in\N$ with $n-N+\nu+1<0$, every $K\Subset D$ and every large enough $k>0$, there is a constant $C>0$ independent of $k$ and $\lambda$ such that
		\begin{equation}\label{e-gue230114ycdq}
		\|A_{k,\lambda}(x,y)\|_{\cC^{\nu}(X\times K)}\leq Ck^{1+\delta}.
		\end{equation}
	\end{lemma}
	
	\begin{proof}
		Write 
		\begin{equation}    
		A_z(x,y)=\int^{+\infty}_0e^{it\psi(x,y)}\frac{\sum^p_{j=0}r_j(x,y,t)z^j}{(z-t)^\ell}t^{n-N}\tau(\varepsilon t)dt,
		\end{equation}
		where $p, \ell\in\N_0$, $p\leq\ell$, $r_j(x,y,t)\in S^{\ell-j}_{{\rm cl\,}}(D\times D\times\mathbb R_+)$, $r_j(x,y,t)$ is properly supported on $D\times D$, $j=0,1,\ldots,p$.
		For simplicity, we assume that $r_j(x,y,t)z^j=r_j(x,y)t^{
			\ell-j}z^j$, $r_j\in\cC^\infty(D\times D)$, $j=0,1,\ldots,p$. By using the expansion $t^{\ell-j}=(z-(z-t))^{\ell-j}$, without loss of generality, we suppose that 
		\begin{equation}
		A_z(x,y)=\int^{+\infty}_0e^{it\psi(x,y)}\frac{r(x,y)z^\ell}{(z-t)^\ell}t^{n-N}\tau(\varepsilon t)dt,
		\end{equation}
		$r(x,y)\in\cC^\infty(D\times D)$. Fix a small enough number $\delta_0>0$. We have 
		\begin{equation}\label{e-gue230114ycd}
		\begin{split}
		&A_{k,\lambda}(x,y)=I(x,y)+II(x,y),\\
		&I(x,y)=\int^{+\infty}_0\int_{|{\rm Im\,}z|\leq k^{1-\delta_0}}\frac{\partial\widetilde\chi_k}{\partial\overline z}\frac{r(x,y)z^\ell}{(z-\lambda)(z-t)^\ell}e^{it\psi(x,y)}\tau(\varepsilon t)t^{n-N}dt\frac{dz\wedge d\overline z}{2\pi i},\\
		&II(x,y)=\int^{+\infty}_0\int_{|{\rm Im\,}z|\geq k^{1-\delta_0}}\frac{\partial\widetilde\chi_k}{\partial\overline z}\frac{r(x,y)z^\ell}{(z-\lambda)(z-t)^\ell}e^{it\psi(x,y)}\tau(\varepsilon t)t^{n-N}dt\frac{dz\wedge d\overline z}{2\pi i}.
		\end{split}
		\end{equation} 
		From $|\frac{\partial\widetilde\chi_k}{\partial\overline z}(z)|\leq C_Q|\frac{{\rm Im\,}z}{k}|^Q$, for every $Q\in\N$, where $C_Q>0$ is a constant independent of $k$, 
		it is straightforward to check that $I=O(k^{-\infty})$.  
		
		We only need to prove that $II$ satisfies \eqref{e-gue230114ycdq}. For simplicity, we just prove \eqref{e-gue230114ycdq} for $\nu=0$. For general $\nu\in\N$ with $n-N+\nu+1<0$, the proof is similar. We have 
		\begin{equation}\label{e-gue230114ycdr}
		\begin{split}
		&|II(x,y)|\\
		\leq&\frac{1}{2\pi}\int^{+\infty}_0\int_{|{\rm Im\,}z|\geq k^{1-\delta_0}}\left|\frac{\partial\widetilde\chi_k}{\partial\overline z}\right|\frac{|z^\ell|}{|(z-\lambda)(z-t)^\ell|}|r(x,y)|\tau(\varepsilon t)t^{n-N}idtdz\wedge d\overline z\\
		\leq&C\int^{+\infty}_0\int_{|{\rm Im\,}z|\geq k^{1-\delta_0},|z|\leq\widehat Ck}\frac{|z^\ell|}{|{\rm Im\,}z|^{\ell+1}}|r(x,y)|\tau(\varepsilon t)t^{n-N}idtdz\wedge d\overline z,
		\end{split}
		\end{equation} 
		where $C, \widehat C>0$ are constants independent of $k$ and $\lambda$. Let $K\Subset D$. 
		From \eqref{e-gue230114ycdr}, we have 
		\begin{equation}
		|II(x,y)|\leq C_K\frac{k^{\ell+2}}{(k^{1-\delta_0})^{\ell+1}}\int t^{n-N}\tau(\varepsilon t)dt
		\leq\widehat C_Kk^{1+(\ell+1)\delta_0},
		\end{equation}
		where $C_K, \widehat C_K>0$ are constants independent of $k$ and $\lambda$. The lemma follows. 
	\end{proof}
	\begin{theorem}
        \label{t-gue230114yyde}
In the situation in Theorem \ref{thm:ExpansionMain}, 
for any $Q\in\N$ sufficiently large and $\nu\in\N$, 
there is an $N_0\in\N$ such that for every $N\in\N$, 
$N\geq N_0$, we have the following property: 
For $A_z\in\widehat S^{(N)}_{\psi}[z]$ and  
		\begin{equation}
		\widehat A_k:=\mathds{1}_{[k^{1-\epsilon},k^{1+\epsilon}]}
		(T_P)\circ\frac{1}{2\pi i}\int\frac{\partial\widetilde\chi_k}
		{\partial\overline z}(z-T_P)^{-1}A_z~dz\wedge d\overline z,
		\end{equation}
then for every $K\Subset D$ and large enough $k>0$ there is a constant $C_K>0$ 
independent of $k$ such that
\begin{equation}
\|\widehat A_k(x,y)\|_{\cC^{\nu}(X\times K)}\leq C_K k^{-Q}.
\end{equation}
\end{theorem} 
	
	\begin{proof}
		We have 
		\begin{equation}
		\widehat A_k(x,y)=\sum_{\lambda_j\in[k^{1-\epsilon},k^{1+\epsilon}]}
		f_j(x)\int A_{k,\lambda_j}(u,y)\overline f_j(u)dV(u),
		\end{equation}
		where $A_{k,\lambda_j}$ is as in \eqref{e-gue230114yydf}. For every $M\in\N$, we have 
		\begin{equation}\label{e-gue230114yydg}
		\begin{split}
		\widehat A_k(x,y)=&\sum_{\lambda_j\in[k^{1-\epsilon},k^{1+\epsilon}]}\lambda^{-M}_j
		f_j(x)\int A_{k,\lambda_j}(u,y)\overline{((T_P)^Mf_j)}(u)dV(u)\\
		=&\sum_{\lambda_j\in[k^{1-\epsilon},k^{1+\epsilon}]}\lambda^{-M}_j
		f_j(x)\int\bigr((T_P)^MA_{k,\lambda_j}\bigr)(u,y)\overline f_j(u)dV(u).
		\end{split}
		\end{equation} 
From \cite[Proposition 12.1]{BG81}, 
for $k$ large we have $\sum_{\lambda_j\leq k^{1+\epsilon}}1
\leq k^{(1+\epsilon)(n+1)}$. Combining this estimate with 
Lemma~\ref{l-gue230114ycd}, \eqref{e-gue230114yydg} 
and $\lambda^{-M}_j\leq k^{-M(1-\epsilon)}$, our theorem follows. 
\end{proof}
	
Let $s(x,y,t)\in S^{-\infty}(D\times D\times\mathbb R_+)$ 
be properly supported on $D\times D$. 
Let $F$ be the operator determined by the Schwartz kernel
\begin{equation}
\frac{1}{2\pi i}\int_{\mathbb C}\int^{+\infty}_0
\frac{\partial\widetilde\chi_k}{\partial\overline z}(z-T_P)^{-1}
s(x,y,t)\frac{(1+z)^{M_2}}{(z-t)^{M_1}}\tau(\varepsilon t)
dtdz\wedge d\overline z,
\end{equation}
where $M_1, M_2\in\N_0$, and
\begin{equation}
R_k(x,y):=\Bigr(\mathds{1}_{[k^{1-\epsilon},
k^{1+\epsilon}]}(T_P)\circ F\Bigr)(x,y).
\end{equation}
	We have 
	\begin{equation}\label{e-gue230115yyd}
	\begin{split}
	&R_k(x,y)\\
	=&\sum_{\lambda_j\in[k^{1-\epsilon},k^{1+\epsilon}]}
	\frac{f_j(x)}{2\pi i}\int_{\mathbb C}\int^{+\infty}_0\frac{\partial\widetilde\chi_k}{\partial\overline z}(z-\lambda_j)^{-1}s(u,y,t)\frac{(1+z)^{M_2}}{(z-t)^{M_1}}\tau(\varepsilon t)\overline f_j(u)\\
	&\quad\quad\times dV(u)dtdz\wedge d\overline z\\
	=&\sum_{\lambda_j\in[k^{1-\epsilon},k^{1+\epsilon}]}\lambda^{-M}_j
	\frac{f_j(x)}{2\pi i}\int_{\mathbb C}\int^{+\infty}_0\frac{\partial\widetilde\chi_k}{\partial\overline z}(z-\lambda_j)^{-1}\bigr((T_P)^Ms\bigr)(u,y,t)\frac{(1+z)^{M_2}}{(z-t)^{M_1}}\\
	&\times\tau(\varepsilon t)\overline f_j(u)dV(u)dtdz\wedge d\overline z,
	\end{split}
	\end{equation}
	for every $M\in\N$. From \eqref{e-gue230115yyd} 
	and since 
	$(T_P)^Ms\in S^{-\infty}(X\times D\times\mathbb R_+)$, we deduce that $R_k=O(k^{-\infty})$ on $X\times D$.
	Summing up, we get the following.
	
	\begin{theorem}\label{t-gue230115yyd}
In the situation in Theorem \ref{thm:ExpansionMain},
for a given $F_z\in\widetilde{\mathcal{R}}_z$, on $X\times D$ 
we have
		\begin{equation}
		\mathds{1}_{[k^{1-\epsilon},k^{1+\epsilon}]}(T_P)\circ\frac{1}{2\pi i}\int\frac{\partial\widetilde\chi_k}{\partial\overline z}(z-T_P)^{-1}F_z~dz\wedge d\overline z=O(k^{-\infty}).
		\end{equation}
	\end{theorem} 
	
	We can repeat the proof of Theorem~\ref{t-gue230115yyd} 
	with minor changes and deduce the following.
	
	\begin{theorem}\label{t-gue230123yydI}
In the situation in Theorem \ref{thm:ExpansionMain}, for 
a given $F_z\in\widehat{\mathcal{R}}_z$, on $X\times D$ we have
		\begin{equation}
		\mathds{1}_{[k^{1-\epsilon},k^{1+\epsilon}]}(T_P)\circ\frac{1}{2\pi i}\int\frac{\partial\widetilde\chi_k}{\partial\overline z}(z-T_P)^{-1}F_z~dz\wedge d\overline z=O(k^{-\infty}).
		\end{equation}
	\end{theorem}
	
	\begin{theorem}\label{t-gue230116yyd}
In the situation in in Theorem \ref{thm:ExpansionMain},
for a given $F_z\in\mathcal{R}_z$, on $X\times D$ we have
		\begin{equation}
		\mathds{1}_{[k^{1-\epsilon},k^{1+\epsilon}]}(T_P)
		\circ\frac{1}{2\pi i}\int\frac{\partial\widetilde\chi_k}
		{\partial\overline z}(z-T_P)^{-1}F_z~dz\wedge d\overline z=O(k^{-\infty}).
		\end{equation}
	\end{theorem}
	\begin{proof}
		Write
		\begin{equation}    F_z(x,y)=\int^{+\infty}_0e^{it\psi(x,y)}a(x,y,t)\frac{(1+z)^{M_2}}{(z-t)^{M_1}}\tau(\varepsilon t)dt,
		\end{equation}
		where $a(x,y,t)\in S^m_{{\rm cl\,}}(D\times D\times\mathbb R_+)$  
		is properly supported on $D\times D$, $a(x,y,t)=O(|x-y|^{+\infty})$, $M_1, M_2\in\N_0$. Fix $p\in D$ and assume that $\frac{\partial\psi}{\partial y_{2n+1}}(p,p)\neq0$. From the Malgrange preparation theorem, we have 
		\begin{equation}\label{e-gue230116ycd}
		\psi(x,y)=(y_{2n+1}+\widehat\psi(x,y'))g(x,y)
		\end{equation}
near $(p,p)$, where $\widehat\psi$, $g\in\cC^\infty$, 
$y'=(y_1,\ldots,y_{2n})$. 
When $D$ is small enough, we may assume that \eqref{e-gue230116ycd} holds on $D\times D$ and as \eqref{eq:Im varphi>C|z-w|^2} we also have 
		\begin{equation}\label{e-gue230123yyd}
		{\rm Im\,}\psi(x,y)\geq C|x'-y'|^2\ \ \mbox{on $D\times D$}, 
		\end{equation}
		where $C>0$ is a constant. Let $\widetilde a$ be an almost analytic extension of $a$ in the $y$ variables. For every $N\in\N$, by using Taylor expansion at $y_{2n+1}=-\widehat\psi(x,y')$, we have 
		\begin{multline}
		a(x,y,t)=\widetilde a(x,y,t)\\
		=\sum^n_{j=0}b_j(x,y',t)(y_{2n+1}+\widehat\psi(x,y'))^j+(y_{2n+1}+\widehat\psi(x,y'))^{N+1}r(x,y,t),
		\end{multline}
		where $b_j(x,y',t)\in S^m_{{\rm cl\,}}(D\times D\times\mathbb R_+)$, 
		$r(x,y,t)\in S^m_{{\rm cl\,}}(D\times D\times\mathbb R_+)$, $j=0,1,\ldots,N$. Since $a(x,y,t)=O(|x-y|^{+\infty})$, we have \begin{equation}\label{e-gue230116ycdI}
		b_j(x,y',t)=O(|x'-y'|^{+\infty}),\ \ j=0,1,\ldots,N.
		\end{equation}
		Let $s(x,y,t):=e^{it\psi(x,y)}\sum^N_{j=0}(y_{2n+1}+\widehat\psi(x,y'))^jb_j(x,y',t)$. From \eqref{e-gue230123yyd} and \eqref{e-gue230116ycdI}, we can check that $s(x,y,t)\in S^{-\infty}(D\times D\times\mathbb R_+)$. From this observation and Theorem~\ref{t-gue230115yyd}, we conclude that on $X\times D$
		\begin{equation}\label{e-gue230116ycdh}
		\begin{split}
		&\mathds{1}_{[k^{1-\epsilon},k^{1+\epsilon}]}(T_P)\circ\frac{1}{2\pi i}\int\frac{\partial\widetilde\chi_k}{\partial\overline z}(z-T_P)^{-1}\\
		&\Bigr(\int^{+\infty}_0e^{it\psi(x,y)}\sum^N_{j=0}((y_{2n+1}+\widehat\psi(x,y'))^jb_j(x,y',t)\frac{(1+z)^{M_2}}{(z-t)^{M_1}}\tau(\varepsilon t)dt\Bigr)dz\wedge d\overline z
		\\
		=&O(k^{-\infty})
		\end{split}
		\end{equation}
		By using integration by parts in $t$, we can show that 
		\begin{multline}
		\zeta_z:=\int^{+\infty}_0e^{it\psi(x,y)}(y_{2n+1}+\widehat\psi(x,y'))^{N+1}r(x,y,t)\frac{(1+z)^{M_2}}{(z-t)^{M_1}}\tau(\varepsilon t)dt\\
		=\int^{+\infty}_0e^{it\psi(x,y)}r(x,y,z,t)\widehat\tau_\varepsilon (t)dt,
		\end{multline}
		where $r(x,y,z,t)\in\widehat S^{M_2+m-M_1-(N+1)}_\psi[z]$, $\widehat\tau_\varepsilon\in\cC^\infty(\mathbb R_+)$, $\widehat\tau_\varepsilon=1$ near $\supp\tau_\varepsilon$, $\widehat\tau=0$ near $0$. Let 
		\begin{equation}
		\widehat\zeta_k:=\mathds{1}_{[k^{1-\epsilon},
			k^{1+\epsilon}]}(T_P)\circ\frac{1}{2\pi i}
		\int\frac{\partial\widetilde\chi_k}
		{\partial\overline z}(z-T_P)^{-1}\zeta_z ~dz\wedge d\overline z.
		\end{equation}
		Fix $Q$, $\nu\in\N$. 
		From Theorem~\ref{t-gue230114yyde}, we see that we can take 
		$N$ large enough so that for every $K\Subset D$, there is a constant $C_K>0$ independent of $k$ such that 
		\begin{equation}\label{e-gue230116ycdi}
		\|\widehat\zeta_k(x,y)\|_{\cC^{\nu}(X\times K)}\leq C_Kk^{-Q}.
		\end{equation}
		From \eqref{e-gue230116ycdh} and \eqref{e-gue230116ycdi}, the theorem follows. 
	\end{proof} 
	Theorem~\ref{t-gue230115yyd}, Theorem~\ref{t-gue230116yyd} and 
	Theorem~\ref{t-gue230123yydI} yield the following. 
	
\begin{theorem}\label{t-gue230123yyd}
In the situation in Theorem \ref{thm:ExpansionMain} we have
for a given $F_z\in\mathscr{C}^\infty_z$ on $X\times D$,
		\begin{equation}
		\mathds{1}_{[k^{1-\epsilon},k^{1+\epsilon}]}(T_P)\circ\frac{1}{2\pi i}\int\frac{\partial\widetilde\chi_k}{\partial\overline z}(z-T_P)^{-1}F_z~dz\wedge d\overline z=O(k^{-\infty}).
		\end{equation}
	\end{theorem}
	Theorem \ref{thm:ExpansionMain} now follows from the next two theorems.
	\begin{theorem}
		\label{thm:chi_k(T_P) by Psi} 
In the situation in Theorem \ref{thm:ExpansionMain}, 
we can find some
$\widehat{\mathcal{A}}_k\in\mathcal{I}^{(0)}_\psi[k]$
such that on $X\times D$
\begin{equation}
\label{eq:chi_k(T_P) as Psi semi-classical FIO}
\chi_k(T_P)(x,y)=\widehat{\mathcal{A}}_k(x,y)+O\left(k^{-\infty}\right)
\end{equation}
where
\begin{equation}
\label{zero and first order of varPsi}
\psi\in {\Ph}(\Pi,\sigma_P(\xi)^{-1}\xi,D),
\end{equation}
\begin{equation}
\label{eq:FIO mathcal A_k}
\widehat{\mathcal{A}}_k(x,y)=\int_0^{+\infty} 
e^{ikt\psi(x,y)}A(x,y,t,k)dt,
\end{equation}		
\begin{equation}
\label{leading term of A(x,y,t,k) wrt hat psi}
\begin{split}
&A(x,y,t,k)\in S^{n+1}_{\mathrm{loc}}(1;D\times D\times\mathbb{R}_+),\\
&A(x,y,t,k)\sim\sum_{j=0}^{+\infty} A_j(x,y,t)k^{n+1-j}
~\mathrm{in}~S^{n+1}_{\mathrm{loc}}(1;D\times D\times{\R}_+),\\
&A_j(x,y,t)\in\mathscr{C}^\infty(D\times D\times{\R}_+),\ \ j=0,1,\ldots,\\
&A_0(x,x,t)=\frac{1}{2\pi^{n+1}}\frac{dV_{\xi}}{dV}(x)~
\sigma_P(\xi_x)^{-n-1}~\chi(t)t^n\not\equiv 0,\\
&\supp_t A(x,y,t,k)\subset\supp_t\chi,\ \ 
\supp_t A_j(x,y,t)\subset\supp\chi,\ \ j=0,1,\ldots~.
\end{split}
\end{equation}
\end{theorem} 
\begin{proof} 
		Fix a large enough $N\in\N$. From Theorem~\ref{t-gue221218yyd} and the Helffer--Sj\"ostrand 
		formula, we have for \(k\) large and on $X\times D$ that
		\begin{equation}\label{e-gue230124yyd}
		\begin{split}
		&\chi_k(T_P)\\
  =&\mathds{1}_{[k^{1-\epsilon},k^{1+\epsilon}]}(T_P)\circ\frac{1}{2\pi i}\int\frac{\partial\widetilde\chi_k}{\partial\overline z}(z-T_P)^{-1}\Pi~dz\wedge d\overline z\\
		=&\mathds{1}_{[k^{1-\epsilon},k^{1+\epsilon}]}(T_P)\circ\frac{1}{2\pi i}\int\frac{\partial\widetilde\chi_k}{\partial\overline z}\sum^N_{j=0}B^{(j)}_z dz\wedge d\overline z\\
		-&\mathds{1}_{[k^{1-\epsilon},k^{1+\epsilon}]}(T_P)\circ
		\frac{1}{2\pi i}\int\frac{\partial\widetilde\chi_k}
		{\partial\overline z}(z-T_P)^{-1}
		(R^{(N+1)}_z+F^{(N+1)}_z)dz\wedge d\overline z,
		\end{split}
		\end{equation}
where $R^{(N+1)}_z\in \widehat S^{(N+1)}_\psi[z]$ 
and $F^{(N+1)}_z\in\mathscr{C}^\infty_z$. 
For $b^{(j)}(x,y,t,k)$ in Theorem~\ref{t-gue230114yyd}, 
$j=0,1,\ldots$, we take $A(x,y,t,k)\sim
\sum^{+\infty}_{j=0}b^{(j)}(x,y,t,k)$ in 
$S^{n+1}_{{\rm loc\,}}(1;D\times D\times\mathbb R_+)$. It is clear that 
		$\supp_t A(x,y,t,k)\subset\supp\chi$. Moreover, our theorem follows from \eqref{e-gue230124yyd}, Theorem~\ref{t-gue230114yyd}, 
		Theorem~\ref{t-gue230114yyde} and Theorem~\ref{t-gue230123yyd}. 
	\end{proof} 
	\begin{theorem}
		\label{thm:ExpMain Sec 4}
In the situation in Theorem \ref{thm:ExpansionMain}
and any $\Psi\in {\Ph}(\Pi,\Lambda\xi,D)$, 
we can find some 
$\widehat{\mathcal{B}}_k\in\mathcal{I}^{(0)}_\Psi[k]$
such that on $X\times D$ 
\begin{equation}
\label{eq:asymptotic expansion of chi_k(T_P) Sec 4}
\chi_k(T_P)(x,y)=\widehat{\mathcal{B}}_k(x,y)+O\left(k^{-\infty}\right),
\end{equation}
where
\begin{equation}
\widehat{\mathcal{B}}_k(x,y)=\int_0^{+\infty} 
e^{ikt\Psi(x,y)}B(x,y,t,k)dt,
\end{equation}
\begin{equation}
\label{Eq:LeadingTermMainThm Sec 4}
\begin{split}
&B(x,y,t,k)\in S^{n+1}_{\mathrm{loc}}
(1;D\times D\times{\R}_+),\\
&B(x,y,t,k)\sim\sum_{j=0}^{+\infty} B_{j}(x,y,t)k^{n+1-j}~
\mathrm{in}~S^{n+1}_{\mathrm{loc}}(1;D\times D\times{\R}_+),\\
&B_j(x,y,t)\in\mathscr{C}^\infty(D\times D\times{\R}_+),~j=0,1,2,\ldots,\\
&B_{0}(x,x,t)=\frac{1}{2\pi ^{n+1}}
\frac{dV_{\xi}}{dV}(x)\,\Lambda(x)^{n+1}\,
\chi\left(\sigma_P(\xi_x)\Lambda(x) t\right)t^n\not\equiv 0,
\end{split}
\end{equation}
and for some compact interval $I\Subset\R_+$,
\begin{equation}\label{eq:supp_t A Sec 4}
\begin{split}
		\supp_t B(x,y,t,k),~\supp_t B_j(x,y,t)\subset I,\ \ j=0,1,\ldots~.
		\end{split}
		\end{equation}
		Moreover, for any $\tau_1,\tau_2\in\cC^\infty(X)$ 
		such that $\supp(\tau_1)\cap\supp(\tau_2)=\emptyset$, 
		we have
		\begin{equation}
		\label{Eq:FarAwayDiagonalMainThm Sec 4}
		\tau_1\chi_k(T_P)\tau_2=O\left(k^{-\infty}\right).
		\end{equation}
	\end{theorem} 
	
	\begin{proof} 
		We first prove \eqref{Eq:FarAwayDiagonalMainThm Sec 4}. We may assume that $\supp\tau_2\subset D$ and $\supp\tau_1\cap\overline{D}=\emptyset$. Since \eqref{eq:chi_k(T_P) as Psi semi-classical FIO} holds on $X\times D$ and $A(x,y,t,k)$ in \eqref{eq:FIO mathcal A_k} is properly supported on $D\times D$, it is easy to see that $\tau_1\chi_k(T_P)\tau_2=O\left(k^{-\infty}\right)$.
		
Let $\Psi\in {\Ph}(\Pi,\Lambda\xi,D)$ and let 
$\psi\in {\Ph}(\Pi,\sigma_P(\xi)^{-1}\xi,D)$ 
be as in Theorem~\ref{thm:chi_k(T_P) by Psi}. 
From~\cite[Theorem 2.4]{Hs18}, we can find 
$B(x,y,t,k)\in S^{n+1}_{\mathrm{loc}}
(1;D\times D\times{\R}_+)$, $B(x,y,t,k)$ satisfies 
\eqref{eq:supp_t A Sec 4} such that 
\begin{equation}\label{e-gue230124ycd}
\int_0^{+\infty}e^{ikt\Psi(x,y)}B(x,y,t,k)dt=
\int_0^{+\infty}e^{ikt\psi(x,y)}A(x,y,t,k)dt+O(k^{-\infty})
\end{equation}
on $D\times D$, where $A(x,y,t,k)$ is as in Theorem~\ref{thm:chi_k(T_P) by Psi}.  
Moreover, we can repeat the proof of~\cite[Lemma 6.8]{Hs18} and conclude that the 
leading term of $B$ satisfies the last equation in \eqref{Eq:LeadingTermMainThm Sec 4}. 
From this observation and \eqref{e-gue230124ycd}, the theorem follows.  
\end{proof} 
We extend now
Theorem \ref{thm:ExpansionMain} under 
the more general hypothesis that $\chi\in\cC^\infty_c(\R)$
instead of $\chi\in\cC^\infty_c(\R_+)$.
\begin{theorem}\label{thm:EGCO}
In the situation of Theorem 
\ref{thm:ExpansionMain} consider $\chi\in\cCc(\R)$. Then
for any $\Psi\in {\Ph}(\Pi,\Lambda\xi,D)$
there exists 
$\widehat{\mathcal{B}}_k\in\mathcal{I}^{(0)}_\Psi[k]$ 
such that on $X\times D$ 
			\begin{equation}
			\label{eq:general cut off exapnsion for large k}
			\Pi\circ\chi_k(T_P)(x,y)=\widehat{\mathcal{B}}_k(x,y)+\widehat{\mathcal R}_k(x,y),
			\end{equation}
			where for any $\delta>0$ and $\ell\in\N_0$, 
			\begin{equation}
			|\widehat{R}_k(x,y)|_{\mathscr C^\ell(X\times D)}=O(k^{\delta+\ell});
			\end{equation}			
			\begin{equation}
			\label{eq:hat B_k(x,y) for general cut off}
			\widehat{\mathcal{B}}_k(x,y)=\int_0^{+\infty} 
			e^{ikt\Psi(x,y)}\tau(\varepsilon k t)\widehat B(x,y,t,k)dt,
			\end{equation}		
			\begin{equation}
			\label{eq:hat B for general cut off}
			\begin{split}
			&\widehat B(x,y,t,k)\in S^{n+1}_{\mathrm{loc}}
			(1;D\times D\times{\R}_+),\\
			&\widehat B(x,y,t,k)\sim\sum_{j=0}^{+\infty} \widehat B_{j}(x,y,t)k^{n+1-j}~\text{in}~S^{n+1}_{\mathrm{loc}}(1;D\times D\times{\R}_+),\\
			&\widehat B_{0}(x,x,t)=\frac{1}{2\pi ^{n+1}}
			\frac{dV_{\xi}}{dV}(x)\,\Lambda(x)^{n+1}\,\chi\left(\sigma_P(\xi_x)\Lambda(x)t\right)t^n\not\equiv 0,
			\end{split}
			\end{equation}
			where $\tau$ is as in \eqref{eq: cut off tau}. Moreover, for any $\tau_1,\tau_2\in\cC^\infty(X)$ 
			such that $\supp(\tau_1)\cap\supp(\tau_2)=\emptyset$, 
			 any $\delta>0$ and $\ell\in\N_0$, we have
			\begin{equation}
			\label{eq:general cut off far away diagonal exapnsion for large k}
			|\tau_1\chi_k(T_P)\tau_2(x,y)|_{\cC^\ell (X\times X)}=O\big(k^{\delta+\ell}\big).
			\end{equation}
		\end{theorem}
		\begin{proof}
		We notice that
			\begin{equation}
			\displaystyle \Pi\circ\chi_k(T_P)=\int_\C \frac{\partial\widetilde{\chi}_k}{\partial\overline{z}}(z-T_P)^{-1}\Pi\frac{dz\wedge d\overline{z}}{2\pi i}.
			\end{equation}
			Thus, we can apply the expansion of $(z-T_P)^{-1}\Pi$ as before to study $\Pi\circ\chi_k(T_P)$, and we here only show how to modify the proof of Theorem~\ref{thm:chi_k(T_P) by Psi} to get our theorem. 
			
First of all, for any $\delta>0$, we have the decomposition
\begin{equation}\label{eq:decomp}
\Pi\circ\chi_k(T_P)(x,y)=\sum_{\lambda_j<k^\delta}\chi_k(\lambda_j)
f_j(x)\overline{f}_j(y)+
\sum_{\lambda_j\geq k^\delta}\chi_k(\lambda_j)f_j(x)\overline{f}_j(y),
\end{equation}
where $\{f_j\}_{j=1}^{+\infty}$ are smooth orthonormal CR eigenfunctions.
For the second sum in the right-hand side of \eqref{eq:decomp}
the arguments like in \eqref{e-gue230114yydg}
still work as in the case of $\chi\in\cCc(\R_+)$, except for
\eqref{e-gue230113yydI}. When $\chi\in\cCc(\R_+)$, 
terms such as $\chi(t)\tau(\varepsilon kt)$ always equal $\chi(t)$ for $k$ large. 
But when $\chi\in\cCc(\R)$, we cannot always 
reduce $\tau(\varepsilon kt)\chi(t)$ to $\chi(t)$ even for $k$ large. 
So we can only deduce that on $X\times D$,
$$\sum_{\lambda_j\geq k^\delta}
\chi_k(\lambda_j)f_j(x)\overline{f}_j(y)=
\widehat{B}_k(x,y)+O(k^{-\infty})$$ with 
\eqref{eq:hat B_k(x,y) for general cut off} and 
\eqref{eq:hat B for general cut off}. 
For the first sum in the right-hand side of \eqref{eq:decomp} 
we can only apply the estimates 
$\sum_{\lambda_j<k^\delta}1\leq k^{(n+1)\delta}$
in \cite[Proposition 12.1]{BG81} and \eqref{eq:eigenfunction estimates} 
to see that the $\mathscr C^\ell$-norm of this term is 
bounded by $O(k^{(n+1)\delta+\ell})$. 
By combining the arguments above we obtain the conclusion.
\end{proof}
We end this section by two useful properties  which will be applied later.
\begin{lemma}\label{lem:firstDifferentialMainThm}
In the situation of Theorem~\ref{thm:ExpansionMain}  
we have in $\mathscr{C}^\infty$-topology as $k\to+\infty$,
\begin{equation}
d_x\chi_k(T_P)(x,y)|_{y=x}=i\xi(x)\frac{k^{n+2}}{2\pi^{n+1}}
\frac{dV_{\xi}}{dV}(x)\sigma_P(\xi_x)^{-n-2}
\int_0^{+\infty} t^{n+1}\chi(t)dt+O(k^{n+1}),
\end{equation}
and
\begin{multline}
\left(d_x\otimes d_y\right)\chi_k(T_P)(x,y)|_{y=x}\\=
\xi(x)\otimes\xi(x)\frac{k^{n+3}}{2\pi^{n+1}}
\frac{dV_{\xi}}{dV}(x)\sigma_P(\xi_x)^{-n-3}
\int_0^{+\infty} t^{n+2}\chi(t)dt+O(k^{n+2}).
\end{multline}
\end{lemma}
\begin{proof}
With the notation and the result in 
Theorem~\ref{thm:ExpansionMain} it follows that
\begin{equation}
\begin{split}
		&d_x\chi_k(T_P)(x,y)\\
		=&\int_0^{+\infty} ikt(d_x\varphi(x,y))e^{ikt\varphi(x,y)}{A}(x,y,t,k)dt\mod O\left(k^{-\infty}\right)\\
		+&\int_0^{+\infty} e^{ikt\varphi(x,y)}(d_x{A})(x,y,t,k)dt\mod O\left(k^{-\infty}\right).
		\end{split}
		\end{equation}
		Since \((d_x{A})(x,y,t,k)|_{y=x}=O(k^{n+1})\), \(d_x\varphi(x,y)|_{y=x}=\xi(x)\) and the formula for \({A}_0(x,x,t,k)\), by the change of variable $t\mapsto\sigma_P(\xi_x)^{-1}t$, the first part of the statement follows. Similarly, we obtain
		\begin{multline}
		\left(d_x\otimes d_y\right)\chi_k(T_P)(x,y)\\
		=\int_0^{+\infty} -k^2t^2(d_x\varphi(x,y))\otimes (d_y\varphi(x,y)) e^{ikt\varphi(x,y)}{A}(x,y,t,k)dt + R(x,y,k)
		\end{multline}
		with \(R(x,x,k)=O(k^{n+2})\). 
		Since \(d_y\varphi(x,y)|_{y=x}=-d_x\varphi(x,y)|_{y=x}\), 
		and again by the change of variable $t\mapsto\sigma_P(\xi_x)^{-1}t$, 
		the second part of the statement also follows.
	\end{proof}
	\begin{lemma}\label{lem:evpowersexpansion}
		In the situation of Theorem~\ref{thm:ExpansionMain}
		let \( 0<\lambda_1\leq \lambda_2\leq\ldots\) 
		be the positive eigenvalues of \(T_P=\Pi P\Pi\)
		with a respective orthonormal system of eigenfunctions 
		\(\{f_j\}_{j\in\N}\subset H_b^0(X)\) that is 
		\(T_Pf_j=\lambda_j f_j\), \((f_j,f_j)=1\) and 
		\((f_j,f_\ell)=0\) for all \(\ell,j\in\N\) with \(j\neq \ell\). 
		Consider the function  \(g_{k,m}\colon X\to\R \) 
		defined by \(g_{k,m}(x)=\sum_{j\geq 1}
		\lambda_j^m\chi\left(k^{-1}\lambda_j\right)|f_j(x)|^2\) 
		where \(\chi\in \mathscr{C}_0^\infty(\R_+)\) and \(m\in\N\). 
		We have that in $\mathscr{C}^\infty$-topology,
		\begin{equation}
		g_{k,m}(x)=\frac{k^{n+1+m}}{2\pi^{n+1}}
		\frac{dV_{\xi}}{\sigma_P(\xi)^{n+1}dV}(x)
		\int_\R t^{n+m}\chi(t)dt+ O(k^{n+m}).
		\end{equation} 
	\end{lemma}
	\begin{proof}
		Put \(\tau(t)=t^{m}\chi(t)\). Then \(\tau\in \mathscr{C}_c^\infty(\R_+)\) 
		and $g_{k,m}(x)=k^{m}\tau_k(T_P)(x,x)$.
		Hence, the claims follows immediately  from Theorem~\ref{thm:chi_k(T_P) by Psi}.
\end{proof}
\section{Applications and examples}
\subsection{The trace formula, scaled spectral measures and a 
Weyl law estimate}
We begin by proving Corollary \ref{C:asyk}.
\begin{proof}[Proof of Corollary \ref{C:asyk}]
We notice that \eqref{eq:1.9a} follows directly from \eqref{eq:1.8} and 
		\eqref{eq:1.9}. 
As for \eqref{eq:1.8}, we have 
\eqref{eq:asymptotic expansion of chi_k(T_P)}
on any coordinate patch $(D,x)$ by  and
since $\varphi(x,x)=0$,
\begin{equation}
		\chi_k(T_P)(x,x)\sim\sum_{j=0}^{+\infty}
		\mathcal{A}_j(x)k^{n+1-j}~\text{in}~S^{n+1}_{\operatorname{loc}}(1;D),
		\end{equation}
		where $\mathcal{A}_j\in\mathscr{C}^\infty(D)$ and
  \begin{equation}
		\label{eq: A_j}
\mathcal{A}_j(x)=\int_0^{+\infty} A_j(x,x,t)dt,
  \end{equation}
  with $A_j(x,x,t)$ as in \eqref{Eq:LeadingTermMainThm}.
  Let $(D',y)$ be another coordinate patch
with $D\cap D'\neq\emptyset$ and
		$\chi_k(T_P)(y,y)\sim\sum_{j=0}^\infty
		\mathcal{B}_j(y)k^{n+1-j}$ in $S^{n+1}_{\operatorname{loc}}(1;D')$,
		with $\mathcal{B}_j\in\cC^\infty(D')$.
		Then we have on $D\cap D'$,
		\begin{equation}
		\label{eq: A_j and B_j}
		\sum_{j=0}^{+\infty} \mathcal{A}_j(\cdot)k^{n+1-j}\sim
		\sum_{j=0}^{+\infty} \mathcal{B}_j(\cdot)k^{n+1-j}
		~\text{in}~S^{n+1}_{\operatorname{loc}}(1;D\cap D').
		\end{equation}
The relation \eqref{eq: A_j and B_j} shows that 
$\mathcal{A}_j=\mathcal{B}_j$ on $D\cap D'$ and hence there
exists global functions 
$\mathcal{A}_j\in\cC^\infty(X)$, $j\in\N_0$, such that \eqref{eq:1.8} holds. 
\end{proof}
	\begin{proof}[Proof of 
		Theorem~\ref{thm:ScaledSpectralMeasures}]
		The scaled spectral measure \(\mu_k\)  were defined
		in~\eqref{eq:scaledspectraclmeasuredefinition}. 
		We can choose a orthonormal system of eigenfunctions 
		\(\{f_j\}_{j\in\N}\) in  \(H_b^0(X)\cap\cC^\infty(X)\) with respect to the 
		positive eigenvalues  \( 0<\lambda_1\leq \lambda_2\leq\ldots\) of 
		\(T_P\), that is, \(T_Pf_j=\lambda_j f_j\), 
		\((f_j|f_\ell)=\delta_{j\ell}\) for all \(\ell,j\in\N\). 
		Given \(\chi\in \cCc(\R_+)\) the kernel of $\chi_k(T_P)$
		is given by \eqref{eq:KerChiTp},
		hence 
		\begin{multline}
		\int_X\chi_k(T_P)(x,x)dV=\int_X\sum_{j=1}^{+\infty}
		\chi(k^{-1}\lambda_j)|f_j(x)|^2dV=\sum_{j=1}^{+\infty}
		\chi(k^{-1}\lambda_j)=k^{n+1}\langle\mu_k,\chi\rangle.
		\end{multline}
		The conclusion of Theorem~\ref{thm:ScaledSpectralMeasures}
		follows now from Corollary \ref{C:asyk} 
		because we have for $k\to+\infty$,
		\begin{equation}
		\langle\mu_k,\chi\rangle=
		\frac{1}{k^{n+1}}\int_X\chi_k(T_P)(x,x)dV
		\longrightarrow
		\frac{1}{2\pi^{n+1}}\int_X\frac{dV_\xi}{\sigma_P(\xi_x)^{n+1}}
		\int_{\R_+}\chi(t)t^ndt\,.
		\end{equation}
	\end{proof}
	
\begin{remark}
Let us note that Theorem~\ref{thm:EGCO} and \eqref{eq:19b}
immediately yield a weak form of the Weyl law.
Namely, under the assumptions of Theorem \ref{thm:ExpansionMain}, 
the spectrum counting function $N(k)=\#\{j\in\N:\lambda_j\leq k\}$
of the Toeplitz operator $T_P=\Pi P\Pi$ has the asymptotics 
\begin{equation}
N(k)=\frac{\operatorname{vol}(\Sigma_1)}{2\pi^{n+1}}k^{n+1}
+o(k^{n+1}),\quad k\to+\infty,
\end{equation}
where $\operatorname{vol}(\Sigma_1)$ is the symplectic volume of the
subset $\Sigma_1\subset\Sigma$ where $\sigma_P\leq 1$.
Compared to \cite[Theorem 13.1]{BG81} the remainder is not sharp, 
cf.~\eqref{eq:Nk}.
\end{remark}
\subsection{Some related results on Grauert tubes}
We present in the sequel a series of interesting examples
where Theorem \ref{thm:ExpansionMain} applies.
\begin{example}\label{Ex:ReebFlow}
Associated with the contact form $\xi$ one has the
so-called Reeb vector field $\mathcal{T}=\mathcal{T}(\xi)$, 
uniquely defined by the equations
$\iota_{\mathcal{T}}\xi=1$, $\iota_{\mathcal{T}}d\xi=0$,
		cf.~\cite[Lemma/Definition 1.1.9]{Gei08}.
		It follows that \(dV_\xi:=\frac{2^{-n}}{n!}\xi\wedge (d\xi)^n\) is a volume form on \(X\) 
		which is preserved under the flow of \(\mathcal{T}\). Hence we find that the operator 
		\(-i\mathcal{T}\colon \mathscr{C}^\infty(X)\to \mathscr{C}^\infty(X)\) 
		is formally self-adjoint with respect to the \(L^2\)-inner product induced by
		the volume form \(dV:=dV_\xi\). 
		We can apply Theorem~\ref{thm:ExpansionMain} with \(P:=-i\mathcal{T}\)
		to study the spectral asymptotics of \(T_P=\Pi (-i\mathcal{T}) \Pi\), 
		which is in this case directly linked to the dynamics of the Reeb flow for CR functions. 
		Furthermore, it turns out that in this case the leading coefficient $\mathcal{A}_0(x)$ 
		in the asymptotic expansion of \(\chi_k(T_P)(x,x)\) is constant on \(X\) and equals 
		\(\frac{k^{n+1}}{2\pi^{n+1}}\int\chi(t)t^ndt\). We note that \(\alpha_P=\xi\) 
		holds in that case.
		In this set-up, 
		with the additional assumption that the flow of \(\mathcal{T}\) 
		preserves the CR structure, the asymptotic expansion of \(\chi_k(T_P)\) 
		on the diagonal is due to 
Herrmann--Hsiao--Li \cite[Theorem 1.2]{HHL20}.
	\end{example} 	
	\begin{example}
		\label{ex:space of holomorphic sections for polarized Kahler manifold}
		A particular case of Example \ref{Ex:ReebFlow}, very important
		for K\"ahler geometry and quantization, was introduced
		by Boutet de Monvel--Guillemin \cite{BG81}.
		Let $M$ be a compact K\"ahler manifold and $(L,h^L)\to M$
		be a positive Hermitian holomorphic line bundle, that is,
		the Chern curvature $\frac{i}{2\pi}R^L$ of $(L,h^L)$ represents 
		a K\"ahler metric on $M$. Then the principal circle bundle 
		$X=\{v\in L^*:|v|_{h^*}=1\}$ is a strictly pseudoconvex CR manifold called Grauert tube \cite{Grau62} and plays an
important application to the Kodaira embedding theorem
for singular spaces.
The connection $1$-form $\xi$ on $X$ associated to the Chern connection 
$\nabla^L$ is a contact form on $X$ and the corresponding Reeb vector field
$\mathcal{T}$ is the infinitesimal generator $\partial_\vartheta$ 
of the $S^1$-action on $X$.
In this case, $\xi(-i\partial_\vartheta)=1$, $\partial_\vartheta$ 
commutes with $\overline\partial_b$ and $\Pi$, and 
${\rm Spec}(-i\partial_\vartheta)=\Z$.
Thus the operator $P=-i\partial_\vartheta$ restricted to
$H^0_b(X)$ is an order one elliptic self-adjoint Toeplitz operator. 
Also, its positive spectrum consists of $m\in\N$ with finite multiplicity,
and the eigenspace corresponding to $m$ can be identified with the space
$H^0(M,L^m)$ of holomorphic sections of $L^m$, the $m$-th power of $L$ 
(cf.\ \cite[Lemma 14.14]{BG81}).
Moreover, $m$ large enough
occurs in the spectrum with the multiplicity $\rho(m)$, 
where $\rho$ is the Hilbert polynomial
of $(M,L)$. 
For related results about spectral asymptotics for Toeplitz operators in this
context we refer to \cite{BPU98}.
\end{example}	
\begin{example}
Given a compact real analytic manifold \(M\) of real dimension \(\dim_\R M=n+1\) and real analytic Riemannian metric \(g\), the adapted complex strucuture on a neighborhood in \(TX\) around the zero section induces a CR structure on \(X= \{v\in TX\colon g(v,v)=r\}\), which is strongly pseudoconvex  when \(r>0\) is small enough. The canonical one form \(\lambda\) on \(TX\) together with the Hamiltonian vector field coming from the function \(v\mapsto \sqrt{g(v,v)}\) and the symplectic form \(d\lambda\) induce \(\omega_0\) and \(T\) on \(X\)  satisfying the conditions in Example~\ref{Ex:ReebFlow}. Hence Theorem~\ref{thm:ExpansionMain} applies in that case. We note that in this specific set-up a scaling asymptotic result related to Theorem~\ref{thm:ExpansionMain} for a different class of cut-off functions $\chi$ and different growth order of the semi-classical parameter $k$ is due to Chang--Rabinowitz \cite[Theorem 1.1]{CRa22} by a different method.
	\end{example}
The situation in Example \ref{ex:space of holomorphic sections 
for polarized Kahler manifold} can actually provide more insight into 
Theorem~\ref{thm:ExpansionMain} and Theorem~\ref{thm:EGCO}. Let $B_m:L^2(M,L^m)\to H^0(M,L^m)$ be the Bergman projection. The smooth Schwartz kernel $B_m(p',p'')$ associated to $B_m$ is called the Bergman kernel, and $B_m(p'):=B_m(p',p')$ is called the Bergman kernel function. It is well known that $B_m(p',p'')$ admits the full asymptotic expansion, cf.~\cite{MM07,HM14} and the references therein. In particular,
		\begin{equation}
		\label{eq:BKE}
		B_m(p')\sim\sum_{j=0}^{+\infty} m^{n-j}b_{j}(p')~\text{in}~S^{n}_{\rm loc}(1;M),~n:=\dim_\C M,~m\to+\infty.
		\end{equation}
		Now, we consider the dynamical Toeplitz operator
		\begin{equation}
T_\vartheta:=\Pi(-i\partial_\vartheta)\Pi		
\end{equation}				
on the circle bundle over $M$ in Example \ref{ex:space of holomorphic sections for polarized Kahler manifold}. For $\chi\in\cCc(\R)$, we can check that 
		\begin{equation}
		\label{eq:Tr chi_k(T_P) on circle bundle}
		(\Pi\circ\chi_k(T_\vartheta))(x,x)=\frac{1}{2\pi}\sum_{m\in\Z}\chi\left(\frac{m}{k}\right)B_m(\pi_M(x)),
		\end{equation}
		where $\pi_M:X\to M$ is the natural projection. 
		\begin{theorem}
			\label{thm:Trace on circle bundle}
			In the situation of Example \ref{ex:space of holomorphic sections for polarized Kahler manifold}, when $\chi\in\cCc(\R_+)$ 
			\begin{equation}
			\label{eq:chi_k(T_vartheta)(x,x) full expansion}
			\begin{split}
			&\chi_k(T_\vartheta)(x,x)\sim\sum_{j=0}^{+\infty}a_j(x)k^{n+1-j}
			~\text{in}~S^{n+1}_{\rm loc}(1;X),
			\end{split}
			\end{equation}
			where $\displaystyle a_j(x)=\frac{1}{2\pi}\left(\int_0^{+\infty}\chi(t)t^{n-j}dt\right)b_j(\pi_M(x))$ and $b_j$ are given in \eqref{eq:BKE}, $j\in\N_0$.
		\end{theorem}
		\begin{proof}
			When $\chi\in\cCc(\R_+)$, we have $\Pi\circ\chi_k(T_\vartheta)=\chi_k(T_\vartheta)$. Also, when $k>0$ is large, by \eqref{eq:BKE} and $\supp\chi\subset\R_+$, for each $N\in\N_0$ we have a constant $c_N>0$ such that
			\begin{multline}
			\label{eq:Riemann sum of Bergman kernel function}
			\left|\sum_{m\in\Z}\chi\left(\frac{m}{k}\right)B_m(\pi_M(x))
			-\sum_{j=0}^{N}k^{n+1-j}\sum_{m\in\Z}k^{-1}\chi\left(\frac{m}{k}\right)\left(\frac{m}{k}\right)^{n-j}b_j(\pi_M(x))\right|\\
			\leq  c_N\sum_{m\in\Z}\left|\chi\left(\frac{m}{k}\right)\right|m^{n-N-1}=O(k^{n-N}).
			\end{multline} 
			By the Poisson summation formula, 
cf.~\cite[Theorem 7.2.1]{Hoe03}, for any $\tau\in\cCc(\R)$ we have
			\begin{equation}
			\label{eq:Possion summation formula}
			k^{-1}\sum_{m\in\Z}\tau\left(\frac{m}{k}\right)=\sum_{m\in\Z}\int_{\R}e^{-it(2\pi k m)}\tau(t)dt.
			\end{equation}
			In the right-hand side of the above equation, when $m\neq 0$ we can apply arbitrary times of integration by parts in $t$, and when $m=0$ we just have a number $\displaystyle\int_{\R}\tau(t)dt$. Accordingly, for any $N\in\N$, we can find a constant $C_N>0$ such that
			\begin{equation}
			\left|k^{-1}\sum_{m\in\Z}\tau\left(\frac{m}{k}\right)-\int_\R\tau(t)dt\right|<C_N k^{-N}.
			\end{equation}
			Our theorem follows immediately from this approximation and \eqref{eq:Riemann sum of Bergman kernel function}.
		\end{proof}
		\begin{remark}
			Although the asymptotic expansion in this example also directly follows from the much general Theorem \ref{thm:ExpansionMain}, \eqref{eq:chi_k(T_vartheta)(x,x) full expansion} gives a hint to general formula for the lower order coefficients of $\chi_k(T_P)(x,x)$ that they could be chosen to be independent of the derivatives of $\chi$.  
		\end{remark}

		When $\chi\in\cCc(\R)$, it is more subtle as we observe in Theorem \ref{thm:EGCO}.
  \begin{theorem}
			\label{thm:main thm is sharp}
			In the situation of Example \ref{ex:space of holomorphic sections for polarized Kahler manifold}, for $\chi\in\cCc(\R)$ such that $\chi\geq 0$ and $\chi(0)\neq 0$, we have 
			\begin{equation}
			(\Pi\circ\chi_k(T_\vartheta))(x,x)=\widehat {\mathcal B}_k(x)+\widehat {\mathcal R}_k(x), 
			\end{equation}
			where   
			\begin{equation}
			\begin{split}
			&\widehat{\mathcal B}_k(x)\sim\sum_{j=0}^{+\infty}k^{n+1-j}\widehat b_j(x)~\text{in}~S^{n+1}_{\rm loc}(1;X),\\
			&{\widehat b}_0(x)=\frac{1}{2\pi}\left(\int_0^{+\infty}\chi(t) t^n dt\right)b_0(\pi_M(x)),\\
			\end{split}
			\end{equation}
			where $b_0$ is as in \eqref{eq:BKE} and
			\begin{equation}
			\begin{split}
			&\text{when}~b_{n+1}\not\equiv 0,~\lim_{k\to+\infty}|\widehat{\mathcal R}_k|=+\infty,~\lim_{k\to+\infty}k^{-\delta}|\widehat {\mathcal R}_k|=0~\text{for all}~\delta>0;\\
			&\text{when}~b_{n+1}\equiv 0,~\widehat{\mathcal{R}}_k=O(1).
			\end{split}
			\end{equation}
In particular, the size of error term in $k$ in the approximation
of $\chi_k(T_P)$ in Theorem \ref{thm:EGCO} is sharp when $\chi(0)\neq 0$.
		\end{theorem}
  
 In this case, for every $m\in\Z$, we can not get 
 \eqref{eq:Riemann sum of Bergman kernel function} 
 as before and we need to consider
\begin{equation}
\varepsilon_m(x):=B_m(\pi_M(x))-\mathds{1}_{\{s\geq 0\}}(m)
\sum_{j=0}^{n} m^{n-j}b_j(\pi_M(x))-
\mathds{1}_{\{s\geq 1\}}(m)m^{-1}b_{n+1}(\pi_M(x)).
\end{equation}
		The formula \eqref{eq:BKE} implies that we can find some constants $m_0,c_0>0$ such that for all $m\geq m_0$, $|\varepsilon_m(x)|<c_0~m^{-2}$.  Now, we write
		\begin{equation}
		\label{eq:chi_k(T_vartheta) for general chi}
		\begin{split}
		&(\Pi\circ\chi_k(T_\vartheta))(x,x)\\
		=&\sum_{m\in\Z}\chi\left(\frac{m}{k}\right)B_m(\pi_M(x))\\
		=&\sum_{m=-\infty}^{-1}\chi\left(\frac{m}{k}\right)B_m(\pi_M(x))+\sum_{m=0}^{+\infty}\chi\left(\frac{m}{k}\right)\varepsilon_m(x)+\sum_{m=1}^{+\infty}\chi\left(\frac{m}{k}\right)m^{-1}b_{n+1}(\pi_M(x))\\
		+&\sum_{m=0}^{+\infty}\chi\left(\frac{m}{k}\right)\sum_{j=0}^{n} m^{n-j}b_j(\pi_M(x)).
		\end{split}
		\end{equation}
It is clear that
		\begin{equation}
  \label{eq:O(1) estimate}
		\begin{split}
		&\left|\sum_{m=-\infty}^{-1}\chi\left(\frac{m}{k}\right)B_m(\pi_M(x))+\sum_{m=0}^{m_0-1}\chi\left(\frac{m}{k}\right)\varepsilon_m(x)\right|=O(1),\\
		&\left|\sum_{m=m_0}^{+\infty}\chi\left(\frac{m}{k}\right)\varepsilon_m(x)\right|=O\left(c_0\sum_{m=m_0}^{+\infty}m^{-2}\right)=O(1).
		\end{split}
		\end{equation}
Also, by the classical formula $\sum_{m=1}^{k}m^{-1}=\log k+\gamma+\frac{1}{2k}-\epsilon_k$, where $\gamma$ is the Euler constant and $0\leq\epsilon_k\leq\frac{1}{8 k^2}$, when $\chi\geq 0$ we have 
		\begin{equation}
\label{eq:+infty}		\lim_{k\to+\infty}\sum_{m=1}^{+\infty}\chi\left(\frac{m}{k}\right){m}^{-1}=+\infty,
		\end{equation}
		and for any $\delta>0$
		\begin{equation}
  \label{eq:0}
		\lim_{k\to+\infty}k^{-\delta}\sum_{m=1}^{+\infty}\chi\left(\frac{m}{k}\right){m}^{-1}=0.
		\end{equation}		
		However, we can not apply \eqref{eq:Possion summation formula} as before to study $\displaystyle\sum_{m\in\N_0}\chi\left(\frac{m}{k}\right)\sum_{j=0}^{n} m^{n-j}b_j(\pi_M(x))$, and we need the following approximation.
		\begin{lemma}
			\label{lem:higher order Riemann sum}
			For any $\tau\in\cCc(\R)$ and each $N\in\N_0$, we can find a constant $C_N>0$ and functions $\tau_\ell\in\cCc(\R)$, $\ell=0,1,\ldots,N$, $\tau_0=\tau$, such that 
			\begin{equation}
			\label{eq:higher order Riemann sum}
			\left|\sum_{m\in\N_0}k^{-1}\tau\left(\frac{m}{k}\right)-\sum_{\ell=0}^N k^{-\ell}\int_0^{+\infty}\tau_\ell(t)dt\right|\leq C_N~k^{-N-1},~k\to+\infty.
			\end{equation}
		\end{lemma}
		\begin{proof}
			We start proving by the estimate of Riemann sum that for any $\tau\in\cCc(\R)$ we have a constant $C_0>0$ only depending on $\tau$ such that
			\begin{equation}
			\left|\sum_{m\in\N_0}k^{-1}\tau\left(\frac{m}{k}\right)-\int_0^{+\infty}\tau(t)dt\right|\leq C_0 ~k^{-1}.
			\end{equation}
			If we assume that \eqref{eq:higher order Riemann sum} holds for all $\tau\in\cCc(\R)$ at some $N_0\in\N_0$, then for any $\tau\in\cCc(\R)$ and as $k\to+\infty$, we can apply Taylor formula such that
			\begin{equation}
			\label{eq:Taylor expansion and Riemann sum}
			\begin{split}
			&\sum_{m\in\N_0}k^{-1}\tau\left(\frac{m}{k}\right)-\int_0^{+\infty}\tau(t)dt\\
			=&\sum_{m\in\N_0}\int_{\frac{m}{k}}^{\frac{m+1}{k}}\left(\tau\left(\frac{m}{k}\right)-\tau(t)\right)dt\\
			=&\sum_{\substack{m\in\N_0\\ m\in k\supp\tau}}\int_{\frac{m}{k}}^{\frac{m+1}{k}}\left(-\sum_{j=1}^{N_0}\frac{d^j\tau}{d t^j}\left(\frac{m}{k}\right)\frac{(t-\frac{m}{k})^j}{j!}dt+O(k^{-N_0-1})\right)\\
			=&-\sum_{j=1}^{N_0}\sum_{m\in\N_0}\frac{k^{-j}}{(j+1)!}k^{-1}\frac{d^j\tau}{d t^j}\left(\frac{m}{k}\right)+O(k^{-N_0-1}).
			\end{split}
			\end{equation}
			By applying induction hypothesis to $\displaystyle k^{-1}\frac{d^j\tau}{d t^j}\left(\frac{m}{k}\right)$ in the last line of \eqref{eq:Taylor expansion and Riemann sum}, we can see that \eqref{eq:higher order Riemann sum} holds for any $\tau\in\cCc(\R)$ at $N_0+1$ and our claim follows by induction.
		\end{proof}
\begin{proof}[Proof of Theorem \ref{thm:main thm is sharp}]  
 Lemma \ref{lem:higher order Riemann sum} implies that
\begin{equation}
\begin{split}
\label{eq:expansion hat b_0}
\sum_{m=0}^{+\infty}\chi\left(\frac{m}{k}\right)\sum_{j=0}^{n} m^{n-j}b_j(\pi_M(x))
&=\sum_{j=0}^{n}k^{n+1-j}\sum_{m=0}^{+\infty}k^{-1}\chi\left(\frac{m}{k}\right)\left(\frac{m}{k}\right)^{n-j}b_j(\pi_M(x))\\
&\sim\sum_{j=0}^{+\infty}k^{n+1-j}\widehat b_j(x)~\text{in}~S^{n+1}_{\rm loc}(1;X),\\
&b_0(x)=\frac{1}{2\pi}\left(\int_0^{+\infty}\chi(t)t^n dt\right)b_0(\pi_M(x)).
\end{split}
\end{equation}
By \eqref{eq:chi_k(T_vartheta) for general chi}, \eqref{eq:O(1) estimate}, \eqref{eq:+infty}, \eqref{eq:0}, \eqref{eq:expansion hat b_0} and taking
		\begin{equation}
		\widehat{\mathcal{B}}_k(x):=\sum_{m=0}^{+\infty}\chi\left(\frac{m}{k}\right)\sum_{j=0}^{n} m^{n-j}b_j(\pi_M(x))
		\end{equation}
		and
		\begin{equation}
		\widehat{\mathcal{R}}_k(x):=\sum_{m=-\infty}^{-1}\chi\left(\frac{m}{k}\right)B_m(\pi_M(x))+\sum_{m=0}^{+\infty}\chi\left(\frac{m}{k}\right)\varepsilon_m(x)+\sum_{m=1}^{+\infty}\chi\left(\frac{m}{k}\right)m^{-1}b_{n+1}(\pi_M(x)),
		\end{equation}
		 our theroem follows.
	\end{proof}	
		\begin{remark}
			When $0\not\in\supp\chi$, $(\Pi\circ\chi_k(T_\vartheta))(x,x)=\chi_k(T_\vartheta)(x,x)$ still has the full asymptotic expansion because here Lemma \ref{lem:higher order Riemann sum} 
works for $\sum_{m=1}^{+\infty}k^{-1}\chi\left(\frac{m}{k}\right)\left(\frac{m}{k}\right)^{-1}$.  We meanwhile notice that for Theorems \ref{thm:Trace on circle bundle} and \ref{thm:main thm is sharp}, when $\chi\in\cCc(\R_+)$, by taking the semi-classical limit $k\to+\infty$ we certainly have $a_j(x)=\widehat{b}_j(x)$, $j=1,\ldots,n$. But actually this relation can also be deduced directly from Lemma \ref{lem:higher order Riemann sum}. For simplicity we only exhibit 
			\begin{equation} a_1(x)=\widehat{b}_1(x)=\frac{1}{2\pi}\left(\int_0^{+\infty}\chi(t)t^{n-1}dt\right)b_1(\pi_M(x)),    
			\end{equation}
			and the rest situation follows similarly. From the proof of Lemma \ref{lem:higher order Riemann sum}, we can check that
			\begin{equation}
2\pi\widehat{b}_1(x)=\left(\int_0^{+\infty}\chi(t)t^{n-1}dt\right)
b_1(\pi_M(x))-\left(\frac{1}{2}\int_0^{+\infty}\frac{d}{dt}(\chi(t)t^n)dt\right)b_0(\pi_M(x)),
\end{equation}
and by using integration by parts we always have 
$\int_0^{+\infty}\frac{d}{dt}(\chi(t)t^n)dt=0$ for any $\chi\in\cCc(\R)$. This observation also suggests that the general proof of Theorem \ref{thm:main thm is sharp} with some minor change works for Theorem \ref{thm:Trace on circle bundle}.
\end{remark}

Finally, we give an explicit calculation for the spectrum of 
a Toeplitz operator on a Grauert tube over a torus. 
Consider the holomorphic action of \(\R^n\) on 
\(\C^n\), \(n\geq 2\) given by \(t\circ z=z+t.\) 
Since the action by the subgroup \(2\pi\Z^n\subset \R^n\) 
is free, discrete and properly discontionuously, 
\(M:=\C^n/2\pi\Z^n\) is a complex manifold of dimension 
	\(\dim_\C M=n\). 
	
	Define a function \(\rho\colon M\to\R\) by
	\(\rho([z])=y_1^2+\ldots+y_n^2\) where \(z=x+iy\) 
	with \(x=(x_1,\ldots,x_n),y=(y_1,\ldots,y_n)\in\R^n.\)
	Then \(\rho\in \mathscr{C}^\infty(M,\R)\) is a 
	strictly plurisubharmonic function and any non-zero value 
	is a regular value for \(\rho\).
	Hence, for some fixed \(\varepsilon>0\) we have that 
	\(X:=\rho^{-1}(\{\varepsilon^2\})\) is a strictly pseudoconvex 
	CR manifold of codimension one with CR strucutre given by 
	\(T^{1,0}X=T^{1,0}M\cap\C TX\). 
	Furhtermore, \(X\) is compact (it is actually homeomorphic to the Cartesian product 
	of an \(n\)-dimensional torus and an \((n-1)\)-dimensional sphere)
	and \(\overline{\partial}_b\) has closed range in $L^2$ 
	on $X$ for all \(n\geq2\) by \cite{Bou75}.
	
	We consider the metric and volume form on \(X\) 
	induced by the standard ones on \(M\) and \(\C^n\), respectively. 
	Then \(T=\sum_{j=1}^ny_j\frac{\partial}{\partial x_j}\) 
	defines a transversal vector field on \(X\), which is not CR, 
	such that \(T\perp T^{1,0}X\oplus T^{0,1}X\) and \(|T|=\varepsilon\). 
	Moreover, \(iT\colon \mathscr{C}^\infty(X)\to \mathscr{C}^\infty(X)\)
	is symmetric. We put \(A:=\Pi \frac{i}{\varepsilon}T\Pi\) where 
	\(\Pi\colon L^2(X)\to L^2(X)\) denotes the Szeg\H{o} projection.
	We note that \(X\) is actually the boundary of the Grauert tube of 
	the \(n\)-torus with flat Riemannian metric. 
	In that context such an example  was studied in \cite{CRa22,Leb18}. 
	We will compute the spectrum of \(A\) 
	in a way that leads to an explicit expression for the eigenvalues 
	which is directly linked to Bessel functions when \(n\) is even. 
	
	Define \(\gamma_n\colon \R\to\R\),
	\(\gamma_n(t)=\int_{0}^{\pi}e^{t\cos(\alpha)}\sin^{n-2}(\alpha)d\alpha\). Then \(\gamma_n\) 
	satisfies the equation \(\gamma_{n+2}(t)=(n-1)t^{-1}\gamma'_{n}(t)\) for all \(n\geq2\). 
	
	If \(n\geq2\) is even we have 
	\begin{equation}
	\gamma_n(t)=\frac{\pi(n-2)!}{2^{n-1}(\frac{n}{2}-1)!}\sum_{j=0}^{+\infty}\frac{\left(\frac{t}{2}\right)^{2j}}{j!(j+\frac{n-2}{2})!}=\frac{\pi(n-2)!}{2^{n-1}t^{\frac{n-2}{2}}}J_{\frac{n-2}{2}}(it),
	\end{equation}
	where \(J_\alpha\), \(\alpha\geq0\), is the Bessel function of first kind. If \(n\) is odd  we find
	\begin{equation}
	\gamma_n(t)=2^{n-1}\left(\frac{n-1}{2}-1\right)!\sum_{j=0}^{+\infty}\frac{(\frac{n-1}{2}+j)!}{j!(n-1+2j)!}t^{2j}.
	\end{equation}
	For example \(\gamma_3(t)=2t^{-1}\sinh(t)\).  
	For \(m\in\Z^n\), we put \(s_{m}\colon X\to \C\), 
	\begin{equation}
	s_{m}([z])=\frac{e^{i\langle m,z\rangle}}{(2\pi)^{n/2}\sqrt{\varepsilon^{n-1}\gamma_n(2\varepsilon|m|)\operatorname{vol}(S^{n-2})}}.
	\end{equation}
	It follows that \(s_{m}\in H_b^0(X)\cap \mathscr{C}^\infty(X)\subset L^2(X)\) holds for all \(m\in\Z^n\).
	\begin{theorem}\label{thm:Spectrum}
		With the notation above we have that \(\{s_{m}\}_{m\in\Z^n}\) is an orthonormal basis for \(H_b^0(X)\) with
		\begin{equation}
		As_{m}=|m|\frac{\gamma'_{n}(2\varepsilon|m|)}{\gamma_n(2\varepsilon|m|)}s_{m},~\text{for all}~ m\in\Z^n.
		\end{equation}
	\end{theorem}
	Accordingly, the spectrum of \(A\) is given by 
	\begin{equation}
	\operatorname{Spec}(A)=\left\{|m|\frac{\gamma'_{n}(2\varepsilon|m|)}{\gamma_n(2\varepsilon|m|)}\colon m\in\Z^n\right\}.
	\end{equation}
	For \(k>0\) let \(\phi_k\) be the measure on \(\R_+\) defined by
	\begin{equation}
	\phi_k(t)=k^{-n}\sum_{m\in\Z^m}\delta\left(t-k^{-1}|m|\frac{\gamma'_{n}(2\varepsilon|m|)}{\gamma_n(2\varepsilon|m|)}\right),
	\end{equation}
	where \(\delta(t)\) denote the Dirac measure at zero. 
	Applying Theorem \ref{thm:ScaledSpectralMeasures} we obtain the following.
	\begin{theorem}
		For any \(\chi\in \mathscr{C}_c^\infty\left(\R_+\right)\) we have
		\begin{equation}
		\lim_{k\to+\infty} \int_0^{+\infty}\phi_k(t)\chi(t)dt=\operatorname{vol}(S^{n-1})\int_{\R}t^{n-1}\chi(t)dt.
		\end{equation}
	\end{theorem}
	Before we prove Lemma~\ref{thm:Spectrum} below we need some preparation.
	For \(m\in\Z^n\) put \(\widetilde{s}_{m}([z])=e^{i\langle m,z\rangle}\). We have that \(\widetilde{s}_{m}\) is a smooth CR function as the restriction of a holomorphic function defined on \(M\). We denote by \(\mathbb{T}^n=\R^n/2{\pi}\Z^n\) the real \(n\)-dimensional torus and by \(S^{n-1}_r=\{x\in\R^n\mid |x|=r\}\) the \((n-1)\)-dimensional sphere of radius \(r>0\).
	\begin{lemma}\label{lem:gmma}
		We have \(\int_{S^{n-1}_1}e^{ty_1}dS_{1}^{n-1}(y)=\gamma_n(t)\operatorname{vol}(S^{n-2}_1)\) for all \(t\in \R\).
	\end{lemma}
	\begin{proof}
		Consider the map \(F\colon[0,\pi]\times S_1^{n-2}\to S^{n-1}_1\), 
		\begin{equation}
		F(\alpha,x_1,\ldots,x_{n-1})=(\cos(\alpha),\sin(\alpha)x_1,\ldots,\sin(\alpha)x_{n-1}).
		\end{equation}
		Since
		\begin{equation}
		\begin{split}
		dS^{n-1}_1(y)
		&=\sum_{j=1}^n(-1)^{j+1}y_jdy_1\ldots\widehat{dy_j}\ldots dy_n\\
		&=y_1dy_2\wedge\ldots\wedge dy_n+dy_1\wedge
		\sum_{j=1}^{n-1}(-1)^jy_{j+1}dy_2\wedge\ldots\wedge \widehat{dy_{j+1}}\wedge\ldots \wedge dy_n,
		\end{split}
		\end{equation} 
		we find 
		\begin{equation}
		\begin{split}
		&(F^* dS_{1}^{n-1})(\alpha,x)\\
		=&\cos^2(\alpha)\sin^{n-2}(\alpha)d\alpha\wedge dS^{n-2}_1(x) +
		\cos(\alpha)\sin^{n-1}(\alpha)dx_1\ldots dx_{n-1}\\
		+&\sin^n(\alpha)d\alpha\wedge dS^{n-2}_1(x).
		\end{split}
		\end{equation}
		Hence by change of variables and Fubini theorem we conclude
		\begin{equation}
		\int_{S^{n-1}_1}e^{ty_1}dS_{1}^{n-1}(y)=\int_{0}^{ \pi}e^{t\cos(\alpha)}\sin^{n-2}(\alpha)d\alpha\cdot \int_{S^{n-2}_1}dS_{1}^{n-2}=\gamma_n(t)\operatorname{vol}(S^{n-2}_1).
		\end{equation}
	\end{proof}
	\begin{lemma}\label{lem:orthogonal}
		Given \(m,m'\in\Z^n\) we have
		\begin{equation}
		(\widetilde{s}_{m},\widetilde{s}_{m'})=\begin{cases}
		(2\pi)^n\varepsilon^{n-1} \gamma_n(2\varepsilon|m|)\operatorname{vol}(S^{n-2}_1)&, \text{ if } m=m'\\
		0&, \text{ else.}
		\end{cases}
		\end{equation}
		and
		\begin{equation}
		(\widetilde{s}_{m},iT\widetilde{s}_{m'})=\begin{cases}
		(2\pi)^n\varepsilon^{n}|m| \gamma'_n(2\varepsilon|m|)\operatorname{vol}(S^{n-2}_1)&, \text{ if } m=m',\\
		0&, \text{ else.}
		\end{cases}
		\end{equation}
	\end{lemma}
	\begin{proof}
		We have
		\begin{equation}
		\int_{X_\varepsilon}\widetilde{s}_{m}\overline{\widetilde{s}_{m'}}dV_{X_\varepsilon}=\int_{x\in\mathbb{T}^n}e^{i\langle m-m',x\rangle}d\mathbb{T}^n(x)\int_{y\in S^{n-1}_\varepsilon}e^{-\langle m+m',y\rangle}dS^{n-1}_\varepsilon(y).
		\end{equation}
		and
		\begin{equation}
		\int_{X_\varepsilon}\widetilde{s}_{m}\overline{iT\widetilde{s}_{m'}}dV_{X_\varepsilon}=-\int_{x\in\mathbb{T}^n}e^{i\langle m-m',x\rangle}d\mathbb{T}^n(x)\int_{y\in S^{n-1}_\varepsilon}\langle m',y\rangle e^{-\langle m+m',y\rangle}dS^{n-1}_\varepsilon(y).
		\end{equation}
		Since
		\begin{equation}
		\int_{x\in\mathbb{T}^n}e^{i\langle m-m',x\rangle}d\mathbb{T}^n(x)=\begin{cases}
		(2\pi)^n&, \text{ if } m=m',\\
		0&, \text{ else,}
		\end{cases}
		\end{equation}
		we need to calculate \(\int_{S_{\varepsilon}^{n-1}}e^{-2\langle m,y\rangle}dS^{n-1}_{\varepsilon}(y)\) and  \(\int_{S_{\varepsilon}^{n-1}}\langle m,y\rangle e^{-2\langle m,y\rangle}dS^{n-1}_{\varepsilon}(y)\).
		By scaling and the invariance of \(dS_{1}^{n-1}\) under orthonormal transformation we find
		\begin{equation}
		\int_{S^{n-1}_\varepsilon}e^{-2\langle m,y\rangle}dS_{\varepsilon}^{n-1}(y)=\varepsilon^{n-1}\int_{S^{n-1}_1}e^{2|m|\varepsilon y_1}dS_{1}^{n-1}(y),
		\end{equation}
		and the first part of the claim follows from Lemma~\ref{lem:gmma}. Furthermore, using the same arguments we observe
		\begin{equation}
		\begin{split}
		\varepsilon^n|m|\operatorname{vol}(S^{n-2}_1)\gamma'_n(2\varepsilon|m|)=&\varepsilon^{n-1}\int_{S^{n-1}_1}\varepsilon|m|y_1e^{2\varepsilon |m|y_1}dS_{1}^{n-1}(y)\\
		=&\int_{S^{n-1}_\varepsilon}-\langle m,y\rangle e^{-2\langle m,y\rangle}dS_{\varepsilon}^{n-1}(y),
		\end{split}
		\end{equation}
		which proves the second part of the statement.
	\end{proof}  
	
	From the first part of Lemma~\ref{lem:orthogonal} it follows
	that \(\{s_{m}\}_{m\in\Z^n}\subset H^0_b(X)\) is an orthonormal system.  
	We will prove now that it is a basis.
	\begin{lemma}\label{lem:completness}
		Let \(v\in H_b^0(X)\) be a CR function. 
		There exist \(a_m\in\C\), \(m\in\Z^n\), such that \(v=\sum_{m\in\Z^n}a_ms_m\) where the sum converges in \(L^2\)-norm.
	\end{lemma}
	We note that Lemma~\ref{lem:completness} can be deduced from results about the holomorphic extension 
	of analytic functions on general Grauert tubes in \cite{Leb18}. For the sake of completness 
	we give an elementary proof for our specific case.
	\begin{proof}[Proof of Lemma~\ref{lem:completness}]
		Since smooth CR functions are dense in \(H_b^0(X)\) it is enough to prove the statement 
		when \(v\) is a smooth CR function. Note that we have a Lie group action on \(X\) 
		by the torus \(\mathbb{T}^n\)  given by \([t]\circ [z]=[z+t]\).
		Hence by  Fourier transform we can write \(v\) as the \(L^2\) converging sum 
		\(v=\sum_{m\in\Z^n}v_m\) where \(v_m\), \(m\in\Z^n\), are smooth functions given by
		\begin{equation}
		v_m([z])=(2\pi)^{-n}\int_{[t]\in\mathbb{T}^n}v([z+t])e^{-i\langle m,t\rangle}d\mathbb{T}^n([t]).
		\end{equation}
		Fix \(m\in\Z^n\). The \(\mathbb{T}^n\)-action is CR as the restriction of a holomorphic action on \(M\). It follows that \(v_m\) is a smooth CR function since \(v\) is smooth and CR. Furthermore, we can diffeomorphically identify \(\mathbb{T}^n\times S^{n-1}_\varepsilon\) and \(X\) via the mapping \(([x],y)\mapsto [x+iy]\) which we will denote by \(F\). It follows that \(v_m\) can be written as \(v_m=f_ms_m\) for some smooth function \(f_m\colon X \to\C\) with \(\frac{\partial f_m}{\partial x_j}=0\), \(1\leq j\leq n\). We will show that \(f_m\) is constant. Therefore, take a point \([x+iy]\in X\) and consider the decomposition \( T_{[x+iy]}X=W_1\oplus W_2\) where \(W_1=dF(T_{[x]}\mathbb{T}^n)\) and \(W_2=dF(T_yS^{n-1}_\varepsilon)\). We have \(\mathcal{X}(f_m)=0\) for all \(\mathcal{X}\in W_1\). So we need to show  \(Y(f_m)=0\) for all \(Y\in W_2\). We note that \(J(W_2)\subset W_1\) with \(T\perp J(W_2)\) where \(J\) is the defining endomorphism for the CR structure. It follows that for \(Y\in W_2\)  we have \(Z:=Y-iJY \in T^{1,0}X \). 
		Since \(v_m\) and \(s_m\) are smooth CR functions we find at \([x+iy]\) that
		\begin{equation}
		0=\overline{Z}(v_m)=\overline{Z}(f_m)s_m=Y(f_m)\widetilde{s}_m
		\end{equation}
		holds.
		We have that \(\widetilde{s}_m\) is nowhere zero and hence \(V(f_m)=0\) for all \(V\in T_{[x+iy]}X\). 
		We conclude that \(v=\sum_{m\in\Z^n}a_ms_m\) for some numbers \(\{a_m\}_{m\in\Z^n}\subset \C\).
	\end{proof}
	\begin{proof}[Proof of Theorem \ref{thm:Spectrum}]
		From Lemma~\ref{lem:completness} and the first part of Lemma~\ref{lem:orthogonal} it follows that the set \(\{s_m\}_{m\in\Z^n}\) forms an orthonormal basis of \(H_b^0(X)\). Hence  we find \(A f=\sum_{m\in\Z^n}(iT f,s_m)s_m\)  for any smooth CR function \(f\).  Choosing \(f=s_m\), \(m\in\Z^n\), the claim follows from the second part of Lemma~\ref{lem:orthogonal}.
	\end{proof}
	\section{Semi-classical CR embeddings and their applications}
	\label{sec:Semi-classical CR embeddings and their applications}
	\subsection{CR Kodaira embedding}\label{sec:KodairaEmbedding}
	We are going to prove Theorem~\ref{thm:embedding}. 
	Let \((X,T^{1,0}X)\) be a compact orientable strictly pseudoconvex CR manifold as in Theorem~\ref{thm:ExpansionMain}. Denote by \(\xi\) a contact form and by \(\mathcal{T}\) the respective Reeb vector field. Let \(T_P\) ba a Toeplitz operator as in Theorem~\ref{thm:ExpansionMain}. We denote by \(0<\lambda_1\leq\lambda_2\leq\ldots\)  its positive eigenvalues counting multiplicities and let \(f_1,f_2\ldots\) be a respective orthonormal system of eigenfunctions in \(H^0_b(X)\). From now on let \(\chi\in \mathscr{C}_c^\infty((\delta_1,\delta_2))\), \(\chi\not\equiv 0\), be a cut-off function for some \(0<\delta_1<\delta_2\) and we put \(\eta(t):=|\chi|^2(t)\), \(t\in \R\). In the situation and notation in Theorem~\ref{thm:ExpansionMain}, we notice that the local asymptotic expansion
	\begin{equation}
	\eta_k(T_P)(x,y)=\int_0^{+\infty} e^{ikt\varphi(x,y)}A(x,y,t,k)dt+O\left(k^{-\infty}\right)
	\end{equation}
	holds and 
	\begin{equation}
	{A}_{0}(x,x,t)=\frac{1}{2\pi ^{n+1}}
	\frac{dV_{\xi}}{dV}(x)\,|\chi(\sigma_P(\xi_x)t)|^2t^n
	\end{equation}
	satisfies $\int_0^{+\infty} A_0(x,x,t)dt\in\cC^\infty(D,]0,\infty[)$.
	
	Setting \(N_k=\#\{j\in\N:0<\lambda_j\leq \delta_2k\}\), we define 
	\begin{equation}
	G_k\colon X\to \C^{N_k},\quad    G_k(x)=\left(\chi(\lambda_1 k^{-1})f_1,\ldots,\chi(\lambda_{N_k} k^{-1})f_{N_k}\right).
	\end{equation}
	We will to show that \(G_k\) is a CR embedding for all sufficiently large \(k\).
	\begin{lemma}\label{lem:Fkgeq0}
		There exists \(k_0>0\) such that \(G_k(X)\subset \C^{N_k}\setminus\{0\}\)  for all \(k\geq k_0\).
	\end{lemma}
	\begin{proof}
		We have \(|G_k(x)|^2=\eta_k(T_P)(x,x)\). From Theorem~\ref{thm:ExpansionMain} using \(\eta\geq 0\), \(\eta\not\equiv 0\), we find that there exists \(C,k_0>0\) such that \(|G_k|^2\geq C\) for all \(k\geq k_0\). 
	\end{proof}
	\begin{lemma}\label{lem:Fkimmersion}
		There exists \(k_0>0\) such that \(G_k\) is an immersion for all \(k\geq k_0\).
	\end{lemma}
	\begin{proof}
		Put \(H=\{H_x\}_{x\in X}\), \(H_x:=\text{Re}( T^{1,0}_xX)\), and  
		define an isomorphism \(L_k\colon   TX \to   TX\) by \(L_k(\mathcal{T})=\frac{1}{k}\mathcal{T}\) and \(L_k(V)=\frac{1}{\sqrt{k}}V\) for \(V\in  H\). Denote by \(h\) the  positive symmetric bilinearform on \(H\) induced by the Levi-form. Extend \(h\) to a symmetric bilinearform on \(TX\) by \(h(\mathcal{T},\cdot)\equiv 0\). Moreover, let \(\langle \cdot\mid\cdot\rangle_X\) be a smooth metric on \(\C TX\) as in Section~\ref{sec:CRmanifoldsMicroLocal}. We will show that there exists positive smooth  functions \(a,b\colon X\to \R_+\) such that
		\begin{equation}
		k^{-n-1}L_k^*(G_k^* g_{\text{eucl}})=a\xi\otimes\xi +bh+O(k^{-\frac{1}{2}})    
		\end{equation}
		where \(g_\text{eucl}\) denotes the Riemannian metric on \(\C^{N_k}\) induced by its standard Hermitian metric \(\langle \cdot,\cdot\rangle\).
		More precisely, there exists a constant \(C>0\) such that
		\begin{equation}\label{eq:PullbackEuclideanCN}
		\left|k^{-n-1}L_k^*(F^*_kg_{\text{eucl}})(V,W) - a\xi(V)\xi(W)-bh(V,W)\right|\leq \frac{C}{\sqrt{k}}|V|_X|W|_X
		\end{equation}
		for all \(k\geq 1\), \(V,W \in T_pX\), \(p\in X\). 
		From this statement it follows that the pullback of the 
		Riemannian metric on \(\C^{N_k}\) under \(G_k\) becomes 
		a Riemannian metric on \(X\) when \(k\) is large enough. 
		Hence \(G_k\) is an immersion for all sufficiently large \(k\).
		
		Let us show that~\eqref{eq:PullbackEuclideanCN} holds.
		With Theorem~\ref{thm:ExpansionMain} and 
		Lemma~\ref{lem:firstDifferentialMainThm} we have 
		\begin{equation}
		\begin{split}
		&k^{-n-1}L_k^*(G_k^*g_{\text{eucl}})(\mathcal{T},\mathcal{T})\\
		&=k^{-n-2}\Re\left( (d_x\otimes d_y\eta_k(T_P)(x,y)|_{x=y})(\mathcal{T}\otimes \mathcal{T})\right)\\
		&=a+O(k^{-1})
		\end{split}
		\end{equation}
		for some positive smooth function \(a\colon X\to\R_+\).  
		Hence there exists \(C_1>0\) with 
		\begin{equation}
		|k^{-n-1}L_k^*(G_k^* g_\text{eucl})(\mathcal{T},\mathcal{T})-a|\leq
		C_1k^{-1},\:\:\text{for all \(k\geq 1\).}
		\end{equation}
		
		From Lemma~\ref{lem:firstDifferentialMainThm} we also have for \(V\in H_x\) that 
		\begin{equation}
		k^{-n-1}(d_x\otimes d_y\eta_k(T_P)(x,y)|_{x=y})(\mathcal{T}\otimes V)=O(k).
		\end{equation}
		We conclude that there exist \(C_2>0\) such that
		\begin{equation}
		|k^{-n-1}L_k^*(G_k^*g_\text{eucl})(\mathcal{T},V)|\leq C_2k^{-\frac{1}{2}}|V|_X
		\end{equation}
		holds for all \(V\in H\) and all \(k\geq 1\). 
		We define a symmetric bilinear form on 
		\(H\) by \(R_k=\{(R_k)_x\}_{x\in X}\), \((R_k)_x\colon H_x\times H_x\to \C \),
		\begin{equation}
		(R_k)_x(V_1,V_2):=k^{-n-1}G_k^*g_{\text{eucl}}(V_1,V_2)=
		k^{-n-1}\text{Re}\langle V_1G_k(x),V_2G_k(x)\rangle.
		\end{equation}
		From Theorem~\ref{thm:ExpansionMain} we find bilinear forms 
		\(r_j=\{(r_j)_x\}_{x\in X}\), \((r_j)_x\colon H_x\times H_x\to \R \), \(j=0,1\), 
		and a constant \(C_3>0\) such that 
		\begin{equation}
		|R_k(V_1,V_2)-k^{2}r_0(V_1,V_2)-k^1r_1(V_1,V_2)|\leq C_3|V_1|_X|V_2|_X
		\end{equation}
		holds for all \(k\geq 1\), \(x\in X\) and \(V_1,V_2\in H_x\). From Lemma~\ref{lem:firstDifferentialMainThm} we obtain \(r_0\equiv 0\). So we will calculate \(r_1\). Given \(x\in X\) and \(V_1,V_2\in H_x\) we can find smooth sections  \(Z_j\in\Gamma(U,T^{1,0}X)\) with \(2\text{Re}(Z_j)_x=V_j\), \(j=1,2\),  defined in an open neighborhood \(U\) around \(x\).
		At \(x\in X\), since \(G_k\) is a CR map, we find
		\begin{equation}
		\begin{split}
		\langle(\chi(k^{-1}\lambda_1)V_1f_1,\ldots,&
		\chi(k^{-1}\lambda_{N_k})V_1f_{N_k})),
		(\chi(k^{-1}\lambda_1)V_2 f_1,\ldots,
		\chi(k^{-1}\lambda_{N_k})V_2 f_{N_k}))\rangle\\
		&=:\langle V_1 G_k, V_2 G_k\rangle=\langle Z_1 G_k,Z_2 G_k\rangle\\
		&=\overline{Z_2}(Z_1\langle G_k, G_k\rangle)-\langle \overline{Z_2}(Z_1 G_k), G_k\rangle\\
		&=\overline{Z_2}(Z_1\langle G_k, G_k\rangle)+\langle [Z_1,\overline{Z_2}] G_k, G_k\rangle.
		\end{split}
		\end{equation}
		We have
		\begin{equation}
		k^{-n-1}\overline{Z_2}(Z_1 \eta_k(T_P)(x,x))=O(1)
		\end{equation}
		and
		\begin{multline}
		k^{-n-1}(d_x\eta_k(T_P)(x,y)|_{x=y})([Z_1,\overline{Z_2}])\\
		-i\left(\frac{k}{2\pi^{n+1}}
		\frac{dV_{\xi}}{\sigma_P(\xi)^{n+2}dV}
		\int_0^{+\infty} t^{n+1}\eta(t)dt\right)\xi([Z_1,\overline{Z}_2])=O(1).
		\end{multline}
		We find, with \(\xi([Z_1,\overline{Z}_2])=-2i\mathcal{L}(Z_1,Z_2)\), that \(r_1(V_1,V_2)=bh(V_1,V_2)\) for some positive smooth function \(b\colon X\to \R_+\). Hence we conclude that there exist a constant \(C_4>0\) such that  
		\begin{equation}
		|k^{-n-1}L_k^*(G_k^*g_\text{eucl})(V_1,V_2)-bh(V_1,V_2)|\leq C_4k^{-1}|V_1|_X|V_2|_X    
		\end{equation}
		holds for all \(k\geq 1\) and all \(V_1,V_2\in H_x\), \(x\in X\). The statement follows.
	\end{proof}
	We are going to show now that \(G_k\) is injective for all sufficiently large \(k\). Given a point \(p\in X\) we fix coordinates \((x,y)\) around \((p,p)\in X\times X\) centered in the origin as in Theorem~\ref{thm:tangential hessian of varphi}.
	We write $x'=(x_1,\ldots,x_{2n}),~z=(z_1,\ldots,z_n)$, $y'=(y_1,\ldots,y_{2n}),~w=(w_1,\ldots,w_n)$, where $z_j=x_{2j-1}+ix_{2j},~w_j=y_{2j-1}+y_{2j}$. By Theorem~\ref{thm:tangential hessian of varphi}, for any phase function \(\varphi\in {\Ph}(\Pi,\xi,D)\), we have
	\begin{equation}
	\label{Eq:TangentialHessianVarPsi}
	\begin{split}
	\varphi(x,y)
	=&x_{2n+1}-y_{2n+1}+i\sum_{j=1}^n\mu_j|z_j-w_j|^2+\sum_{j=1}^ni\mu_j(\overline{z}_jw_j-z_j\overline{w}_j)\\
	+&\sum_{j=1}^n \left(d_j(z_jx_{2n+1}-w_jy_{2n+1})+\overline{d}_j(\overline{z}_jx_{2n+1}-\overline{w}_jy_{2n+1})\right)\\
	+&(x_{2n+1}-y_{2n+1})f_\varphi(x,y)+O\left(|(x,y)|^3\right).
	\end{split}
	\end{equation}
	Furthermore, for this coordinates there exists a constant \(C>0\) with \(\Im\varphi(x,y)\geq C |z-w|^2\).
	\begin{lemma}\label{lem:Fkinjective}
		There exists \(k_0>0\) such that \(\widetilde{G}_k\colon X\to  \C^{N_k}\setminus\{0\}\), \(\widetilde{G}_k(x)=\frac{G_k(x)}{|G_k(x)|}\) is well-defined and injective  for all \(k\geq k_0\).
	\end{lemma}
	\begin{proof}
		The fact that \(\widetilde{G_k}\) is well-defined for all sufficiently large \(k>0\) follows immediately from Lemma~\ref{lem:Fkgeq0}. 
We prove the injectivity statement by contradiction following 
the methods in \cite[Theorem 4.8]{HLM21}. 
Suppose we can find sequences of integers $\{k_j\}_{j=1}^\infty$,
sqeunces of points $\{x_{k_j}\}_{j=1}^\infty, 
\{y_{k_j}\}_{j=1}^\infty$ in $X$ with
$\lim_{j\to+\infty}k_j=+\infty$ such that $x_{k_j}\neq y_{k_j}$
and $\widetilde{G}_{k_j}(x_{k_j})=\widetilde{G}_{k_j}(y_{k_j})$
for each $j\geq 1$. For simplicity, we just use the index 
$k$ instead of $k_j$. Because $X$ is compact, we may assume 
by passing to a subsequence that $x_k\to p\in X$ and 
$y_k\to q\in X$ as $k\to+\infty$.
Since \(\langle G_k(x), G_k(y)\rangle=
\eta_k(T_P)(x,y)\) and \(|G_k(x)|=\sqrt{\eta_k(T_P)(x,x)}\), here we have 
\begin{eqnarray}\label{eq:NotInjectiveAssumption}
\frac{\eta_k(T_P)(x_k,y_k)}{\sqrt{\eta_k(T_P)(x_k,x_k)}
\sqrt{\eta_k(T_P)(y_k,y_k)}}=\frac{|G_k(x_k)|^2}{|G_k(x_k)||G_k(x_k)|}=1
\end{eqnarray}
for all sufficiently large \(k>0\). 
From~\eqref{Eq:FarAwayDiagonalMainThm} in 
Theorem~\ref{thm:ExpansionMain} 
and the above equation we conclude that \(p=q\) must hold.
We consider several cases for the convergent speed of 
$x_k$ and $y_k$ to $p$. 
Denote $x_k=(z^k,x^k_{2n+1})$ and $y_k=(w^k,y^k_{2n+1})$ in local coordinates around \(p\) as above when \(k>0\) is large enough. We identify $p$ as $0\in\R^{2n+1}$, and we let \( \mathscr{M}\geq 0\) be a constant with
		\begin{equation}
		k|z^k-w^k|^2\geq\mathscr{M}
		\end{equation}
		for infinitely many sufficiently large \(k>0\).\\
		Case I: \(\mathscr{M}>0\).  
		We find
		by (\ref{eq:asymptotic expansion of chi_k(T_P)}) in Theorem \ref{thm:ExpansionMain} and Fatou lemma that for some subsequence we have
		\begin{equation}
		\begin{split}
		\limsup_{k\to+\infty}k^{-n-1}|\eta_k(T_P)(x_k,y_k)|
		&\leq\limsup_{k\to+\infty}k^{-n-1}\int e^{-kt\operatorname{Im}
			\varphi(x_k,y_k)}|{A}(x_k,y_k,t,k)|dt\\
		&\leq\limsup_{k\to+\infty}k^{-n-1}
		\int e^{-Ck|z^k-w^k|^2t}|{A}(x_k,y_k,t,k)|dt\\
		&=\int e^{-C\mathscr{M}t}{A}_0(0,0,t)dt<\int {A}_0(0,0,t)dt.
		\end{split}
		\end{equation}
		However, we also have 
		\begin{equation}\label{eq:limitOnDiagonal}
		\lim_{k\to+\infty}k^{-n-1}\eta_k(T_P)(x_k,x_k)=
		\lim_{k\to+\infty}k^{-n-1}\eta_k(T_P)(y_k,y_k)=\int {A}_0(0,0,t)dt>0,
		\end{equation}
		and hence
		\begin{equation}
		\limsup_{k\to+\infty}\frac{|\eta_k(T_P)(x_k,y_k)|}
		{\sqrt{\eta_k(T_P)(x_k,x_k)}\sqrt{\eta_k(T_P)(y_k,y_k)}}<1
		\end{equation}
		which contradicts (\ref{eq:NotInjectiveAssumption}). So, we must have 
		\begin{equation}
		\lim_{k\to+\infty}k|z^k-w^k|^2=0.
		\end{equation}
		Now, take a subsequence such that
		\begin{equation}
		\lim_{k\to+\infty} k|\langle \partial_x\varphi(y_k,y_k),x_k-y_k\rangle|
		=\mathscr{N}
		\end{equation}
		for some \(\mathscr{N}\in [0,+\infty]\). \\
		Case II: \(\lim_{k\to+\infty}k|z^k-w^k|^2=0\) and $\mathscr{N}\in (0,+\infty]$. 
		From Taylor's formula, (\ref{Eq:PhaseFuncMainThm}) and \(\partial_x\Im\varphi(x,x)=0\) we have
		\begin{equation}
		\begin{split}
		|\varphi(x_k,y_k)-\langle\partial_x\varphi(y_k,y_k),x_k-y_k\rangle|&\leq C_1 |x_k-y_k|^2\\
		|\Re\varphi(x_k,y_k)-\langle\partial_x\varphi(y_k,y_k),x_k-y_k\rangle|&\leq C_2 |x_k-y_k|^2
		\end{split}
		\end{equation}
		and
		\begin{equation}
		|x^k_{2n+1}-y^k_{2n+1}|\leq C_3(|\langle\partial_x\varphi(y_k,y_k),x_k-y_k\rangle|+|z^k-w^k|)
		\end{equation}	
		for some constants \(C_1,C_2,C_3>0\) independent of \(k>0\).
		Since \(\lim_{k\to+\infty}\langle\partial_x\varphi(y_k,y_k),x_k-y_k\rangle=0\)
		we find
		\begin{equation}
		\label{eq:varphigeqleqMN}
		\begin{split}
		\left|\varphi(x_k,y_k)\right|&\geq\left|\Re\varphi(x_k,y_k)\right|\geq c(|\langle\partial_x\varphi(y_k,y_k),x_k-y_k\rangle|-|z_k-w_k|^2)\\
		\left|\varphi(x_k,y_k)\right|&\leq \frac{1}{c}(|\langle\partial_x\varphi(y_k,y_k),x_k-y_k\rangle|+|z_k-w_k|^2)
		\end{split}
		\end{equation}
		when \(k>0\) is large enough for some constant \(c>0\).
		Suppose that $\mathscr{N}=\infty$. 
		Since we have \(\lim_{k\to+\infty}k|z^k-w^k|^2=0\) it follows from~\eqref{eq:varphigeqleqMN} that \(\lim_{k\to+\infty}k\left|\varphi(x_k,y_k)\right|=\infty\).
		We obtain
		\begin{equation}
		\begin{split}
		&\limsup_{k\to+\infty}k^{-n-1}|\eta_k(T_P)(x_k,y_k)|\\
		=&\limsup_{k\to+\infty}k^{-n-1}\left|\int e^{ikt\varphi(x_k,y_k)}{A}(x_k,y_k,t,k)dt\right|\\
		\leq&\limsup_{k\to+\infty}k^{-n-1}\int \left|e^{ikt\varphi(x_k,y_k)}\frac{1}{ik\varphi(x_k,y_k)}\frac{d}{dt}{A}(x_k,y_k,t,k)\right|dt\\
		=&\limsup_{k\to+\infty}\int e^{-kt\operatorname{Im}\varphi(x_k,y_k)}\frac{1}{k|\varphi(x_k,y_k)|}\left|\frac{d}{dt}{A}(x_k,y_k,t,k)\right|k^{-n-1}dt\\
		\leq&\int\limsup_{k\to+\infty}\left( e^{-Ctk|z^k-w^k|^2}\frac{1}{k|\varphi(x_k,y_k)|}\left|\frac{d}{dt}{A}(x_k,y_k,t,k)\right|k^{-n-1}\right)dt\\
		=&0.
		\end{split}
		\end{equation} 
		With~\eqref{eq:limitOnDiagonal} we find
		\begin{equation}
		\limsup_{k\to+\infty}\frac{|\eta_k(T_P)(x_k,y_k)|}{\sqrt{\eta_k(T_P)(x_k,x_k)}\sqrt{\eta_k(T_P)(y_k,y_k)}}=0,
		\end{equation}
		which contradicts~\eqref{eq:NotInjectiveAssumption}.\\
		If $\mathscr{N}\in(0,+\infty)$, then by~\eqref{eq:varphigeqleqMN} we have that \(k|\varphi(x_k,y_k)|\) is bounded. Furthermore, we have that \(k|\Re\varphi(x_k,y_k)|\) is bounded from below by a positive constant when \(k\) is large and \(\Im\varphi(x_k,y_k)\geq 0\). Hence, by passing to a subsequence we can assume \(\lim_{k\to+\infty}k\varphi(x_k,y_k)=a+ib\) with \(a>0\) and \(b\geq 0\). 
		With \(A_0(0,0,t)\geq 0\) and \(A_0(0,0,\cdot)\not\equiv 0\) we conclude
		\begin{equation}
		\begin{split}
		&\lim_{k\to+\infty}k^{-n-1}|\eta_k(T_P)(x_k,y_k)|\\
		=&\lim_{k\to+\infty}\left|\int e^{itk\operatorname{Re}\varphi(x_k,y_k)}e^{-tk\operatorname{Im}\varphi(x_k,y_k)}k^{-n-1}{A}(x_k,y_k,t,k)dt\right|\\
		=&\left|\int e^{ita}e^{-tb}{A}_0(0,0,t)dt\right|
		<\left|\int e^{-tb}{A}_0(0,0,t)dt\right|\\
		\leq&\int {A}_0(0,0,t)dt.
		\end{split}
		\end{equation}
		With~\eqref{eq:limitOnDiagonal} it follows that
		\begin{equation}
		\limsup_{k\to+\infty}\frac{|\eta_k(T_P)(x_k,y_k)|}{\sqrt{\eta_k(T_P)(x_k,x_k)}\sqrt{\eta_k(T_P)(y_k,y_k)}}<1
		\end{equation}	
		which contradicts~\eqref{eq:NotInjectiveAssumption}.
		So we may write
		\begin{equation}
		\lim_{k\to+\infty}k|\langle \partial_x\varphi(y_k,y_k),x_k-y_k\rangle|=0.
		\end{equation}
		Case III: $\lim_{k\to+\infty}k|z^k-w_k|^2=0$ and $\lim_{k\to+\infty}k|\langle \partial_x\varphi(y_k,y_k),x_k-y_k\rangle|=0$. 
		Define
		\begin{equation}
		\begin{split}
		&\mathscr{F}_k(u):=|\eta_k(T_P)(ux_k+(1-u)y_k,y_k)|^2\\
		&\mathscr{G}_k(u):=\eta_k(T_P)(y_k,y_k)
		\eta_k(T_P)(ux_k+(1-u)y_k,ux_k+(1-u)y_k),\\
		&\mathscr{H}_k(u):=\frac{\mathscr{F}_k(u)}{\mathscr{G}_k(u)}.
		\end{split}
		\end{equation}
		From~\eqref{eq:NotInjectiveAssumption} and basic inequalities,
		\begin{equation}	
		0\leq\mathscr{H}_k(u)\leq 1.
		\end{equation}
		Moreover, because $\mathscr{H}_k(0)=\mathscr{H}_k(1)=1$, 
		we can find a $u_k\in (0,1)$ such that
		\begin{equation}\label{eq:Hukgeq0}
		\mathscr{H}_k''(u_k)\geq 0.
		\end{equation}
		We compute
		\begin{equation}
		\mathscr{H}_k''(u_k)=\frac{\mathscr{F}_k''(u_k)}
		{\mathscr{G}_k(u_k)}-2\frac{\mathscr{F}_k'(u_k)
			\mathscr{G}_k'(u_k)}{\mathscr{G}_k^2(u_k)}-
		\frac{\mathscr{F}_k(u_k)\mathscr{G}_k''(u_k)}
		{\mathscr{G}_k^2(u_k)}+
		2\frac{\mathscr{F}_k(u_k)\mathscr{G}_k'^2(u_k)}
		{\mathscr{G}_k^3(u_k)}.
		\end{equation}
		For simplicity we denote
		\begin{equation}
		\alpha_k(u):=\eta_k(T_P)(ux_k+(1-u)y_k,y_k).
		\end{equation}
		Then
		\begin{equation}
		\begin{split}
		&\mathscr{F}_k(u)=|\alpha_k(u)|^2,\\
		&\mathscr{F}_k'(u)=\alpha_k'(u)\overline{\alpha}_k(u)+
		\alpha_k(u)\overline{\alpha}_k'(u),\\
		&\mathscr{F}_k''(u)=\alpha_k''(u)\overline{\alpha}_k(u)+
		2|\alpha_k'(u)|^2+\alpha_k(u)\overline{\alpha}_k''(u),
		\end{split}
		\end{equation}
		and (\ref{eq:asymptotic expansion of chi_k(T_P)}) 
		in Theorem~\ref{thm:ExpansionMain} implies that
		\begin{equation}
		\alpha_k(u)=\int e^{ikt\varphi(ux_k+(1-u)y_k,y_k,t)}
		A(ux_k+(1-u)y_k,y_k,t,k)+\gamma_k(u), 
		\end{equation}
		for some $\gamma_k(u)=O\left(k^{-\infty}\right)$. 
		Again for simplicity, we denote
		\begin{equation}
		\beta_k(u):=\int e^{ikt\varphi(ux_k+(1-u)y_k,y_k,t)}
		{A}(ux_k+(1-u)y_k,y_k,t,k)dt.
		\end{equation}
		We have
		\begin{equation}
		\begin{split}
		\beta_k'(u)=&\int e^{ikt\varphi(ux_k+(1-u)y_k,y_k)}ikt\left\langle\frac{\partial\varphi}{\partial x}(ux_k+(1-u)y_k,y_k),(x_k-y_k)\right\rangle \\
		&{A}(ux_k+(1-u)y_k,y_k,t,k)dt\\
		+&\int e^{ikt\varphi(ux_k+(1-u)y_k,y_k)}\left\langle\frac{\partial {A}}{\partial x}(ux_k+(1-u)y_k,y_k),(x_k-y_k)\right\rangle dt,
		\end{split}
		\end{equation}
		\begin{equation}
		\begin{split}
		\beta_k''(u)=&\int  e^{ikt\varphi(ux_k+(1-u)y_k,y_k)}(ikt)^2\left\langle\frac{\partial\varphi}{\partial x}(ux_k+(1-u)y_k,y_k),(x_k-y_k)\right\rangle^2\\
		&{A}(ux_k+(1-u)y_k,y_k,t,k)dt
		\\
		+&\int e^{ikt\varphi(ux_k+(1-u)y_k,y_k,t)}ikt\sum_{j,l}\frac{\partial^2\varphi}{\partial x_j\partial x_l}(ux_k+(1-u)y_k,y_k)(x_k^j-y_k^j)(x_k^l-y_k^l)\\
		&{A}(ux_k+(1-u)y_k,y_k,t,k)dt\\
		+&\int e^{ikt\varphi(ux_k+(1-u)y_k,y_k,t)}ikt\left\langle\frac{\partial\varphi}{\partial x}(ux_k+(1-u)y_k,y_k),(x_k-y_k)\right\rangle\\
		&\frac{\partial {A}}{\partial x}(ux_k+(1-u)y_k,y_k,t,k)dt\\
		+&\int e^{ikt\varphi(ux_k+(1-u)y_k,y_k,t)}\sum_{k,l}\frac{\partial^2 {A}}{\partial x_j\partial x_l}(ux_k+(1-u)y_k,y_k,t,k)(x_k^j-y_k^j)(x_k^l-y_k^l).
		\end{split}
		\end{equation}
		With the computations above we find
		\begin{equation}
		\begin{split}
		\mathscr{F}_k''(u_k)
		=&2|\beta_k'(u_k)|^2+\beta_k''(u_k)\overline{\beta}_k(u_k)+\beta_k(u_k)\overline{\beta_k}''(u_k)\\
		+&2\beta_k'(u_k)\overline{\gamma}_k'(u_k)+2\gamma_k'(u_k)\overline{\beta}_k'(u_k)\\
		+&\beta_k''(u_k)\overline{\gamma}_k(u_k)+\gamma_k(u_k)\overline{\beta}_k''(u_k)+\overline{\beta}_k(u_k)\gamma_k''(u_k)+\beta_k(u_k)\overline{\gamma}_k''(u_k)\\
		+&\gamma_k''(u_k)\overline{\gamma}(u_k)+\gamma_k(u_k)\overline{\gamma}_k''(u_k)+2|\gamma_k'(u_k)|^2.
		\end{split}
		\end{equation}
		Directly, by Taylor formula,
		we can write
		\begin{equation}
		\begin{split}
		&2|\beta_k'(u_k)|^2+\beta_k''(u_k)\overline{\beta}_k(u_k)+\overline{\beta''_k(u_k)}\beta_k(u_k)\\
		=&2k^{2n+4}\left(\left|\int t{A}_0(0,0,t)dt\right|^2-\int t^2{A}_0(0,0,t)\int {A}_0(0,0,t)dt\right)|\langle \partial_x\varphi(y_k,y_k),x_k-y_k\rangle|^2\\
		&-2k^{2n+3}\int\sum_{j,l=1}^{2n}\frac{\partial^2\Im\varphi}{\partial x_j\partial x_l}(0,0)(x_k^j-y_k^j)(x_k^l-y_k^l){A}_0(0,0,t)dt\int {A}_0(0,0,t)dt\\
		+&o(k^{2n+2})O\left(\left(\sqrt{k}\sum_{j=1}^{2n}|x_k^j-y_k^j|+k\left|\left\langle\frac{\partial\varphi}{\partial x}(y_k,y_k),y_k-x_k\right\rangle\right|\right)^2\right).
		\end{split}
		\end{equation}
		Notice that from H\"older's inequality,
		\begin{equation}
		\left|\int t{A}_0(0,0,t)dt\right|^2<\int t^2{A}_0(0,0,t)dt\int {A}_0(0,0,t)dt,
		\end{equation}
		so we can find a constant $c_1>0$ such that
		\begin{equation}
		\begin{split}
		& \left(\left|\int t{A}_0(0,0,t)dt\right|^2-\int t^2{A}_0(0,0,t)\int {A}_0(0,0,t)dt\right)|\langle \partial_x\varphi(y_k,y_k),x_k-y_k\rangle|^2\\
		&<-c_1|\langle \partial_x\varphi(y_k,y_k),x_k-y_k\rangle|^2.
		\end{split}
		\end{equation}
		Also, we can find a constant $c_2>0$ such that
		\begin{equation}
		-\int\sum_{j,l=1}^{2n}\frac{\partial^2\Im\varphi}{\partial x_j\partial x_l}(0,0)(x_k^j-y_k^j)(x_k^l-y_k^l){A}_0(0,0,t)dt<-c_2\sum_{j=1}^{2n}|x_k^j-y_k^j|^2.
		\end{equation}
		Accordingly, we can find a constant $c_0>0$ such that
		\begin{equation}
		\limsup_{k\to+\infty}\left(\sqrt{k}\sum_{j=1}^{2n}|x_k^j-y_k^j|+k|\langle\partial_x\varphi(y_k,y_k),y_k-x_k\rangle|\right)^{-2}\frac{\mathscr{F}_k''(u_k)}{\mathscr{G}_k(u_k)}\leq -c_0<0.
		\end{equation}
		One can check that
		\begin{equation}
		\begin{split}
		&\limsup_{k\to+\infty}\left(\sqrt{k}\sum_{j=1}^{2n}|x_k^j-y_k^j|+k|\langle\partial_x\varphi(y_k,y_k),y_k-x_k\rangle|\right)^{-2}\\
		&\left(-2\frac{\mathscr{F}_k'(u_k)\mathscr{G}_k'(u_k)}{\mathscr{G}_k^2(u_k)}-\frac{\mathscr{F}_k(u_k)\mathscr{G}_k''(u_k)}{\mathscr{G}_k^2(u_k)}+2\frac{\mathscr{F}_k(u_k)\mathscr{G}_k'^2(u_k)}{\mathscr{G}_k^3(u_k)}\right)=0.
		\end{split}
		\end{equation}
		holds. Thus, we obtain
		\begin{equation}
		\limsup_{k\to+\infty}\left(\sqrt{k}
		\sum_{j=1}^{2n}|x_k^j-y_k^j|+
		k|\langle\partial_x\varphi(x_k,x_k),y_k-x_k\rangle|\right)^{-2}
		\mathscr{H}_k''(u_k)\leq -c_1<0.
		\end{equation}
		This is a contradiction to~\eqref{eq:Hukgeq0}.
	\end{proof}
	
	\begin{proof}[Proof of Theorem  \ref{thm:embedding}]
		From Lemma~\ref{lem:Fkimmersion} and Lemma~\ref{lem:Fkinjective} it follows that \(G_k\) is an injective immersion for all sufficiently large \(k\). Since \(X\) is compact it follows that there exists \(k_0>0\) such that \(G_k\) is an embedding for all \(k\geq k_0\). From the definition it follows that \(G_k\) is a CR map for any \(k>0\). Since \(X\) is a codimension one CR manifold we conclude that \(G_k\) is a CR embedding for all \(k\geq k_0\). 
	\end{proof}
	
	\subsection{$T_{\lambda(k)}$ and an almost equivariant embedding}\label{sec:AlmostEquivariant}
	We consider the same setting as in Section~\ref{sec:KodairaEmbedding} but we rescale the map \(F_k\) by some \(k\)-dependent factor. So from now on we put
	\begin{equation}\label{eq:FkSectionAlmostEquivariant}
	F_k(x):=\sqrt{\frac{2\pi^{n+1}}{k^{n+1}}}
	\left(\chi(k^{-1}\lambda_1)f_1(x),\ldots,
	\chi(k^{-1}\lambda_{N_k})f_{N_k}(x)\right)
	\end{equation}
	with \(\lambda_j,f_j,\chi, N_k\) as in Section~\ref{sec:KodairaEmbedding}. Recall the definition of \(\mathcal{T}_{\lambda(k)}\) in~\eqref{eq:DefinitionTbetaIntroduction} with \(\lambda(k)=(\lambda_1,\ldots,\lambda_{N_k})\). 
	\begin{proof}[Proof of Theorem~\ref{thm:vectorfield}]
		First, we notice that in this situation we have
		\begin{equation}
		\sigma_P(\xi)=\frac{dV_{\xi}}{dV}=1.
		\end{equation}
		By writing  \(|dF_k\mathcal{T}|^2=
		\frac{2\pi^{n+1}}{k^{n+1}}
		\sum_{j=1}^{N_k}|\chi(k^{-1}\lambda_j)|^2|
		\mathcal{T}f_j|^2\) and using the second relation in Lemma~\ref{lem:firstDifferentialMainThm}, 
		we obtain \(|k^{-1}dF_{k}\mathcal{T}|^2=C'_\chi+O(k^{-1})\) 
		in $\mathscr{C}^\infty$-topology. We also have
		\begin{equation}
		\frac{k^{n+1}}{2\pi^{n+1}}|\mathcal{T}_{\lambda(k)}\circ F_k|^2=
		\sum_{j=1}^{N_k}\lambda_j^2|\chi(k^{-1}\lambda_j)|^2|f_j|^2.
		\end{equation}
		Hence, \(|k^{-1}\mathcal{T}_{\lambda(k)}\circ F_k|^2=C'_\chi+O(k^{-1})\) in $\mathscr{C}^\infty$-topology by Lemma~\ref{lem:evpowersexpansion}. 
		Furthermore, 
		\begin{equation}
		\frac{k^{n+1}}{2\pi^{n+1}}\langle \mathcal{T}_{\lambda(k)}\circ F_k,dF_{k}\mathcal{T}\rangle=\sum_{j=1}^{N_k} i\lambda_jf_j\mathcal{T}\overline{f_j}=ik(\mathcal{T}_y\tau_k(T_P)(x,y))|_{y=x}.
		\end{equation}
		Hence, using Lemma~\ref{lem:firstDifferentialMainThm} again, we find in $\mathscr{C}^\infty$-topology that
		\begin{equation}
		\operatorname{Re}\langle \mathcal{T}_{\lambda(k)}\circ F_k,dF_{k}\mathcal{T}\rangle=k^2C'_\chi+O(k^{1}).
		\end{equation}
		The third relation in this theorem then follows from
		\begin{equation}
		|k^{-1}T_{\lambda(k)}\circ F_k-k^{-1}dF_k\mathcal{T}|^2=|k^{-1}\mathcal{T}_{\lambda(k)}|^2+|k^{-1}dF_k\mathcal{T}|^2-2k^{-2}\operatorname{Re}\langle \mathcal{T}_{\lambda(k)},dF_{k}\mathcal{T}\rangle).
		\end{equation}
	\end{proof}
	Since an equality \((F_k)_*\mathcal{T}=
	\mathcal{T}_{\lambda(k)}\circ F_k\) would imply that the map 
	\(F_k\) is equivariant, Theorem~\ref{thm:vectorfield} 
	could be seen as an almost equivariant embedding result.
	\subsection{$\omega_{\lambda(k)}$ and an almost spherical embedding}
	We consider the same setting as in Section~\ref{sec:AlmostEquivariant} 
	with \(F_k\) given by~\eqref{eq:FkSectionAlmostEquivariant}.
	Let \(\alpha\) be the smooth $1$-form on \(\C^N\) defined by
	\begin{equation}
	\label{eq:OneFormAlpha}
	\alpha=\frac{1}{2i}\sum_{j=1}^{N}(\overline{z}_j dz_j-z_jd\overline{z_j}),
	\end{equation}
	With \(\eta=|\chi|^2\) in Lemma~\ref{lem:firstDifferentialMainThm}, we obtain that in $\mathscr{C}^\infty$-topology
	\begin{equation}
	\frac{k^{n+1}}{2\pi^{n+1}}F_k^*\alpha=
	\frac{1}{2i}\sum_{j=1}^{N_k}|\chi(k^{-1}\lambda_j)|^2(\overline{f_j}df_j-f_jd\overline{f_j})=\operatorname{Im}(d_x\eta_k(T_P)(x,y)|_{y=x}), 
	\end{equation}
	and hence
	\begin{equation}\label{eq:pullbackOneForm}
	k^{-1}F_k^*\alpha(x)=C_\chi\frac{dV_{\xi}}{\sigma_P(\xi)^{n+2}dV}(x)\xi(x)+O(k^{-1}),
	\end{equation}
	as \(k\to +\infty\) where \(C_\chi=
	\int_0^{+\infty} t^{n+1}|\chi(t)|^2 dt\).
	Recall the definition of \(\omega_{\lambda(k)}\) in~\eqref{eq:DefinitionOmegabetaIntroduction} with  \(\lambda(k)=(\lambda_1,\ldots,\lambda_{N_k})\).
	\begin{proof}[Proof of Theorem \ref{thm:pullbackweightoneform}]
		Consider the function \(g\colon \C^{N_k}\to\R\), 
		\(g(z)=\sum_{j=1}^{N_k}\lambda_j|z_j|^2\). We obtain
		\begin{equation}
		\frac{k^{n+1}}{2\pi^{n+1}}F_k^*g=
		\sum_{j=1}^{N_k}\lambda_j|\chi(k^{-1}\lambda_j)|^2|f_j|^2.
		\end{equation}
		Hence, by Lemma~\ref{lem:evpowersexpansion},
		\begin{equation}\label{eq:PullbackGInOneFormOmega}
		F_k^*g=kC_\chi\frac{dV_{\xi}}{\sigma_P(\xi)^{n+1}dV}+O(1),
		\end{equation}
		in $\mathscr{C}^\infty$-topology as \(k\to+\infty\) and it follows that \(F_k^*g>0\) for all sufficiently large \(k>0\). 
		We have \(\omega_{\lambda(k)}=g^{-1}\alpha\) with \(\alpha\) as in \eqref{eq:OneFormAlpha}. Hence, \(F_k^*\omega_{\lambda(k)}\) is well-defined for all sufficiently large \(k>0\) and by using~\eqref{eq:pullbackOneForm} the statement follows. 
	\end{proof}
	Let \(S^{2N-1}\subset \C^N\) be the sphere of radius one centered in zero. Given a smooth function \(f\colon S^{2N-1}\to (-1,+\infty)\) we recall that the perturbed sphere \(S^{2N-1}(f)\) is defined by
	\begin{equation}
	S^{2N-1}(f):=\left\{z\in\C^N\setminus\{0\}\colon |z|=1+f\left(\frac{z}{|z|}\right)\right\}.
	\end{equation}
	We note that \(S^{2N-1}(f)\) carries naturally the structure of a codimension one CR manifold as a real hypersurface in \(\C^N\setminus\{0\}\), and recall that for \(\beta\in\R_+^N\) the induced CR structure \(T^{1,0}S^{2N-1}(f)\) of \(S^{2N-1}(f)\) is annihilated by the real $1$-form \(\alpha_{f,\beta}\) defined in~\eqref{eq:DefinitionAlphafbetaIntroduction}.
	\begin{proof}[Proof of Theorem \ref{thm:embeddingperturbedsphere}]
		We choose $dV=dV_\xi$ and $P=-i\mathcal{T}$ so that $\sigma_P(\xi)=\frac{dV}{dV_\xi}=1$. Consider the map $\mathcal{F}_k$ as in \eqref{eq:FkSectionAlmostEquivariant} and 
		\begin{equation}
		F_k\colon X\to \C^{N_k},~F_k(x)=\left(\int_0^{+\infty} t^{n}|\chi(t)|^{2}dt\right)^{-\frac{1}{2}}\mathcal{F}_k.
		\end{equation}
		With that rescaling it follows that here
		\(|F_k|^2=1+O(k^{-1})\) in 
		$\mathscr{C}^\infty$-topology and hence we find for 
		\(u_k(x):=|F_k(x)|-1\) that \(u_k=O(k^{-1})\) in 
		$\mathscr{C}^\infty$-topology as \(k\to+\infty\). 
		We will construct a sequence of functions 
		\(v_k\colon S^{2N_k-1}\to (-1,\infty)\) such that 
		\(F_k(X)\subset S^{2N_k-1}(v_k)\) for any sufficiently large 
		\(k>0\). By examining the \(k\)-dependence of \(v_k\) 
		we will deduce the statement by putting \(f=v_k\) for some large enough \(k\).
		Consider the map \(\widetilde{F}_k\colon X\to S^{2N_k-1}\), \(\widetilde{F}_k(x)
		=F_k(x)/|F_k(x)|\). From Lemma~\ref{lem:Fkinjective}
		it follows that \(\widetilde{F}_k\) is well-defined and injective 
		when \(k\) is large. By our choice here, the leading coefficient 
		of $|\chi|^2_k(T_P)(x,x)$ is a constant. It follows that \(d|F_k|=O(k^{-1})\) and hence with~\eqref{eq:PullbackEuclideanCN} we can check that \(\widetilde{F_k}\) is an immersion for all sufficiently large \(k\). We conclude that there exist \(k_0>0\) such that \(\widetilde{F}_k\) is an embedding of \(X\) as a real manifold into \(S^{2N_k-1}\) for all \(k\geq k_0\). Fix \(k\geq k_0\). Define \(\widetilde{v}_k\colon \widetilde{F}_k(X)\to \R\) by \(\widetilde{v}_k(\widetilde{F}_k(x))=|F_k(x)|-1=u_k(x)\). We extend \(\widetilde{v}_k\) to a function \(v_k\colon S^{2N_k-1}\to \R\) as follows. Put \(\widetilde{X}_k=\widetilde{F}_k(X)\) and let   \(N \widetilde{X}_k\) be the normal bundle of \(\widetilde{X}_k\subset S^{2N_{k}-1}\) with projection map \({\rm Pr}\). The standard metric on \(S^{2N_{k}-1}\subset \C^{N_k}\) induces a fibre metric \(h\) on \(N\widetilde{X}_k\). Let \(N^\delta\widetilde{X}_k\) be a \(\delta\)-tube in \(N\widetilde{X}_k\) with respect to \(h\).  Choosing \(\delta\) small enough we can identify \(N^\delta \widetilde{X}_k\) with a small open neighborhood \(U\) of \(\widetilde{X}_k\) in  \(S^{2N_{k}-1}\) via the exponential map.  Let \(\nu\in \cC_c^\infty(\R,[0,1])\) be a function supported in \([-\delta/2,\delta/2]\) with \(\nu\equiv 1\) in a neighborhood of \(0\). We then define \(v_k\) on \(S^{2N_k-1}\) by \(v_k(z)=0\) for \(z\notin U\) and \(v_k(z)=\nu(h(z,z))\widetilde{v}_k({\rm Pr}(z))\) for \(z\in U\). From this construction we obtain a family of functions \(\{v_k\}_{k\geq k_0}\) such that for each \(k\geq k_0\) we have that \(v_k\colon S^{2N_k-1}\to (-1,\infty)\) is smooth with \(v_k=\widetilde{v}_k\) on \(\widetilde{X}_k\) and \(\sup_{S^{2N_k-1}} |v_k|\leq \sup_{\widetilde{X}_k}|\widetilde{v}_k|\). Furthermore, for \(k\geq k_0\) and \(x\in X\), let \(V_1,\ldots,V_{2N_k-1}\) be an orthonormal basis of \(T_{\widetilde{F}_k(x)}S^{2N_k-1}\) such that \(V_1,\ldots,V_{2n+1}\in T_{\widetilde{F}_k(x)}\widetilde{X}_k\).
		We find \(|(dv_{k})_{\widetilde{F}_k(x)}|^2=|V_1(v_k)|^2+\ldots+|V_{2n+1}(v_k)|^2\). Using \(d|F_k|=O(k^{-1})\) again, by Lemma~\ref{lem:PullbackMetricIsMetricReebCase} below, we can find a constant \(C>0\) independent of \(k\) such that \(|(d\widetilde{F}_k)^{-1}V|^2\leq Ck^{-1}|V|^2\) for all \(V\in T\widetilde{X}_k\) and all sufficiently large \(k\geq k_0\). Since \(V_j(v_k)=((d\widetilde{F}_k)^{-1}V_j)(u_k)\), \(j=1,\ldots,2n+1\), we obtain \(|(dv_{k})|^2 =O(k^{-1})\) on \(\widetilde{X}_k\) as \(k\to +\infty\). Now consider the map \(H\colon \C^{N_k}\setminus\{0\}\to S^{2N_k-1}\), \(H(z)=z/|z|\). We find \(\|dH\|_{\operatorname{op}}\leq |z|^{-1}+1\).
		By Theorem~\ref{thm:pullbackweightoneform}, we have
		\begin{equation}
		F^*_k\alpha_{\lambda_{(k)},v_k}=\frac{F^*_k\omega_{\lambda_{(k)}}}{|F_k|^3}+\frac{\gamma_k}{g(F_k)}=\xi+\frac{\gamma_k}{g(F_k)}+O(k^{-1}),
		\end{equation}
		where \(\gamma_k:=F^*_ki(\partial (v_k\circ H)-\overline{\partial}(v_k\circ H))\) and \(g(z)=\sum_{j=1}^{N_k}\lambda_j|z|^2\). From~\eqref{eq:PullbackEuclideanCN} we find a constant \(C>0\) such that \(|dF_kV|\leq Ck|V|\) holds for all \(V\in TX\) and all sufficiently large \(k>0\).  Since
		\(\gamma_k=dv_k\circ dH\circ J\circ dF_k\) where \(J\) denotes the complex structure of \(\C^{N_k}\), we have \(\gamma_k = O(\sqrt{k})\). Recall that \(g(F_k)=ck +O(1)\) for some constant \(c>0\) independent of \(k\) by \eqref{eq:PullbackGInOneFormOmega}. As a consequence we obtain \(F^*_k\alpha_{\lambda_{(k)},v_k}-\xi=O(k^{-\frac{1}{2}})\). \\
		Now given \(\varepsilon>0\) and \(m\in\N\) we can choose \(k\geq k_0\) large enough such that \(\sup|v_k|\leq \varepsilon\), \(\sup_{\widetilde{X}_k}|dv_k|\leq \varepsilon/2\), \(\|v_k\circ \widetilde{F}_k\|_{\mathscr{C}^m(X)}=\|u_k\|_{\mathscr{C}^m(X)}\leq \varepsilon\) and \(|F^*_k\alpha_{\lambda_{(k)},v_k}-\xi|\leq \varepsilon\). In addition, using Theorem~\ref{thm:vectorfield} and Theorem~\ref{thm:pullbackweightoneform}, we can achieve that (vii) is also valid.  By construction we have \(F_k(X)\subset S^{2N_k-1}(v_k)\) and since \(|dv_k|\) is continuous we find an open neighborhood around \(\widetilde{X}_k\) in \(S^{2N_k-1}\) such that \(|dv_k|\leq \varepsilon\) is valid there. Claim (vi) follows from (v) by choosing \(\varepsilon<1\). 
	\end{proof}
	\begin{lemma}\label{lem:PullbackMetricIsMetricReebCase}
		Under the assumptions of Theorem~\ref{thm:ExpansionMain} with \(dV=dV_\xi\) and \(P=-i\mathcal{T}\) where \(\mathcal{T}\) is the Reeb vector field with respect to \(\xi\) let  \(F_k\) be as in~\eqref{eq:FkSectionAlmostEquivariant}. There exists a constant \(c>0\) such that \(|dF_k V|^2\geq ck|V|^2\) holds for all \(V\in TX\) and all sufficiently large \(k> 0\) where we choose a Hermitian metric on \(X\) as in Section~\ref{sec:CRmanifoldsMicroLocal}.
	\end{lemma}
	\begin{proof}
		Put \(H=\Re(T^{1,0}X)\). From \eqref{eq:PullbackEuclideanCN} 
		we find that there is a constant \(c_0>0\) such that \(|dF_k V|^2\geq c_0k|V|^2\) and \(|dF_k \mathcal{T}|^2\geq c_0k\) holds for all \(V\in H\) and all sufficiently large \(k> 0\). In order to prove the statement it is enough to show then that there exists a constant \(C>0\) with
		\begin{eqnarray}\label{eq:InnerProdBetweenTandHBounded}
		|\Re\langle dF_k \mathcal{T}, dF_k V\rangle|\leq C|V|
		\end{eqnarray}
		for all \(V\in H\) and all large enough \(k\). For each \(V\in H\) we find a unique \(Z\in T^{1,0}X\) with \(V=Z+\overline{Z}\). By the choice of Hermitian metric on \(X\) we further have \(|V|^2=|Z|^2+|\overline{Z}|^2\). Hence, using that \(\mathcal{T}\) is a real vector field, \eqref{eq:InnerProdBetweenTandHBounded} follows immediately from showing that for some constant \(C_1>0\) we have  
		\begin{eqnarray}\label{eq:InnerProdBetweenTandT10XBounded}
		|\langle dF_k \mathcal{T}, dF_k Z\rangle|\leq C_1|Z| \text{ for all } Z\in T^{1,0}X
		\end{eqnarray}
		and all large enough \(k\).
		We have
		\begin{eqnarray}\label{eq:RelationInnerProdAndOneForm}
		|\langle dF_k \mathcal{T}, dF_k Z\rangle|=|(F^*_kd\alpha)( \mathcal{T},Z)|=|(\iota_\mathcal{T} d(F^*_k\alpha))(Z)|
		\end{eqnarray}
		with \(\alpha\) as in \eqref{eq:OneFormAlpha}.  With \(\sigma_P(\xi)\equiv 1\) and \(dV=dV_\xi\)  we obtain from \eqref{eq:pullbackOneForm} that
		\begin{eqnarray}\label{eq:pullbackOneFormReebCase}
		F^*_k\alpha=kC_\chi\xi +O(1)
		\end{eqnarray} 
		holds in $\mathscr{C}^\infty$-topology as $k\to+\infty$ where \(C_\chi=\int_0^{+\infty} t^{n+1}|\chi|^2(t)dt>0\). 
		Applying the exterior differential on both sides 
		in~\eqref{eq:pullbackOneFormReebCase} and using 
		\(\iota_\mathcal{T} d\xi = 0\) we find from~\eqref{eq:RelationInnerProdAndOneForm} 
		that there is a constant \(C_1>0\) such that~\eqref{eq:InnerProdBetweenTandT10XBounded} 
		holds for all large enough \(k\) which in conclusion verifies the statement.
	\end{proof}
	From Theorem~\ref{thm:embeddingperturbedsphere} it follows that any CR manifold \((X,T^{1,0}X)\) satisfying the assumptions in Theorem~\ref{thm:ExpansionMain} can be CR embedded into a sphere with a CR structure arbitrarily close to the standard one. Let us introduce the following notation in order to state this fact more precisely.   
	Given a Riemannian manifold \(M\) and two smooth complex subbundles \(E,E'\) of \(\C TM\) of same rank we denote for any \(x\in M\) by \(\operatorname{dist}(E_x,E'_x)\) the Grassmannian distance between \(E_x\) and \(E'_x\) with respect to the metric on \(\C TM\) induced by the Riemannian metric on \(M\). For \(U\subset\subset M\) put \(\operatorname{dist}_U(E,E'):=\sup_{x\in U}\operatorname{dist}(E_x,E'_x)\). As a consequence of Theorem~\ref{thm:embeddingperturbedsphere} we obtain the following.
	\begin{theorem}\label{cor:SphericalEmbedding}
		In the situation of Theorem~\ref{thm:embeddingperturbedsphere} given \(\varepsilon>0\) there exist \(N\in\N\), a CR structure \(\mathcal{C}^{1,0}\) on \(S^{2N-1}\), an open set \(U\subset S^{2N-1}\), \(\beta\in(\R_+)^N\) 	and a CR embedding \(F\colon (X,T^{1,0}X)\to (S^{2N-1},\mathcal{C}^{1,0}) \), i.e \(F\) is a smooth embedding with \(dF T^{1,0}X=\C TF(X)\cap \mathcal{C}^{1,0}\), such that
		\begin{itemize}
			\item[(i)]  \(F(X)\subset U\) and \(\operatorname{dist}_U(T^{1,0}S^{2N-1},\mathcal{C}^{1,0})\leq \varepsilon\).
			\item[(ii)] \(\left|1-\frac{|F_*\mathcal{T}|}{|\mathcal{T}_\beta\circ F|}\right|\leq \varepsilon\) and \(\left|1-\frac{\langle F_*\mathcal{T},\mathcal{T}_\beta\circ F\rangle}{|F_*\mathcal{T}||\mathcal{T}_\beta\circ F	|}\right|\leq \varepsilon\).
			\item[(iii)] \(\operatorname{dist}_X(T^{1,0}X\oplus T^{0,1} X,\C \ker F^*\omega_\beta)\leq \varepsilon\). 
		\end{itemize}
	\end{theorem}
	For the proof of Theorem~\ref{cor:SphericalEmbedding} we need two results from linear algebra.
	\begin{lemma}\label{lem:GDist-AlmostOrthogonal}
		Let \(E\) be a finite dimensional complex Hermitian vector space and let \(V,W\subset E\) two subspaces of equal dimension. Assume that there exists \(0\leq \delta<1\) such that
		\begin{eqnarray}
		|\langle v,w\rangle|&\leq& \delta |v||w|\text{ for all }v\in V,w\in W^\perp,\\
		|\langle v,w\rangle|&\leq& \delta |v||w|\text{ for all }v\in V^\perp,w\in W.	
		\end{eqnarray}
		Then, for the Grassmannian distance between \(V\) and \(W\) we have \(\operatorname{dist}(V,W)\leq\sqrt{2\delta}\).
	\end{lemma}
	\begin{proof}
		Let \(P_V\) and \(P_W\) denote the orthogonal projections to \(V\) and \(W\) respectively. Given \(a\in E\) we find
		\begin{equation}
		\begin{split}
		&|(P_V-P_W)a|^2\\
		=&\langle P_V a,a\rangle+\langle P_W a,a\rangle- \langle P_V a, P_Wa\rangle- \langle P_W a,P_Va\rangle\\
		=&\langle P_V a,(1-P_W)a\rangle+\langle P_W a,(1-P_V)a\rangle.
		\end{split}
		\end{equation}
		It follows from the assumptions that \(|(P_V-P_W)a|^2\leq 2\delta|a|^2\) and hence
\begin{equation}
\operatorname{dist}(V,W):=\|P_V-P_W\|_{\text{op}}\leq \sqrt{2\delta}.
\end{equation}
	\end{proof}
	\begin{lemma}\label{lem:GDist-KernelMatrix}
		Let \(E=\C^n\) be the \(n\)-dimensional complex Euclidean space with the standard Hermitian metric and let \(A,D\in\operatorname{Mat}_{m\times n}(\C)\) two matrices with \(m<n\). Assume that there exists \(c,\varepsilon>0\) such that \(|A^* v|\geq c|v|\) for all \(v\in \C^m\) and \(\|D\|_{\text{op}}\leq \varepsilon \) with \(c>2\varepsilon\). Put \(V=\ker A\) and \(W=\ker (A+D)\). We have
		\begin{eqnarray}
		|\langle v,w\rangle|&\leq& \frac{\varepsilon}{c-2\varepsilon} |v||w|\text{ for all }v\in V,w\in W^\perp,\\
		|\langle v,w\rangle|&\leq& \frac{\varepsilon}{c-2\varepsilon} |v||w|\text{ for all }v\in V^\perp,w\in W.	
		\end{eqnarray}
	\end{lemma}
	\begin{proof}
		Put \(B=A+D\) and let \(v\in V\) and \(w\in W^\perp\) be arbitrary. We have \(Av=0\) and \(w\in \operatorname{ran}B^*\). Hence, there exists \(b\in \C^m\) with \(w=A^*b+D^*b\). We obtain
\begin{equation}
|\langle v,w\rangle|=|\langle v,A^*b\rangle|+|\langle v,D^*b\rangle|\leq|v||D^*b|\leq \varepsilon |v||b|.
\end{equation}
		Then the first inequality follows from 
		\(|w|\geq|A^* b|-|D^* b|\geq (c-\varepsilon)|b|\). The second inequality follows similarly using \(A=B+(-D)\) with \(|B^* v|\geq (c-\varepsilon)|v|\) for all \(v\in \C^m\).
	\end{proof}
	The first claim in Theorem~\ref{cor:SphericalEmbedding} will be a consequence of the following lemma.
	\begin{lemma}\label{lem:PerturbedCRStructureAtOnePoint}
		Let \(S^{2N+1}(f)\) be a perturbed sphere where \(f\colon S^{2N+1}\to (-1,+\infty)\) is a smooth function and put \(H\colon S^{2N+1}(f)\to S^{2N+1} \), \(H(w)=\frac{w}{|w|}\). We have that \(\mathcal{C}^{1,0}:=dH T^{1,0}S^{2N+1}(f)\) defines a (codimension one) CR structure on \(S^{2N+1}\). Furthermore, given \(p\in S^{2N+1}\) and \(0<\varepsilon<1\)  with  \(|df_p|(1+f(p))^{-1}< \varepsilon/2\) one has
		\begin{equation}
		\operatorname{dist}(C^{1,0}_p,T^{1,0}_pS^{2N+1})\leq \sqrt{\frac{\varepsilon}{1-\varepsilon}}
		\end{equation}
	\end{lemma}
	\begin{proof}
		Since \(H\) is a diffeomorphism between \(S^{2N+1}(f)\) and \(S^{2N+1}\) it follows that \(\mathcal{C}^{1,0}\) defines a CR structure on \(S^{2N+1}\).
		Given \(w\in S^{2N+1}(f)\) we have that \(v\in \C T_wS^{2N+1}(f)\) is in element of \(T^{1,0}S^{2N+1}(f)\) if and only if 
		\begin{eqnarray}
		d\overline{w}_j(v)&=&0,\,\,\text{ for all }j=0,\ldots, N\\
		\sum_{j=0}^N\overline{w}_jdw_j(v)&=&2|w|d(f\circ H)v.
		\end{eqnarray}
		We note that \(H^{-1}(z)=(1+f(z))z\). Hence, given \(z\in S^{2N+1}\) we have that \(v\in \C T_zS^{2N+1}\) is in element of \(\mathcal{C}^{1,0}_z\) if and only if 
		\begin{equation}\label{eq:perturbedStructure}
		\begin{split}\left(d\overline{z}_j+\frac{\overline{z}_jdf_z}{1+f(z)}\right)(v)&=0,\,\,\text{ for all }j=0,\ldots, N\\
\Big(\sum_{j=0}^N\overline{z}_jdz_j\Big)(v)-\frac{df_z(v)}{1+f(z)}&=0.
		\end{split}
		\end{equation}
Recall that \(v\in \C T_zS^{2N+1}\) is in element of 
\(T^{1,0}_zS^{2N+1}\) if and only if 
		\begin{equation}\label{eq:unperturbedStructure}
		\begin{split}
		d\overline{z}_j(v)&=0,\,\,\text{ for all }j=0,\ldots, N\\
		\Big(\sum_{j=0}^N\overline{z}_jdz_j\Big)(v)&=0.
		\end{split}
		\end{equation}
		After an unitary transformation we can assume \(p=(1,0,\ldots,0)\). Using \(z_j=x_j+y_j\), \(x_j,y_j\in\R\), \(j=0,\ldots,N\), we can isometrically identify \(\C TS^{2N+1}\) with \(\C^{2N+1}\) using the orthonormal basis 
		\begin{equation}
		\frac{\partial}{\partial y_0},\frac{\partial}{\partial x_1},\frac{\partial}{\partial y_1},\ldots,\frac{\partial}{\partial x_N},\frac{\partial}{\partial y_N}.
		\end{equation}
		From~\eqref{eq:perturbedStructure} and~\eqref{eq:unperturbedStructure} it follows that \(T^{1,0}_pS^{2N+1}=\ker A\), \(\mathcal{C}^{1,0}_p=\ker(A+D)\), for some matrices \(A,D\in\operatorname{Mat}_{(N+1)\times (2N+1)}\), such that \(|A^*v|\geq |v|\) and \(\|D\|_{\text{op}}\leq \varepsilon/2\). Then the claim follows from  Lemma~\ref{lem:GDist-AlmostOrthogonal} and Lemma~\ref{lem:GDist-KernelMatrix}.
	\end{proof}
	\begin{proof}[Proof of Theorem~\ref{cor:SphericalEmbedding}]
		In order to verify claim~(ii) and claim~(iii) consider the sequences of maps \(F_k\), \(\widetilde{F}_k\) and vectors \(\lambda(k)\in(\R_+)^{N_k}\) in the proof of Theorem~\ref{thm:embeddingperturbedsphere}, for sufficiently large \(k>1\). Since \(|F_k|=1+O(k^{-1})\) in \(\mathscr{C}^\infty\)-topology we observe that there is a constant \(C_1>0\) such that \(|dF_kv-d\widetilde{F}_kv|\leq C_1\) for all \(v\in \C TX\) and all sufficiently large \(k>1\). Furthermore,  for \(k\) large enough one has \(|\mathcal{T}_{\lambda(k)}\circ \widetilde{F}_k|\geq C_2k\), \(|\mathcal{T}_{\lambda(k)}\circ \widetilde{F}_k- \mathcal{T}_{\lambda(k)}\circ F_k |\leq C_3\) for some constants \(C_2,C_3>0\) independent of \(k>1\). Similarly, we obtain \(|\widetilde{F}_k^*\omega_{\lambda(k)}-\xi|=O(k^{-1})\) from Theorem~\ref{thm:pullbackweightoneform}. We note that \(T^{1,0}X\oplus T^{0,1}X=\C\ker \xi\). Then for any given \(\varepsilon>0\) putting \(F=\widetilde{F_k}\) and \(\beta=\lambda(k)\) for \(k\) large enough,  claim~(ii) can be deduced form Theorem~\ref{thm:vectorfield} and claim~(iii) follows from  Lemma~\ref{lem:GDist-AlmostOrthogonal} and Lemma~\ref{lem:GDist-KernelMatrix}.\\ In addition, we can choose \(k\) large enough such that \(\widehat{F}:=F_k\) is a CR embedding \(\widehat{F}\colon X\to S^{2N-1}(f)\subset \C^N\)  satisfying the properties in Theorem~\ref{thm:embeddingperturbedsphere} for \(\varepsilon'>0\) with
		\begin{equation}
		\varepsilon':= \frac{1}{3}\frac{\varepsilon^2}{(1+\varepsilon^2)}.
		\end{equation} 
		We observe
		\begin{equation}
		\frac{|df_p|}{1+f(p)}\leq \frac{\varepsilon'}{1-\varepsilon'}\leq\frac{1}{2}\frac{\varepsilon^2}{(1+\varepsilon^2)}
		\end{equation}
		for all \(p\in U\) with \(U\) as in Theorem~\ref{thm:embeddingperturbedsphere}.
Put \(H\colon S^{2N-1}(f)\to S^{2N-1}\), \(H(w)=\frac{w}{|w|}\). 
From Lemma~\ref{lem:PerturbedCRStructureAtOnePoint}
we conclude  that \(\mathcal{C}^{1,0}:=dH T^{1,0}S^{2N-1}(f)\) 
defines a (codimension one) CR structure on \(S^{2N-1}\) 
such that 
\begin{equation}
\operatorname{dist}(C^{1,0}_p,T^{1,0}_pS^{2N+1})\leq \varepsilon
\end{equation}		
for all \(p\in U\). This shows that claim~(i) holds when taking
\(F:=H\circ\widehat{F}=\widetilde{F}_k\).
\end{proof}


\end{document}